\documentclass[11pt]{article} 

\usepackage{amsfonts}
\usepackage{xcolor}
\usepackage{amsthm}

\usepackage{amsmath}
\usepackage{amssymb}
\usepackage[authoryear,sort&compress]{natbib}

\usepackage{multirow}
\usepackage{mathtools}
\usepackage{graphicx}
\usepackage{url}
\usepackage{algorithm2e}
\PassOptionsToPackage{algo2e,ruled}{algorithm2e}
\usepackage{amsthm}
\usepackage[margin=1in]{geometry}
\usepackage{hyperref}
\usepackage[capitalize]{cleveref}
\hypersetup{colorlinks=true,linkcolor=blue,citecolor=blue}

\usepackage{natbib}
\usepackage[mathscr]{eucal}

\newtheorem{theorem}{Theorem}
\newtheorem{lemma}[theorem]{Lemma}

  \newtheorem{corollary}[theorem]{Corollary}
  \newtheorem{definition}[theorem]{Definition}

\newenvironment{proc}[1][]
  {\renewcommand{\algorithmcfname}{Procedure}%
   \begin{algorithm}[#1]
   \long\def\@caption##1[##2]##3{%
     \par
     \begingroup\@parboxrestore
     \if@minipage\@setminipage\fi
     \normalsize \@makecaption{\AlCapSty{\AlCapFnt\algorithmcfname}}{\ignorespaces ##3}%
     \par\endgroup
   }}
  {\end{algorithm}
}

\newenvironment{game}[1][]
  {\renewcommand{\algorithmcfname}{Game}%
   \begin{algorithm}[#1]
   \long\def\@caption##1[##2]##3{%
     \par
     \begingroup\@parboxrestore
     \if@minipage\@setminipage\fi
     \normalsize \@makecaption{\AlCapSty{\AlCapFnt\algorithmcfname}}
     \par\endgroup
   }}
  {\end{algorithm}
}

\usepackage{caption}
\newcommand\saveAlgoCounter[1]{%
    \expandafter\xdef\csname saveAlgoCounter@#1\endcsname{\thealgocf}}
\newcommand{\repeatAlgoCaption}[2]{%
    \expandafter\let\expandafter\thealgocf\csname saveAlgoCounter@#1\endcsname
    \captionsetup{list=no}%
    \caption{#2}
    \addtocounter{algocf}{-1}
}

\newcommand{\comment}[1]{}
\setcitestyle{numbers,open={[},close={]}} 

\title{Gradient Descent is Pareto-Optimal in the Oracle Complexity and Memory Tradeoff for Feasibility Problems}

\author{
  \textbf{Mo\"ise Blanchard}\\
  Massachusetts Institute of Technology\\
  \small{\texttt{moiseb@mit.edu}}
}

\date{}

\usepackage{thmtools}
\usepackage{thm-restate}
\usepackage[mathscr]{eucal}
\usepackage{bbm}

\newcommand{\nonl}
{\renewcommand{\nl}{\let\nl\oldnl}}
\definecolor{codepurple}{rgb}{0.58,0,0.82}

\usepackage{latexsym,tikz,graphicx}
\usetikzlibrary{calc}
\usetikzlibrary{decorations.pathreplacing}
\usetikzlibrary{patterns}

\renewenvironment{proof}[1][]{\par\noindent{\bf Proof #1\ }}{\hfill$\blacksquare$\\[2mm]}

\begin{document}

\newcommand{\trw}{\text{\small TRW}}
\newcommand{\maxcut}{\text{\small MAXCUT}}
\newcommand{\maxcsp}{\text{\small MAXCSP}}
\newcommand{\suol}{\text{SUOL}}
\newcommand{\wuol}{\text{WUOL}}
\newcommand{\crf}{\text{CRF}}
\newcommand{\sual}{\text{SUAL}}
\newcommand{\suil}{\text{SUIL}}
\newcommand{\fs}{\text{FS}}
\newcommand{\fmv}{{\text{FMV}}}
\newcommand{\smv}{{\text{SMV}}}
\newcommand{\wsmv}{{\text{WSMV}}}
\newcommand{\trwp}{\text{\small TRW}^\prime}
\newcommand{\alg}{\text{ALG}}
\newcommand{\rhos}{\rho^\star}
\newcommand{\brhos}{\brho^\star}
\newcommand{\bzero}{{\mathbf 0}}
\newcommand{\bs}{{\mathbf s}}
\newcommand{\bw}{{\mathbf w}}
\newcommand{\bws}{\bw^\star}
\newcommand{\ws}{w^\star}
\newcommand{\Prt}{{\mathsf {Part}}}
\newcommand{\Fs}{F^\star}

\newcommand{\Hs}{{\mathsf H} }

\newcommand{\hL}{\hat{L}}
\newcommand{\hU}{\hat{U}}
\newcommand{\hu}{\hat{u}}

\newcommand{\bu}{{\mathbf u}}
\newcommand{\ubf}{{\mathbf u}}
\newcommand{\hbu}{\hat{\bu}}

\newcommand{\primal}{\textbf{Primal}}
\newcommand{\dual}{\textbf{Dual}}

\newcommand{\Ptree}{{\sf P}^{\text{tree}}}
\newcommand{\bv}{{\mathbf v}}

\newcommand{\bq}{\boldsymbol q}

\newcommand{\rvM}{\text{M}}

\newcommand{\Acal}{\mathcal{A}}
\newcommand{\Bcal}{\mathcal{B}}
\newcommand{\Ccal}{\mathcal{C}}
\newcommand{\Dcal}{\mathcal{D}}
\newcommand{\Ecal}{\mathcal{E}}
\newcommand{\Fcal}{\mathcal{F}}
\newcommand{\Gcal}{\mathcal{G}}
\newcommand{\Hcal}{\mathcal{H}}
\newcommand{\Ical}{\mathcal{I}}
\newcommand{\Lcal}{\mathcal{L}}
\newcommand{\Mcal}{\mathcal{M}}
\newcommand{\Ncal}{\mathcal{N}}
\newcommand{\Pcal}{\mathcal{P}}
\newcommand{\Scal}{\mathcal{S}}
\newcommand{\Tcal}{\mathcal{T}}
\newcommand{\Ucal}{\mathcal{U}}
\newcommand{\Vcal}{\mathcal{V}}
\newcommand{\Wcal}{\mathcal{W}}
\newcommand{\Xcal}{\mathcal{X}}
\newcommand{\Ycal}{\mathcal{Y}}
\newcommand{\Ocal}{\mathcal{O}}
\newcommand{\Qcal}{\mathcal{Q}}
\newcommand{\Rcal}{\mathcal{R}}

\newcommand{\brho}{\boldsymbol{\rho}}

\newcommand{\Cbb}{\mathbb{C}}
\newcommand{\Ebb}{\mathbb{E}}
\newcommand{\Nbb}{\mathbb{N}}
\newcommand{\Pbb}{\mathbb{P}}
\newcommand{\Qbb}{\mathbb{Q}}
\newcommand{\Rbb}{\mathbb{R}}
\newcommand{\Sbb}{\mathbb{S}}
\newcommand{\Vbb}{\mathbb{V}}
\newcommand{\Wbb}{\mathbb{W}}
\newcommand{\Xbb}{\mathbb{X}}
\newcommand{\Ybb}{\mathbb{Y}}
\newcommand{\Zbb}{\mathbb{Z}}

\newcommand{\Rbbp}{\Rbb_+}

\newcommand{\bX}{{\mathbf X}}
\newcommand{\bx}{{\boldsymbol x}}

\newcommand{\btheta}{\boldsymbol{\theta}}

\newcommand{\Pb}{\mathbb{P}}

\newcommand{\hPhi}{\widehat{\Phi}}

\newcommand{\Sigmah}{\widehat{\Sigma}}
\newcommand{\thetah}{\widehat{\theta}}

\newcommand{\indep}{\perp \!\!\! \perp}
\newcommand{\notindep}{\not\!\perp\!\!\!\perp}

\newcommand{\one}{\mathbbm{1}}
\newcommand{\1}{\mathbbm{1}}
\newcommand{\aprx}{\alpha}

\newcommand{\ST}{\Tcal(\Gcal)}
\newcommand{\x}{\mathsf{x}}
\newcommand{\y}{\mathsf{y}}
\newcommand{\Ybf}{\textbf{Y}}
\newcommand{\smiddle}[1]{\;\middle#1\;}

\definecolor{dark_red}{rgb}{0.2,0,0}
\newcommand{\detail}[1]{\textcolor{dark_red}{#1}}

\newcommand{\ds}[1]{{\color{red} #1}}
\newcommand{\rc}[1]{{\color{green} #1}}

\newcommand{\mb}[1]{\ensuremath{\boldsymbol{#1}}}

\newcommand{\metric}{\rho}
\newcommand{\proj}{\text{Proj}}

\newcommand{\paren}[1]{\left( #1 \right)}
\newcommand{\sqb}[1]{\left[ #1 \right]}
\newcommand{\set}[1]{\left\{ #1 \right\}}
\newcommand{\floor}[1]{\left\lfloor #1 \right\rfloor}
\newcommand{\ceil}[1]{\left\lceil #1 \right\rceil}
\newcommand{\abs}[1]{\left|#1\right|}

\maketitle

\begin{abstract}
    In this paper we provide oracle complexity lower bounds for finding a point in a given set using a memory-constrained algorithm that has access to a separation oracle. We assume that the set is contained within the unit $d$-dimensional ball and contains a ball of known radius $\epsilon>0$. This setup is commonly referred to as the \emph{feasibility problem}. We show that to solve feasibility problems with accuracy $\epsilon \geq e^{-d^{o(1)}}$, any deterministic algorithm either uses $d^{1+\delta}$ bits of memory or must make at least $1/(d^{0.01\delta }\epsilon^{2\frac{1-\delta}{1+1.01 \delta}-o(1)})$ oracle queries, for any $\delta\in[0,1]$. Additionally, we show that randomized algorithms either use $d^{1+\delta}$ memory or make at least $1/(d^{2\delta} \epsilon^{2(1-4\delta)-o(1)})$ queries for any $\delta\in[0,\frac{1}{4}]$. Because gradient descent only uses linear memory $\Ocal(d\ln 1/\epsilon)$ but makes $\Omega(1/\epsilon^2)$ queries, our results imply that it is Pareto-optimal in the oracle complexity/memory tradeoff. Further, our results show that the oracle complexity for deterministic algorithms is always polynomial in $1/\epsilon$ if the algorithm has less than quadratic memory in $d$. This reveals a sharp phase transition since with quadratic $\Ocal(d^2 \ln1/\epsilon)$ memory, cutting plane methods only require $\Ocal(d\ln 1/\epsilon)$ queries.
\end{abstract}

\section{Introduction}

We consider the \emph{feasibility problem} in which one has access to a separation oracle for a convex set contained in the unit ball $Q \subset B_d(\mb 0,1) := \{\mb x\in\Rbb^d, \|\mb x\|_2\leq 1\}$ and aims to find a point $\mb x\in Q$. For the feasibility problem with accuracy $\epsilon$ we assume that $Q$ contains a Euclidean ball of radius $\epsilon$. 
The feasibility problem is arguably one of the most fundamental problems in optimization and mathematical programming. For reference, the feasibility problem is tightly related to the standard non-smooth convex optimization problem \cite{nemirovskij1983problem,nesterov2003introductory} in which one aims to minimize Lipschitz smooth functions having access to a first-order oracle: the gradient information can be used as a separation oracle from the set of minimizers. Both setups have served as key building blocks for numerous other problems in optimization, computer science, and machine learning.

The efficiency of algorithms for feasibility problems is classically measured by either their time complexity or their oracle complexity, that is, the number of calls to the oracles needed to provide a solution. Both have been extensively studied in the literature: in the regime $\epsilon \leq 1/\sqrt d$ the textbook answer is that $\Theta(d\ln 1/\epsilon)$ oracle calls are necessary \cite{nemirovskij1983problem} and these are achieved by the broad class of cutting plane methods, which build upon the seminal ellipsoid method \cite{yudin76evaluation,khachiyan1980polynomial}. While the ellipsoid method has suboptimal $\Ocal(d^2\ln 1/\epsilon)$ oracle complexity, subsequent works including the inscribed ellipsoid \cite{tarasov1988method,nesterov1989self}, the celebrated volumetric center or Vaidya's method \cite{atkinson1995cutting,vaidya1996new,anstreicher1997vaidya}, or random walk based methods \cite{levin1965algorithm,bertsimas2004solving} achieve the optimal $\Ocal(d\ln 1/\epsilon)$ oracle complexity. In terms of runtime, in a remarkable tour-de-force \cite{lee2015faster,jiang2020improved} improved upon previous best-known $\Ocal(d^{1+\omega}\ln 1/\epsilon)$\footnote{$\omega<2.373$ denotes the exponent of matrix multiplication} complexity of Vaidya's method and showed that cutting planes can be implemented with $\Ocal(d^3 \ln 1/\epsilon)$ time complexity.

In practice, however, cutting planes are seldom used for large-dimensional applications and are rather viewed as impractical. While they achieve the optimal oracle complexity, these typically require storing all previous oracle responses, or at the very least, a summary matrix that uses $\Omega(d^2 \ln 1/\epsilon)$ bits of memory and needs to be updated at each iteration (amortized $\Omega(d^2)$ runtime per iteration). Instead, gradient-descent-based methods are often preferred for their practicality. These only keep in memory a few vectors, hence use only $\Ocal(d\ln 1/\epsilon)$ bits and $\Ocal(d\ln 1/\epsilon)$ runtime per iteration, but require $\Ocal(1/\epsilon^2)$ oracle queries which is largely suboptimal for $\epsilon\ll 1/\sqrt d$. These observations, as well as other practical implementation concerns, motivated the study of tradeoffs between the oracle complexity and other resources such as memory usage \cite{woodworth2019open,marsden2022efficient,blanchard2023quadratic,chen2023memory,blanchard2024memory} or communication \cite{lan_communication-efficient_2020, Reddi2016AIDEFA,shamir14,Smith2017,Mota2013,Zhang2012Communication,Wangni2018,wang17aMemo}.

\paragraph{Previous results on oracle complexity/memory tradeoffs for convex optimization and feasibility problems.}

Understanding the tradeoffs between oracle complexity and memory usage was first formally posed as a COLT 2019 open problem in \cite{woodworth2019open} in the convex optimization setup. Quite surprisingly, in the standard regime $\frac{1}{\sqrt d}\geq \epsilon\geq e^{-d^{o(1)}}$, gradient descent and cutting planes are still the only two known algorithms in the oracle complexity/memory tradeoff for both the feasibility problem and first-order Lipschitz convex optimization. Other methods have been proposed for other optimization settings to reduce memory usage including limited-memory Broyden– Fletcher– Goldfarb– Shanno (BFGS) methods \cite{Nocedal1980,liu_limited_1989}, conjugate gradient methods \cite{fletcher1964function,hager2006survey}, Newton methods variants \cite{pilanci2017newton,roosta2019sub}, or custom stepsize schedules \cite{das2023branch,altschuler2023acceleration}, however these do not improve in memory usage or oracle complexity upon gradient descent or cutting-planes in our problem setting. However, for super-exponential regimes \cite{blanchard2024memory} proposed recursive cutting-plane algorithms that provide some tradeoff between memory and oracle complexity, and in particular strictly improve over gradient descent whenever $\epsilon\leq e^{-\Omega(d\ln d)}$.

More advances were made on impossibility results.
\cite{marsden2022efficient} made the first breakthrough by showing that having both optimal oracle complexity $\Ocal(d\ln 1/\epsilon)$ and optimal $\Ocal(d\ln 1/\epsilon)$-bit memory is impossible: any algorithm either uses $d^{1+\delta}$ memory or makes $\tilde\Omega(d^{4/3(1-\delta)})$ oracle calls, for any $\delta\in [0,1/4]$. This result was then improved in \cite{blanchard2023quadratic,chen2023memory} to show that having optimal oracle complexity is impossible whenever one has less memory than cutting planes: for $\delta\in[0,1]$, any deterministic (resp. randomized) first-order convex optimization algorithm uses $d^{1+\delta}$ memory or makes $\tilde\Omega(d^{4/3-\delta/3})$ queries (resp. $\Omega(d^{7/6 - \delta/6 -o(1)})$ queries if $\epsilon \leq e^{-\ln^5 d}$). All these previous query lower bounds were proved for \emph{convex optimization}. For the more general \emph{feasibility problem}, the query lower bound can be further improved to $\tilde\Omega(d^{2-\delta})$ for deterministic algorithms \cite{blanchard2023quadratic}.

\paragraph{Our contribution.}

While the previous query lower bounds demonstrated the advantage of having larger memory, their separation in oracle complexity is very mild: the oracle complexity of memory-constrained algorithms is lower bounded to be $d$ times more than the optimal complexity $\Ocal(d \ln 1/\epsilon)$\footnote{In fact, the lower bound tradeoffs in \cite{marsden2022efficient,blanchard2023quadratic,chen2023memory} do not include any dependency in $\epsilon$. However, \cite{blanchard2024memory} noted that all previous lower bounds can be modified to add the $\ln 1/\epsilon$ factor}. This significantly contrasts with the oracle complexity $\Ocal(1/\epsilon^2)$ of gradient descent which is arbitrarily suboptimal for small accuracies $\epsilon$. In particular, the question of whether there exists linear-memory algorithms with only logarithmic dependency $\ln 1/\epsilon$ for their oracle complexity remained open. Given that gradient descent is most commonly used and arguably more practical than cutting-planes, understanding whether it can be improved in the oracle complexity/memory tradeoff is an important question to address. 

In this work, we provide some answers to the following questions for the \emph{feasibility problem}. (1) Can we improve the query complexity of gradient descent while keeping its optimal memory usage? This question was asked as one of the COLT 2019 open problems of \cite{woodworth2019open} for convex optimization. (2) What is the dependency in $\epsilon$ for the oracle complexity of memory-constrained algorithms? Our first result for deterministic algorithms is summarized below, where $o(1)$ refers to a function of $d$ vanishing as $d\to\infty$.

\begin{theorem}\label{thm:main_result_deterministic}
    Fix $\alpha\in(0,1]$. Let $d$ be a sufficiently large integer (depending on $\alpha$) and $\frac{1}{\sqrt d}\geq \epsilon \geq e^{-d^{o(1)}}$. Then, for any $\delta\in[0,1]$, any deterministic algorithm solving feasibility problems up to accuracy $\epsilon$ either uses $M=d^{1+\delta}$ bits of memory or makes at least $1/\paren{d^{\alpha\delta}\epsilon^{2\cdot \frac{1-\delta}{1+(1+\alpha)\delta} - o(1)}}$ queries.
\end{theorem}

For randomized algorithms, we have the following.

\begin{theorem}\label{thm:main_result_randomized}
    Let $d$ be a sufficiently large integer and $\frac{1}{\sqrt d}\geq \epsilon \geq e^{-d^{o(1)}}$. Then, for any $\delta\in[0,1/4]$, any randomized algorithm solving feasibility problems up to accuracy $\epsilon$ with probability at least $\frac{9}{10}$ either uses $M=d^{1+\delta}$ bits of memory or makes at least $1/\paren{d^{2\delta} \epsilon^{2(1-4\delta) - o(1)}}$
    queries.
\end{theorem}

\begin{figure}[tb]
    \centering
    \definecolor{blue-violet}{rgb}{0.54, 0.17, 0.89}
    \definecolor{cerisepink}{rgb}{0.93, 0.23, 0.51}
    \begin{tikzpicture}[scale=0.7]
    
    \draw[thick,draw=none,pattern=north west lines,pattern color=red!50] (-1,-1) -- (-1,7) -- (0,7) -- (0,0) -- (8,0) -- (8,-1);
    
    \draw[thick,draw=none,pattern=north west lines,pattern color=green] (0,7) -- (0,6)-- (7,6)--(7,0) -- (8,0) -- (8,7);

    \draw[thick,black] (-1,-1) rectangle (8,7);

    \draw[draw=none,fill=red!90!black]  (7/4,0) -- (7,0) .. controls (5,0.8) and (2,2) .. (0,6) ;

    \draw[draw=none,fill=red!30]  (7/4,0) -- (0,0) -- (0,6);

    \draw[thick, black]  (7,0) .. controls (5,0.2) and (2,0.8) .. (0,2) ;

    \draw[thick, black]  (0.23,0.75) .. controls (2,0.4) and (5,0.05)  .. (7,0)  ;
    
    \draw[thick, black]  (0,1) .. controls (0.5,0.4) and (1,0.2)  .. (7/4,0) ;

    \draw[thick, black]  (8,0) -- (0,0) -- (0,7);

    \draw[thick, black]  (7,0) -- (7,6) -- (0,6);

    \draw[thick,black] (0,-1.2) node[below]{$d\ln \frac{1}{\epsilon}$}  -- (0,-0.8);

    \draw[thick,black] (7/4,-1.2) node[below]{$d^{1.25}$}  -- (7/4,-0.8);
    \draw[thick,black] (7,-1.2) node[below]{$d^2\ln\frac{1}{\epsilon}$}  -- (7,-0.8);

    \draw[thick,black] (-1.2,0) node[left]{$d\ln\frac{1}{\epsilon}$}  -- (-0.8,0);
    \draw[thick,black] (-1.2,1) node[left]{$d^{1.33}$}  -- (-0.8,1);
    \draw[thick,black] (-1.2,2) node[left]{$d^{2}$}  -- (-0.8,2);
    \draw[thick,black] (-1.2,6) node[left]{$1/\epsilon^2$}  -- (-0.8,6);

    \draw (3.5,-2.5) node{Memory (bits)};
    \draw (-3.5,3) node[rotate=90]{Oracle complexity};

    \filldraw (0,6) circle (3pt);
    \draw [thick,black,-latex] (0,6) -- (-3,7) node[left,align=center,text width=1.5cm]{Gradient Descent} ;

    \draw [thick,black,-latex] (7,0) -- (9,-0.3)node [align=center,right,text width=1.5cm]{Cutting Planes};
    \filldraw (7,0) circle (3pt);

    \draw[thick,-latex] (0.5,3) -- (2.5,3.3);
    \draw[thick,-latex] (2,2) -- (3,3);
    \draw (4,3.5) node{This work};

    \draw (0.3,0.3) node{\textbf{1}};
    \draw (1.8,0.25) node{\textbf 2};
    \draw (0.8,1.1) node{\textbf 3};

    \end{tikzpicture}
    
    \caption{Tradeoffs between available memory and oracle complexity for the feasibility problem with accuracy $\epsilon$ in dimension $d$, in the regime $\frac{1}{\sqrt d}\geq \epsilon \geq e^{-d^{o(1)}}$ (adapted from \cite{woodworth2019open}). The dashed pink  (resp.~green) region corresponds to historical information-theoretic lower bounds (resp.~upper bounds). The region \textbf{1} and \textbf 2 correspond to the lower bound tradeoffs from \cite{marsden2022efficient} and \cite{chen2023memory} respectively for randomized algorithms.  The region \textbf{3} corresponds to the lower bound from \cite{blanchard2023quadratic} for deterministic algorithms. In this work, we show that the red (resp.~pink) solid region is not achievable for deterministic (resp.~randomized) algorithms.}
    \label{fig:visual_results_v2}
\end{figure}
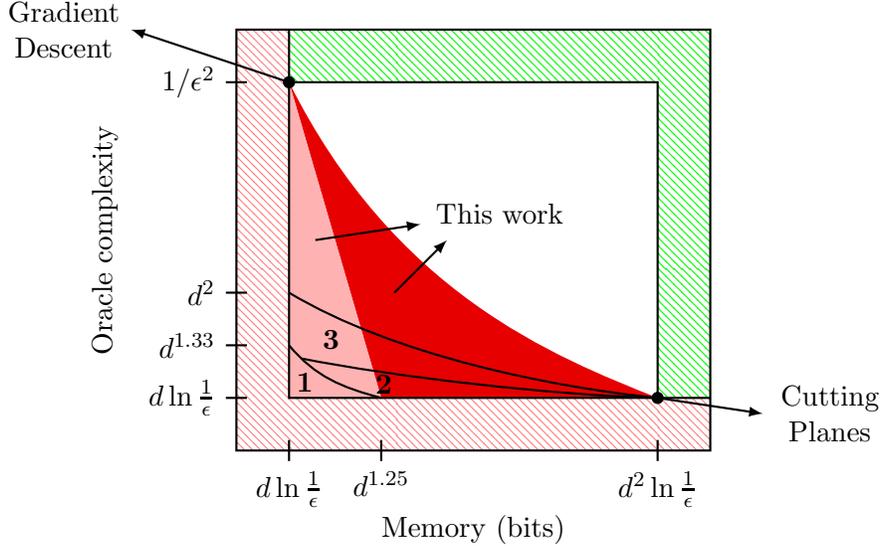

\cref{fig:visual_results_v2} provides a visualization of these tradeoffs and a comparison to previous results.
As a result, gradient descent is Pareto-optimal (up to a factor $d$) in the tradeoff between oracle complexity and memory usage for the feasibility problem. In other words, having optimal memory usage requires suffering the worst-case dependency $1/\epsilon^2$ for the oracle complexity. Further, our results also reveal a sharp phase transition for the oracle complexity of memory-constrained algorithms: deterministic algorithms that have less than quadratic memory in the dimension $d$ suffer a factor polynomial in $1/\epsilon$ in their oracle complexity. In contrast, we recall that using quadratic memory $\Ocal(d^2\ln 1/\epsilon)$, cutting plane methods achieve a logarithmic dependency in $1/\epsilon$ for the oracle complexity: $\Ocal(d\ln 1/\epsilon)$. For randomized algorithms, however, our result only implies that if an algorithm has less than $d^{5/4}$ memory it suffers a factor polynomial in $1/\epsilon$ in their oracle complexity.

We stress that since the feasibility problem generalizes the convex optimization setup, lower bound tradeoffs for feasibility problems do not imply lower bound tradeoffs for convex optimization, while the converse holds. We hope, however, that the techniques developed in this work can lead to future results for convex optimization.

\paragraph{On the tightness of the results.} \cref{thm:main_result_deterministic,thm:main_result_randomized} are simplified versions of the more precise lower bound results in \cref{thm:deterministic_alg_lower_bound,thm:query_lower_bound_randomized} respectively. These explicit the term $o(1)$ in the exponent of our lower bounds, which is of order $\frac{ \ln\ln d\; \lor\; \ln \frac{\ln(1/\epsilon)}{\ln d}}{\ln d}$, and also provide lower bound trade-offs in exponential regimes when $\epsilon = e^{-d^c}$, that degrade gracefully as $c>0$ grows. While our lower bounds are likely not perfectly tight, the results from \cite{blanchard2024memory} show that some dependency in $\ln \frac{\ln(1/\epsilon)}{\ln d} /\ln d$ in the query lower bound exponents is necessary, since they provide memory-constrained algorithms for exponential regimes. In particular, when $\epsilon \leq e^{-\Omega(d\ln d)}$ they achieve query complexity $(\Ocal(\ln 1/\epsilon))^d$ with the optimal $\Ocal(d\ln 1/\epsilon)$-bit memory. Hence in this regime, lower bounds cannot be polynomial in $1/\epsilon$ anymore. 

The bounds for the deterministic case depend on a parameter $\alpha\in(0,1]$. An inspection of the complete bound from \cref{thm:deterministic_alg_lower_bound} shows that in \cref{thm:main_result_deterministic} one can take $\alpha = \omega(\ln \ln d/\ln d)$. This term is due to our analysis of a probing game described in \cref{sec:technical_overview}. Improving the bounds for this game to delete this term may be possible; we refer to \cref{sec:technical_overview} for more details.

\paragraph{Additional works on learning with memory constraints.}

The impact of memory constraints on learning problems has received much attention in the past years. For parity learning, \cite{Raz2017} first proved that an exponential number of queries is necessary if the memory is sub-quadratic. Similar results were then obtained for other parity problems \cite{Kol2017,Moshkovitz2017,raz2018fast,Moshkovitz2018,Beame2018,Garg2018,garg2019time,garg2021memory}. Tradeoffs between memory usage and sample complexity were also studied for linear regression \cite{Steinhardt15,Sharan2019}, principal component analysis (PCA) \cite{Mitliagkas2013}, learning under the statistical query model \cite{steinhardt16}, minimizing regret in online learning  \cite{peng2023online,peng2023near,srinivas2022memory} and other general learning problems \cite{Brown2021, brown22a}.

\subsection{Outline of the paper}

We formalize the setup and give preliminary definitions in \cref{sec:formal_setup} then give an overview of the proof structure and main technical contributions in \cref{sec:technical_overview}. We mainly give details for the deterministic case (\cref{thm:main_result_deterministic}) but discuss additional proof components for the randomized case (\cref{thm:main_result_randomized}) in \cref{subsec:randomized_main_ideas}. We then prove our lower bound tradeoff for deterministic algorithms in \cref{sec:deterministic} and for randomized algorithms in \cref{sec:randomized}.

\section{Formal setup and notations}
\label{sec:formal_setup}

We aim to study the tradeoff between memory usage and query complexity for the feasibility problem over the unit ball $B_d(\mb 0,1)$. The goal is to find an element $\mb x\in Q$ for a convex set $Q\subset B_d(\mb 0,1)$ having access to a separation oracle $\Ocal_Q:B_d(\mb 0,1)\to \Rbb^d\cup\{\mathsf{Success}\}$ such that for any query $\mb x\in B_d(\mb 0,1)$, the oracle either returns $\mathsf{Success}$ if $\mb x\in Q$ or provides a separating hyperplane $\mb g\in\Rbb^d$ for $Q$, i.e., such that
\begin{equation*}
    \forall \mb x'\in Q,\quad \mb g^\top (\mb x'-\mb x)< 0.
\end{equation*}
The oracle is allowed to be randomized and potentially iteration-dependent (sometimes referred to as oblivious). For the feasibility problem with accuracy $\epsilon>0$, we assume that the set $Q$ contains a ball of radius $\epsilon$ and $\epsilon$ is known to the algorithm. We note that other works sometimes consider a stronger version of the feasibility problem in which the algorithm either finds a feasible point $\mb x\in Q$ or proves that $Q$ does not contain a ball of size $\epsilon$ \cite{lee2015faster,jiang2020improved}. Our impossibility results hence also apply to this setting as well.

We next formally define $M$-bit memory-constrained algorithms given a query space $\Scal$ and a response space $\Rcal$. Intuitively, these can only store $M$ bits of memory between each oracle call.

\begin{definition}[Memory-constrained algorithm]
    An $M$-bit memory-constrained deterministic algorithm is specified by query functions $\psi_{query,t}:\{0,1\}^M\to \Scal$ and update functions $\psi_{update,t}:\{0,1\}^M \times \Scal \times\Rcal\to \{0,1\}^M$ for $t\geq 1$. At the beginning of the feasibility problem, the algorithm starts with the memory state $\mathsf{Memory}_0=0_M$. At each iteration $t\geq 1$, it makes the query $\mb x_t = \psi_{query,t}(\mathsf{Memory_{t-1}})$, receives a response $r_t\in\Rcal$ from a separation oracle, then updates its memory state $\mathsf{Memory}_t = \psi_{update,t}(\mathsf{Memory}_{t-1},\mb x_t,r_t)$. 

    For $M$-bit memory-constrained randomized algorithms, the query and update functions can additionally use fresh randomness at every iteration.
\end{definition}

For the feasibility problem, we have in particular $\Scal=B_d(\mb 0,1)$ and $\Rcal = \Rbb^d\cup\{\mathsf{Success}\}$. Note that the defined notion of memory constraint is quite mild. Indeed, the algorithm is not memory-constrained for the computation of the query and update functions and can potentially use unlimited memory and randomness for these; the constraint only applies \emph{between} iterations. A fortiori, our lower bounds also apply to stronger notions of memory constraints.

\paragraph{Notations.} For any $\mb x\in\Rbb^d$ and $r\geq 0$, we denote by $B_d(\mb x,r) = \{\mb x'\in\Rbb^d: \|\mb x-\mb x'\|_2\leq 1\}$ the ball centered at $\mb x$ and of radius $r$. We denote the unit sphere by $S_{d-1}=\{\mb x:\|\mb x\|_2=1\}$ By default, the norm $\|\cdot\|$ refers to the Euclidean norm. For any integer $n\geq 1$, we use the shorthand $[n]=\{1,\ldots,n\}$. We denote $\mb e:=(1,0,\ldots,0)\in\Rbb^d$ and write $\mb I_d$ for the identity matrix in $\Rbb^d$. We use the notation $Span(\cdot)$ to denote the subspace spanned by considered vectors or subspaces. For a subspace $E$, $\proj_E$ denotes the orthogonal projection onto $E$. For a finite set $S$, we denote by $\Ucal(S)$ the uniform distribution over $S$. Last, for $d\geq k\geq 1$, to simplify the wording, a randomi uniform $k$-dimensional subspace of $\Rbb^d$ always refers to a $k$-dimensional subspace sampled according to the normalized Haar measure over the Grassmannian $\mb{Gr}_k(\Rbb^d)$.

\section{Technical overview of the proofs}
\label{sec:technical_overview}

In this section, we mainly give an overview of the proof of \cref{thm:main_result_deterministic}. The result for randomized algorithms \cref{thm:main_result_randomized} follows the same structure with a few differences discussed in \cref{subsec:randomized_main_ideas}.

\subsection{Challenges for having $\epsilon$-dependent query lower bounds} We first start by discussing the challenges in improving the lower bounds from previous works \cite{marsden2022efficient,blanchard2023quadratic,chen2023memory}. To make our discussion more concrete, we use as an example the construction of \cite{marsden2022efficient} who first introduced proof techniques to obtain query complexity/memory lower bounds, and explain the challenges to extend their construction and have stronger lower bounds. The constructions from the subsequent works \cite{blanchard2023quadratic,chen2023memory} use very similar classes of functions and hence present similar challenges. \cite{marsden2022efficient} defined the following hard class of functions to optimize
\begin{equation}\label{eq:previous_form}
    F_{\mb A, \mb v_1,\ldots,\mb v_N}(\mb x) := \max\left\{\|\mb A\mb x\|_\infty - \eta,\delta\left(\max_{i\leq N} \mb v_i^\top \mb x -i\gamma \right) \right\},
\end{equation}
where $\mb A\sim\Ucal(\{\pm 1\}^{d/2\times d})$ is sampled with i.i.d. binary entries, the vectors $\mb v_i \overset{i.i.d.}{\sim} \Ucal(\frac{1}{\sqrt d} \{\pm 1\}^d)$ are sampled i.i.d. from the rescaled hypercube, and $\gamma,\delta,\eta \ll 1$ are fixed parameters. We briefly give some intuition for why this class is hard to optimize for memory-constrained algorithms. 

The first term $\|\mb A\mb x\|_\infty -\eta$ acts as a barrier wall term: in order to observe the second term of the function which has been scaled by a small constant $\delta$, one needs to query vectors $\mb x$ that are approximately orthogonal to $\mb A$. On the other hand, the second term is a Nemirovski function \cite{nemirovski1994parallel,balkanski2018parallelization,bubeck2019complexity} that was used for lower bounds in parallel optimization and enforces the following behavior: with high probability an optimization algorithm observes the vectors $\mb v_1,\ldots,\mb v_N$ in this order and needs to query ``robustly-independent'' queries to observe these. A major step of the proof is then to show that finding many queries that are (1) approximately orthogonal to $\mb A$ and (2) robustly-independent, requires re-querying $\Omega(d)$ of the rows of $\mb A$. This is done by proving query lower bounds on an Orthogonal Vector Game that simulates this process. At the high level, one only receives subgradients that are rows of $\mb A$ \emph{one at a time}, while findings vectors that are approximately orthogonal to $\mb A$ requires information on \emph{all} of its rows.

A possible hope to improve the query lower bounds is to repeat this construction recursively at different depths, for example sampling other matrices $\mb A_p$ for $p\in[P]$ and adding a term $\delta^{(p)}(\|\mb A_p\mb x\|_\infty -\eta^{(p)})$ to Eq~\eqref{eq:previous_form}.
This gives rise to a few major challenges. First, for the recursive argument to work, one needs to ensure that the algorithm has to explore many robustly-independent queries at each depth. While this is ensured by the Nemirovski function (because it guides the queries in the direction of vectors $-\mb v_i$ sequentially) for the last layer, this is not true for any term of the form $\|\mb A\mb x\|_\infty - \eta$: a randomly generated low-dimensional subspace can easily observe all rows of $\mb A$ through subgradients. Intuitively, a random $l$ dimensional subspace cuts a hypercube $[-1,1]^d$ through all faces even if $l\ll d$ (potentially logarithmic in $d$). Conversely, the Nemirovski function does not enforce queries to be within a nullspace of a large incompressible random matrix. Further, once the vectors $\mb v_1,\ldots,\mb v_m$ have been observed, querying them again poses no difficulty (without having to query robustly-independent vectors): assuming we can find a point $\hat{\mb x}$ at the intersection of the affine maps for which $\mb v_i^\top \hat{\mb x} - i\gamma$ is roughly equal for all $i\in[N]$, it suffices to randomly query points in the close neighborhood of $\hat {\mb x}$ to observe all vectors $\mb v_1,\ldots,\mb v_N$ again.

In our proof, we still use a recursive argument to give lower bounds and build upon these proof techniques, however, the construction of the hard instances will need to be significantly modified.

\subsection{Construction of the hard class of feasibility problems}
\label{subsec:construction_procedure_intro}

The construction uses $P\geq 2$ layers and our goal is to show that solving the constructed feasibility problems requires an exponential number of queries in this depth $P$. We sample $P$ i.i.d. uniform $\tilde d$-dimensional linear subspaces $E_1,\ldots,E_P$ of $\Rbb^d$ where $\tilde d=\ceil{ d/(2P)}$. The feasible set is defined for some parameters $0<\delta_1<\ldots<\delta_P$ via
\begin{equation}\label{eq:def_feasible_set_deterministic_intro}
    Q_{E_1,\ldots,E_P} := B_d(\mb 0,1) \cap \set{ \mb x : \mb e^\top \mb x \leq -\frac{1}{2} } \cap \bigcap_{p\in[P]} \set{\mb x: \|\proj_{E_p}(\mb x)\| \leq \delta_p}.
\end{equation}
Hence, the goal of the algorithm is to query vectors approximately perpendicular to all subspaces $E_1,\ldots,E_P$. The first constraint $\mb e^\top \mb x\leq -\frac{1}{2}$ is included only to ensure that the algorithm does not query vectors with small norm. The parameters $\delta_p$ are chosen to be exponentially small in $P-p$ via $\delta_p = \delta_P/\mu^{P-p}$ for a factor $1\leq \mu=\Ocal(d^{3/2})$. The choice of the factor $\mu$ is crucial: it needs to be chosen as small as possible to maximize the number of layers $P$ in the construction while still emulating a feasibility problem with given accuracy $\epsilon$.

We now introduce the procedure to construct our separation oracles, which are designed to reveal information on subspaces $E_p$ with smallest possible index $p\in[P]$. For a sequence of linear subspaces $V_1,\ldots,V_r$ of $\Rbb^d$ that we refer to as \emph{probing} subspaces, we introduce the following function which roughly serves as a separation oracle for $Span(V_1,\ldots,V_r)^\perp$.
\begin{equation}\label{eq:perpendicular_span_separation_oracle}
    \mb g_{V_1,\ldots, V_r}(\mb x;\delta) := \begin{cases}
        \frac{\proj_V(\mb x)}{\| \proj_V(\mb x) \|} & \text{if }\|\proj_V(\mb x)\| > \delta,\\
        \mathsf{Success} &\text{otherwise},
    \end{cases}
    \quad \text{where } V=Span(V_i,i\in[r]).
\end{equation}
In practice, all probing subspaces will have the same fixed dimension $l=\Ocal(\ln d)$.
We introduce an adaptive feasibility procedure (see Procedure~\ref{proc:feasibility}) with the following structure. For each level $p\in[P]$ we keep some probing subspaces $V_1^{(p)},V_2^{(p)},\ldots$. These are sampled uniformly within $l$-dimensional subspaces of $E_p$ and are designed to ``probe'' for the query being approximately perpendicular to the complete space $E_p$. Because they have dimension $l=\Ocal(\ln d)\ll \tilde d$, they each reveal little information about $E_p$. On the other hand, they may not perfectly probe for the query being perpendicular to $E_p$ since they have much lower dimension than $E_p$. Fortunately, we show these fail only with small probability for adequate choice of parameters (see \cref{lemma:most_periods_proper}).

Before describing the procedure, we define \emph{exploratory} queries at some depth $p\in[P]$. These are denoted $\mb y^{(p)}_i$ for $i\in[n_p]$ where $n_p$ is the number of current depth-$p$ exploratory queries.

\begin{definition}[Exploratory queries]
\label{def:explo_query_deterministic}
    Given previous exploratory queries $\mb y_i^{(p)}$ and probing subspaces $V_i^{(p)}$ for $p\in[P], i\in[n_p]$, we say that $\mb x\in B_d(\mb 0,1)$ is a depth-$p$ \emph{exploratory} query if:
    \begin{enumerate}
        \item $\mb e^\top \mb x \leq -\frac{1}{2}$,
        \item the query passed all probes from levels $q\leq p$, that is $\mb g_{V_1^{(q)},\ldots, V_{n_q}^{(q)}}(\mb x;\delta_q) = \mathsf{Success}$. Equivalently,
        \begin{equation}\label{eq:pass_depths}
            \|\proj_{Span(V_i^{(q)},i\in[n_q])}(\mb x)\| \leq \delta_q,\quad q\in[p]
        \end{equation}
        \item and it is robustly-independent from all previous depth-$p$ exploratory queries,
        \begin{equation}\label{eq:robustly-independent_query}
            \|\proj_{Span(\mb y_r^{(p)},r\leq n_p)^\perp}(\mb x)\| \geq \delta_p.
        \end{equation}
    \end{enumerate}
\end{definition}

The probing subspaces are updated throughout the procedure. Whenever a new depth-$p$ exploratory query is made, we sample a new $l$-dimensional linear subspace $V^{(p)}_{n_p+1}$ uniformly in $E_p$ unless there were already $n_p=k$ such subspaces for some fixed parameter $k\in[\tilde d]$.  In that case, we reset all exploratory queries at depth depths $q\leq p$, pose $n_q=1$, and sample new subspaces $V_1^{(q)}$ for $q\in[p]$. For each level $p$, denoting by $\mb V^{(p)}:=(V_1^{(p)},\ldots,V_{n_p}^{(p)})$ the list of current depth-$p$ probing subspaces, we define the oracle as follows
\begin{equation}
\label{eq:def_separation_oracle}
    \Ocal_{\mb V^{(1)},\ldots, \mb V^{(P)}}(\mb x):= \begin{cases}
        \mb e &\text{if } \mb e^\top \mb x > -\frac{1}{2}\\
        \mb g_{\mb V^{(p)}}\paren{\mb x;\delta_p } &\text{if } \mb g_{\mb V^{(p)}}\paren{\mb x;\delta_p } \neq \mathsf{Success} \\
        &\text{and } p=\min\set{q\in[P], \mb g_{\mb V^{(q)}}\paren{\mb x;\delta_q } \neq \mathsf{Success}}\\
        \mathsf{Success} &\text{otherwise}.
    \end{cases}
\end{equation}
The final procedure first tries to use the above oracle, then turns to some arbitrary separation oracle for $Q_{E_1,\ldots,E_P}$ when the previous oracle returns $\mathsf{Success}$. When the procedure returns a vector of the form $\mb g_{\mb V^{(p)}}(\mb x;\delta_p)$ for some $p\in[P]$, we say that $\mb x$ was a depth-$p$ query. We can then check that with high probability, this procedure forms a valid adaptive feasibility problem for accuracy $\epsilon=\delta_1/2$ (\cref{lemma:check_procedure_is_feasibility_pb}).

\subsection{Structure of the proof for deterministic algorithms.}

For all $p\in[P]$, we refer to a depth-$p$ period as the interval of time between two consecutive times when the depth-$p$ probing subspaces were reset. We first introduce a depth-$p$ game (Game~\ref{game:feasibility_game}) that emulates the run of the procedure for a given depth-$p$ period, the main difference being that the goal of the algorithm is not to query a feasible point in $Q_{E_1,\ldots,E_P}$ anymore, but to make $k$ depth-$p$ exploratory queries. This makes the problem more symmetric in the number of layers $p\in[P]$ which will help for a recursive query lower bound argument.

\paragraph{Properties of probing subspaces.} We show that the depth-$p$ probing subspaces are a good proxy for testing orthogonality to $E_p$. This is formalized with the notion of \emph{proper} periods (see \cref{def:proper_period}) during which if the algorithm performed a depth-$p'$ query $\mb x_t$ with $p'>p$---hence passed the probes at level $p$---it satisfies $\|\proj_{E_p}(\mb x_t)\| \leq \eta_p$ for some parameter $\eta_p\geq \delta_p$ as small as possible. We show in \cref{lemma:most_periods_proper} that periods are indeed proper with high probability if we take $\eta_p =\Omega(\sqrt{\tilde d k^\alpha}  \delta_p)$ for any fixed $\alpha\in(0,1]$. The proof of this result is one of the main technicalities of the paper and uses a reduction to the following Probing Game.

\begin{game}[h!]

\caption{Probing game}\label{game:probing_game}
\saveAlgoCounter{alg:probing_game}
\setcounter{AlgoLine}{0}

\SetAlgoLined
\LinesNumbered

\everypar={\nl}

\hrule height\algoheightrule\kern3pt\relax
\KwIn{dimension $d$, response dimension $l$, number of exploratory queries $k$, objective $\rho>0$}

\textit{Oracle:} Sample independent uniform $l$-dimensional subspaces $V_1,\ldots,V_k$ in $\Rbb^d$.

\For{$i\in[k]$}{
    \textit{Player:} Based on responses $V_j,j<i$, submit query $\mb y_i\in\Rbb^d$

    \textit{Oracle:} \Return $V_i$ to the player
}

Player wins if for any $i\in[k]$ there exists a vector $\mb z\in Span(\mb y_j,j\in[i])$ such that
\begin{equation*}
    \|\proj_{Span(V_j,j\in[i])}(\mb z) \| \leq \rho \quad \text{and} \quad \|\mb z\|=1.
\end{equation*}

\hrule height\algoheightrule\kern3pt\relax
\end{game}

Our goal is to show that no player can win at this game with reasonable probability. In this game, the player needs to output $k$ robustly-independent queries that are perpendicular to the probing spaces. Because these only span at most $k$ dimensions, if the probing subspaces had dimension $l=\Omega(k)$ we could easily prove the desired result (from high-dimensional concentration results akin to the Johnson–Lindenstrauss lemma). This is prohibitive for the tightness of our results, however. In fact, for our result for randomized algorithms, we are constrained to this suboptimal choice of parameters, which is one of the reasons why the lower bound tradeoff from \cref{thm:main_result_randomized} does not extend to full quadratic memory $d^2$. Instead, we show that $l=\Ocal(\ln d)$ are sufficient to ensure that Game~\ref{game:probing_game} is impossible with high probability (\cref{thm:no_small_vectors}), which is used for the deterministic case. This requires the following result for adaptive random matrices.

\begin{theorem}\label{thm:random_triangular_matrix}
    Let $C\geq 2$ be an integer and $m=Cn$. Let $\mb M\in\Rbb^{n\times m}$ be a random matrix such that all coordinates $M_{i,j}$ in the upper-triangle $j>(i-1)C$ are together i.i.d. Gaussians $\Ncal(0,1)$. Further, suppose that for any $i\in[n]$, the Gaussian components $M_{u,v}$ for $v>(u-1)C$ with $v>C(i-1)$ are together independent from the sub-matrix $(M_{u,v})_{u\in[i],v\in[C(i-1)]}$. Then for any $\alpha\in(0,1]$, if $C\geq C_\alpha\ln n$, we have
    \begin{equation*}
        \Pbb \paren{\sigma_1(\mb M) < \frac{1}{6} \sqrt{\frac{C}{n^{\alpha}}} } \leq 3e^{-C/16}.
    \end{equation*}
    Here $C_\alpha = (C_2/\alpha)^{\ln 2/\alpha}$ for some universal constant $C_2\geq 8$.
\end{theorem}

As a remark, even for the case when the matrix is exactly upper-triangular, that is, $M_{i,j}=0$ for all $j\leq (i-1)C$, we are not aware of probabilistic lower bounds on the smallest singular value. Note that this matrix is upper-triangular on a rectangle instead of a square, which is non-standard. We state the corresponding result for upper-triangular rectangular matrices in \cref{cor:M0_singular_value}, which may be of independent interest. For square matrices, previous works showed that the smallest singular value is exponentially small in the dimension $n$ \cite{viswanath1998condition}, which is prohibitive for our purposes. Alternatively, \cite{rudelson2016singular} give bounds when the i.i.d. Gaussian part of the matrix is broadly connected, a notion similar to graph expansion properties. However, these assumptions are not satisfied if that Gaussian part corresponds to the upper triangle, for which some nodes are sparsely connected. Further, here we potentially allow the coordinates in the lower-triangle to be adaptive in a subset of coordinates in the upper-triangle. We are not aware of prior works that give singular values lower bounds for adaptive components on non-i.i.d. parts, as opposed to simply deterministic components as considered in \cite{rudelson2016singular}. 

We briefly discuss the tightness of \cref{thm:random_triangular_matrix}. We suspect that the extra factor $n^\alpha$ in the denominator for the bound of $\sigma_1(\mb M)$ may be superfluous. One can easily check that the best bound one could hope for here is $\sigma_1(\mb M) = \Omega(\sqrt C)$. This extra factor $n^\alpha$ is the reason for the term $\alpha$ in the query lower bound from \cref{thm:main_result_deterministic}, hence shaving off this factor would directly improve our lower bounds tradeoffs. The success probability (exponentially small in $C$, but not in $n$) is however tight, and experimentally it seems that having $C=\Omega(\ln n)$ is also necessary.

\paragraph{Query lower bounds for an Orthogonal Subspace Game.}

In the construction of the adaptive feasibility procedure (Procedure~\ref{proc:feasibility}) we reset all exploratory queries and probing subspaces for depths $p'\leq p$ whenever $k$ depth-$p$ queries are performed. This gives a nested structure to periods: a depth-$p$ period is a union of depth-$(p-1)$ periods. We then show that we can reduce the run of a depth-$p$ period to the following Orthogonal Subspace Game~\ref{game:orthogonal_subspace_game} for appropriate choice of parameters, provided that all contained depth-$(p-1)$ periods are proper. This game is heavily inspired by the Orthogonal Vector Game introduced in \cite{marsden2022efficient}.

\begin{game}[h!]

\caption{Orthogonal Subspace Game}\label{game:orthogonal_subspace_game}
\saveAlgoCounter{alg:orthogonal_subspace_game}
\setcounter{AlgoLine}{0}

\SetAlgoLined
\LinesNumbered

\everypar={\nl}

\hrule height\algoheightrule\kern3pt\relax

\KwIn{dimensions $d$, $\tilde d$; memory $M$; number of robustly-independent vectors $k$; number of queries $m$; parameters $\beta$, $\gamma$}

\textit{Oracle:} Sample a uniform $\tilde d$-dimensional subspace $E$ in $\Rbb^d$ and $\mb v_1,\ldots,\mb v_m\overset{i.i.d.}{\sim} \Ucal(S_d\cap E)$ 

\textit{Player:} Observe $E$ and $\mb v_1,\ldots,\mb v_m$, and store an $M$-bit message $\mathsf{Message}$ about these

\textit{Oracle:} Send samples $\mb v_1,\ldots,\mb v_m$ to player

\textit{Player:} Based on $\mathsf{Message}$ and $\mb v_1,\ldots,\mb v_m$ only, return unit norm vectors $\mb y_1,\ldots, \mb y_k$

The player wins if for all $i\in[k]$
\begin{enumerate}
    \item $\| \proj_E(\mb y_i)\| \leq \beta$
    \item $\|\proj_{Span(\mb y_1,\ldots,\mb y_{i-1})^\perp}(\mb y_i)\| \geq \gamma$.
\end{enumerate}
\hrule height\algoheightrule\kern3pt\relax
\end{game}

We prove a query lower bound $\Omega(d)$ for this game if the player does not have sufficient memory and needs to find too many robustly-independent vectors.

\begin{theorem}
\label{thm:memory_query_lower_bound}
    Let $d\geq 8$, $C\geq 1$, and $0<\beta,\gamma\leq 1$ such that $\gamma/\beta \geq 12 \sqrt{kd/\tilde d}$. Suppose that $\frac{\tilde d}{4}\geq k \geq 50C\frac{M+2}{\tilde d}\ln\frac{\sqrt d}{\gamma}$. If the player wins the Orthogonal Subspace Game~\ref{game:orthogonal_subspace_game} with probability at least $1/C$, then $m\geq \frac{\tilde d}{2}$.
\end{theorem}

For our lower bounds to reach the query complexity $1/\epsilon^2$ of gradient descent, we need the parameter choices of \cref{thm:memory_query_lower_bound} to be tight. In particular, the robustly-independent parameter $\gamma$ is allowed to be roughly of the same order of magnitude as the orthogonality parameter $\delta$. Previous works required at least a factor $\sqrt d$ between these two parameters. For intuition, having the constraint $\gamma/\beta \geq d^c$ for some constant $c\geq 0$ would at best yield a lower bound query complexity of $1/\epsilon^{2/(1+2c)}$ even for linear memory $d$. 

\paragraph{Reduction from the feasibility procedure to the Orthogonal Subspace Game.} We briefly explain the reduction from running a depth-$p$ period of the feasibility game to the Orthogonal Subspace Game~\ref{game:orthogonal_subspace_game}. This will also clarify why query lower bounds heavily rely on the constraint for $\gamma/\beta$ from \cref{thm:memory_query_lower_bound}. In this context, finding exploratory queries translates into finding robustly-independent vectors (as in Eq~\eqref{eq:robustly-independent_query}). Further, provided that the corresponding depth-$(p-1)$ periods are proper, exploratory queries also need to be orthogonal to the subspace $E_{p-1}$ up to the parameter $\eta_{p-1}$. In turn, we show that this gives a strategy for the Orthogonal Subspace Game~\ref{game:orthogonal_subspace_game} for parameters $(\beta,\gamma) = (2\eta_{p-1},\delta_p)$. Hence, a lower bound on $\gamma/\beta$ directly induces a lower bound on the factor parameter $\mu=\delta_{p-1}/\delta_p$ on which the parameter $\delta_1$ depends exponentially. 

Also, because the procedure is adaptive, responses for depth-$p'$ queries with $p'>p$ may reveal information about the subspace $E_p$---this also needs to be taken into account by the reduction. Indeed, although the subspaces $E_p$ and $E_{p'}$ are independent, depth-$p'$ responses are constructed from the depth-$p'$ probing subspaces which are added adaptively on previous depth-$p'$ exploratory queries (see \cref{subsec:construction_procedure_intro}), for which the algorithm can use any information it previously had on $E_p$. However, we show that this information leakage is mild and can be absorbed into a larger memory for the Orthogonal Subspace Game~\ref{game:orthogonal_subspace_game}. 

\paragraph{Recursive query lower bounds.}

The query lower bound for \cref{thm:memory_query_lower_bound} then implies that the algorithm must complete many depth-$(p-1)$ periods within a depth-$p$ period (\cref{lemma:many_depth_p-1}). By simulating one of these depth-$(p-1)$ periods, we show that from a strategy for depth-$p$ periods that performs $T_p$ queries we can construct a strategy for depth-$(p-1)$ periods that uses at most $ T_{p-1}\approx \frac{lk}{\tilde d}T_p$ (\cref{lemma:recursion}). Applying this result recursively reduces the number of allowed queries exponentially with the depth $P$. Selecting the parameters $k$ and $P$ appropriately gives the query lower bound from \cref{thm:main_result_deterministic} for memory-constrained algorithms.

\subsection{Other proof components for randomized algorithms}
\label{subsec:randomized_main_ideas}

For randomized algorithms we cannot construct an adaptive feasibility procedure. In particular, we cannot add a new probing subspace whenever a new exploratory query was performed. Instead, we construct an oblivious iteration-dependent oracle and resample probing subspaces regularly (see formal definition in Eq~\eqref{eq:def_final_oracle_randomized}) hoping that for most periods the algorithm did not have time to perform $k$ exploratory queries. 

This has a few main implications. 
First, we cannot use the convenient dimension $l=\Ocal(\ln d)$ for probing subspaces anymore because this heavily relied on the structure of the Probing Game~\ref{game:probing_game}, for which we sample a new probing subspace just after every new direction is explored. Instead, the probing subspaces $V_1^{(p)},\ldots,V_k^{(p)}$ need to be present at all times in the oracle and kept constant within a depth-$p$ period. 
Second, we still need to ensure that the algorithm discovers the probing subspaces in the exact order $V_1^{(p)},\ldots,V_k^{(p)}$. Hence we use for each one of these a slightly different orthogonality tolerance parameter. Precisely, instead of using the oracle $\mb g_{V_1,\ldots,V_k}$ as in the deterministic case, we define
\begin{equation*}
    \Ical_{V_1,\ldots,V_r}(\mb x;\mb \delta) = \set{i\in[k]: \|\proj_{V_i}(\mb x)\| > \delta_i}.
\end{equation*}
for some parameter $\mb\delta=(\delta_1>\ldots>\delta_k)$ and use the following oracle
\begin{equation*}
    \tilde{\mb g}_{V_1,\ldots,V_r}(\mb x;\mb \delta) := \begin{cases}
        \frac{\proj_{V_i}(\mb x)}{\|\proj_{V_i}(\mb x)\|} &\text{if } \Ical_{V_1,\ldots,V_k}(\mb x;\mb \delta)\neq \emptyset \text{ and } i=\min \Ical_{V_1,\ldots,V_k}(\mb x;\mb \delta),\\
        \mathsf{Success} &\text{otherwise}.
    \end{cases}
\end{equation*}
The main idea is that while the algorithm does not query vectors slightly orthogonal to $V_1$, it cannot see $V_2$, and so on. Unfortunately, concentration bounds only ensure that this is true if the dimension $l$ of the probing subspaces is sufficiently large. In practice, we need $l=\Omega(k\ln d)$. Third and last, as it turns out, we cannot ensure that for all depth-$(p-1)$ periods within a given depth-$p$ period, the probing subspaces were good proxies for being orthogonal to $E_{p-1}$. In fact, this will most of the time be false. However, given the index of the depth-$(p-1)$ that was improper, we show that the algorithm does perform $k$ exploratory queries during this period. This still gives the desired recursive argument at the expense of giving the algorithm some extra power to select which period to play. As a consequence, the probing subspaces in the played period are not distributed uniformly among $l$-dimensional subspaces of $E_p$ anymore. To resolve this issue, we adapt the Orthogonal Subspace Game~\ref{game:orthogonal_subspace_game} to include this additional degree of liberty of the player (see Game~\ref{game:adapted_orthogonal_subspace_game}) and show that the query lower bounds still hold (\cref{thm:adapted_memory_query_lower_bound}).

\section{Query complexity/memory tradeoffs for deterministic algorithms}
\label{sec:deterministic}

\subsection{Definition of the feasibility procedure}

We give here the detailed construction for the hard class of feasibility problems. We mainly specify the parameters that were already introduced in \cref{subsec:construction_procedure_intro}. Fix the parameter $P\geq 2$ for the depth of the nested construction. We assume throughout the paper that $d\geq 40P$. We sample $P$ uniformly random $\tilde d$-dimensional linear subspaces $E_p$ of $\Rbb^d$ where $d:=\floor{d/(2P)}$. Note that with our choice of parameters ($d\geq 40P$) we have in particular $\tilde d\geq \frac{d}{3P}$. We next introduce a parameter $k\in[\tilde d]$ and fix $\alpha\in(0,1)$ a constant that will define the sharpness of the Pareto-frontier (the smaller $\alpha$, the stronger the lower bounds, but these would apply for larger dimensions). Next, we define
\begin{equation}\label{eq:def_l}
    l:=\ceil{ 16\ln (32d^2P) \lor C_\alpha\ln k},
\end{equation}
for some constant $C_\alpha = (\Ocal(1/\alpha))^{\ln 2/\alpha}$ that only depends on $\alpha$ and that will formally be introduced in \cref{thm:random_triangular_matrix}. For the last layer $P$, we use an extra-parameter $l_P\in[\tilde d]$ such that $l_P\geq l$. For convenience, we then define $l_p = l$ for all $p\in[P-1]$.
It remains to define some parameters $\eta_p$ for $p\in[P]$ that will quantify the precision needed for the algorithm at depth $p$. For the tightness of our results, we need to define the last layer $P$ with different parameters, as follows:
\begin{equation*}
    \eta_P:= \frac{1}{10}\sqrt{\frac{\tilde d}{d}} \quad \text{and}\quad \mu_P:= 600\sqrt{\frac{dk^{1+\alpha}}{l_P}}.
\end{equation*}
For $p\in[P-1]$, we let
\begin{equation}\label{eq:def_eta}
    \eta_p := \frac{\eta_P/\mu_P}{\mu^{P-p-1}}  \quad \text{where} \quad  \mu := 600 \sqrt {\frac{d k^{1+\alpha} }{l}}.
\end{equation}
The orthogonality parameters $\delta_1,\ldots,\delta_P$ are then defined as
\begin{equation}\label{eq:def_delta}
    \delta_p := \frac{\eta_p}{36}\sqrt{\frac{l_p}{\tilde dk^{\alpha}}},\quad p\in[P-1].
\end{equation}
Note that at this point, the three main remaining parameters are the depth $P$, the dimension $l_P$ at the last layer, and $k$ which will serve as the maximum number of exploratory queries within a period. Also, note that if $l_P=l$, then the last layer for $p=P$ is constructed identically as the other ones.

Given these parameters, the feasible set $Q_{E_1,\ldots,E_P}$ is given as in Eq~\eqref{eq:def_feasible_set_deterministic_intro}. For any probing subspaces $V_1,\ldots,V_r$, we recall the form of the function $\mb g_{V_1,\ldots,V_r}(\cdot;\delta)$ from Eq~\eqref{eq:perpendicular_span_separation_oracle}. Next, we recall the definition of depth-$p$ exploratory queries for $p\in[P]$ from \cref{def:explo_query_deterministic}. Intuitively, these are queries that pass probes for all depths $q<p$ and are robustly-independent from previous depth-$p$ exploratory queries. We recall the notation $n_p$ for the number of depth-$p$ exploratory queries. These are denoted by $\mb y_1^{(p)},\ldots,\mb y_{n_p}^{(p)}$.

\begin{proc}[ht!]
\caption{Adaptive separation oracle for optimization algorithm $alg$}\label{proc:feasibility}

\setcounter{AlgoLine}{0}
\SetAlgoLined
\LinesNumbered
\everypar={\nl}

\hrule height\algoheightrule\kern3pt\relax
\KwIn{depth $P$; dimensions $d$, $\tilde d$, $l_1.\ldots,l_P$; number of exploratory queries $k$; algorithm $alg$}
\vspace{5pt}

Sample independently $E_1,\ldots, E_P$, uniform $\tilde d$-dimensional subspaces of $\Rbb^d$

Initialize $n_p\gets 0$ for $p\in[P]$ and set memory of $alg$ to $\mb 0$

\While{$alg$ has not queried a successful point (in $Q_{E_1,\ldots,E_P}$)}{
    Run $alg$ with current memory to obtain query $\mb x$\,

    \For{$p\in[P]$}{
        \uIf{$\mb x$ is a depth-$p$ exploratory query, i.e., satisfies Eq~\eqref{eq:pass_depths} and \eqref{eq:robustly-independent_query}}{
            \uIf{$n_p<k$}{
                $n_p \gets n_p+1$
            
                $\mb y_{n_p}^{(p)}\gets \mb x$ and sample $V_{n_p}^{(p)}$ a uniform $l_p$-dimensional subspace of $E_p$
            }
            \uElse{
                Reset $n_q\gets 1$ for $q\in[p]$

                $\mb y_1^{(q)}\gets \mb x$ and sample $V_1^{(q)}$ a uniform $l_q$-dimensional subspace of $E_q$ for $q\in [p]$
            }
            
        }
    }

    \lIf{$\Ocal_{\mb V}(\mb x) \neq  \mathsf{Success}$}{\Return $\Ocal_{\mb V}(\mb x)$ as response to $alg$}
    
    \lElse{\Return $\Ocal_{E_1,\ldots,E_P}(\mb x)$ as response to $alg$}
}

\hrule height\algoheightrule\kern3pt\relax
\end{proc}

We are now ready to introduce the complete procedure to construct the adaptive separation oracles, which is formally detailed in Procedure~\ref{proc:feasibility}. The procedure has $P\geq 1$ levels, each of which is associated to a $\tilde d$-dimensional subspace $E_1,\ldots,E_P$ that one needs to query perpendicular queries to. Whenever a new depth-$p$ exploratory query is made, we sample a new $l_p$-dimensional linear subspace $V^{(p)}_{n_p+1}$ uniformly in $E_p$ unless there were already $n_p=k$ such subspaces. In that case, we reset all exploratory queries at depth $p$ as well as all depths $q\leq p$, we pose $n_q=1$, and sample new subspaces $V_1^{(q)}$. We recall that $l_p=l$ except for $p=P$. For each level we denote by $\mb V^{(p)}:=(V_1^{(p)},\ldots,V_{n_p}^{(p)})$ the list of current linear subspaces at depth $p$. We next recall the form of the oracle $\Ocal_{\mb V^{(1)},\ldots, \mb V^{(P)}}$ from Eq~\eqref{eq:def_separation_oracle}. When this oracle returns a vector of the form $\mb g_{\mb V^{(p)}}(\mb x;\delta_p)$ for some $p\in[P]$, we say that $\mb x$ was a depth-$p$ query. For simplicity, we may use the shorthand $\Ocal_{\mb V}(\mb x)$ where $\mb V=(\mb V^{(1)},\ldots,\mb V^{(P)})$ is the collection of the previous sequences when there is no ambiguity. 

Note that $\Ocal_{(E_1),\ldots,(E_P)}$ is a valid separation oracle for $Q_{E_1,\ldots,E_P}$. By abuse of notation, we will simply write it as $\Ocal_{E_1,\ldots,E_P}$. Indeed, we can check that
\begin{equation*}
    \Ocal_{E_1,\ldots,E_P}(\mb x) = \begin{cases}
        \mb e &\text{if } \mb e^\top \mb x > -\frac{1}{2}\\
        \frac{\proj_{E_p}(\mb x)}{\|\proj_{E_p}(\mb x)\|} &\text{if } \|\proj_{E_p}(\mb x)\| >\delta_p \text{ and } p=\min\set{q\in[P], \|\proj_{E_q}(\mb x)\| >\delta_q }\\
        \mathsf{Success} &\text{otherwise}.
    \end{cases}
\end{equation*}
We will fall back to this simple separation oracle for $Q_{E_1,\ldots,E_P}$ whenever the oracles from $\Ocal_{\mb V^{(1)},\ldots,\mb V^{(P)}}$ return $\mathsf{Success}$ (lines 14-15 of Procedure~\ref{proc:feasibility}).

We first check that Procedure~\ref{proc:feasibility} indeed corresponds to a valid run for a feasibility problem and specify its corresponding accuracy parameter.

\begin{lemma}\label{lemma:check_procedure_is_feasibility_pb}
    If $P\geq 2$ and $d\geq 40P$, then with probability at least $1-e^{-d/40}$ over the randomness of $E_1,\ldots,E_P$, the set $Q_{E_1,\ldots,E_P}$ contains a ball of radius $\delta_1/2$ and all responses of the procedure are valid separating hyperplanes for this feasible set.
\end{lemma}

\begin{proof}
    Note that $\Ocal_{\mb V}$ either returns $\mathsf{Success}$ or outputs a separating hyperplane for $Q_{E_1,\ldots,E_P}$. When $\Ocal_{\mb V}$ outputs $\mathsf{Success}$, the procedure instead uses an arbitrary valid separation oracle $\Ocal_{E_1,\ldots,E_P}$ (line 15 of Procedure~\ref{proc:feasibility}). We now bound the accuracy parameter $\epsilon$. Because $E_1\,\ldots,E_P$ are sampled uniformly randomly, on an event $\Ecal$ of probability one, they are all linearly independent. Hence $Span(E_p,p\in[P])^\perp$ has dimension $d-\tilde d  P\geq d/2$ and this space is uniformly randomly distributed. In particular, \cref{lemma:concentration_projection} implies that with $\mb f:=\proj_{Span(E_p,p\in[P])^\perp}(\mb e)$,
    \begin{equation*}
        \Pbb\paren{  \|\mb f\| \leq \frac{1}{2}+ \frac{1}{40}} \leq \Pbb\paren{ \|\mb f\| \leq \sqrt{\frac{d-\tilde d P}{d}\paren{1-\frac{1}{\sqrt 5}}}} \leq e^{-(d-\tilde d P)/20} \leq e^{-d/40}.
    \end{equation*}
    Denote by $\Fcal$ the complement of this event in which $\|\mb f\|>\frac{1}{2}+ \frac{1}{40}$. Then, under $\Ecal\cap\Fcal$ which has probability at least $1-e^{-d/72}$,
    \begin{equation*}
        B_d(\mb 0,1)\cap B_d\paren{\mb f,\delta_1\land 1/40} \subset Q_{E_1,\ldots,E_P}.
    \end{equation*}
    Note that $\delta_1\leq 1/240$. Hence the left-hand side is simply $B_d(\mb 0,1)\cap B_d(\mb f,\delta_1)$. Now because $\mb f\in B_d(\mb 0,1)$, this intersection contains a ball of radius $\delta_1/2$.
\end{proof}

\subsection{Construction of the feasibility game for all depths}

Before trying to prove query lower bounds for memory-constrained algorithms under the feasibility Procedure~\ref{proc:feasibility}, we first define a few concepts. For any $p\in[P]$, we recall that a depth-$p$ \emph{period} as the interval $[t_1,t_2)$ between two consecutive times $t_1<t_2$ when $n_p$ was reset---we consider that it was also reset at time $t=1$. Except for time $t=1$, the reset happens at lines 11-12 of Procedure~\ref{proc:feasibility}. Note that at the beginning of a period, all probing subspaces $V_i^{(p)}$ for $i\in[k]$ are also reset (they will be overwritten during the period). We say that the period $[t_1,\ldots,t_2)$ is \emph{complete} if during this period the algorithm queried $k$ depth-$p$ exploratory queries, or equivalently, if at any time during this period one had $n_p=k$. 

The query lower bound proof uses an induction argument on the depth $p\in[P]$. Precisely, we aim to show a lower bound on the number of iterations needed to complete a period for some depth $p\in[P]$. To do so, instead of working directly with Procedure~\ref{proc:feasibility}, we prove lower bounds on the following Depth-$p$ Game~\ref{game:feasibility_game} for $p\in[P]$. Intuitively, it emulates the run of a completed depth-$p$ period from the original feasibility procedure except for the following main points.

\begin{game}[ht!]
        
\caption{Depth-$p$ Feasibility Game}\label{game:feasibility_game}

\setcounter{AlgoLine}{0}
\SetAlgoLined
\LinesNumbered

\everypar={\nl}

\hrule height\algoheightrule\kern3pt\relax
\KwIn{game depth $p\in[P]$; number of exploratory queries $k$; dimensions $d$, $\tilde d$, $l_1,\ldots,l_P$; $M$-bit memory algorithm $alg$; number queries $T_{max}$}
\vspace{5pt}

\textit{Oracle:} Sample independently $E_1,\ldots,E_P$, uniform $\tilde d$-dimensional linear subspaces of $\Rbb^d$ and for $i\in[P-p]$ sample $k$ uniform $l_{p+i}$-dimensional subspaces of $E_{p+i}$: $V_j^{(p+i)}$ for $j\in[k]$ 

\textit{Player:} Observe $E_1,\ldots,E_P$ and $V_j^{(p+i)}$ for $i\in[P-p]$, $j\in[k]$. Submit to oracle an $M$-bit message $\mathsf{Message}$, and for all $i\in[P-p]$ submit an integer $n_{p+i}\in[k]$ and vectors $\mb y_j^{(p+i)} \in \Rbb^d$ for $j\in[n_{p+i}]$

\textit{Oracle:} Initialize memory of $alg$ to $\mathsf{Message}$. Set $n_{p'}\gets 0$ for $p'\in[p]$ and for $i\in[P-p]$, and $n_{p+i}$ as submitted by player. Reset probing subspaces $V_j^{(p+i)}$ for $i\in[P-p]$ and $j>n_{p+i}$

\textit{Oracle:} \For{$t\in [T_{max}]$}{
    Run $alg$ with current memory to get query $\mb x_t$

    Update exploratory queries $\mb y_i^{(p')}$ and subspaces $V_i^{(p')}$ for $p'\in[P], i\in[n_{p'}]$ as in Procedure~\ref{proc:feasibility} and \Return $\mb g_t = \Ocal_{\mb V}(\mb x_t)$ as response to $alg$ 
    
    \lIf{$n_p$ was reset because a deeper period was completed ($n_p\gets 1$ in line 12 of Procedure~\ref{proc:feasibility})}{player loses, \textbf{end} game}

    \lIf{$n_p=k$}{player wins. \textbf{end} game}
}

Player loses

\hrule height\algoheightrule\kern3pt\relax
\end{game}

\begin{enumerate}
    \item We allow the learner to have access to some initial memory about the subspaces $E_1,\ldots,E_P$. This intuitively makes it simpler for the player.
    \item The learner can also submit some vectors $\mb y_j^{(p+i)}$ for $i\in[P-p]$ that emulate exploratory queries for deeper depths than $p$.
    \item The objective is not to find a point in the feasible set $Q_{E_1,\ldots,E_P}$ anymore, but simply to complete the depth-$p$ period, that is, make $k$ depth-$p$ exploratory queries. Note that the player loses if a deeper period was completed (line 7 of Game~\ref{game:feasibility_game}). Hence, during a winning run, no depth-$p'$ periods with $p'>p$ were completed.
    \item The player has a maximum number of calls to the separation oracle available $T_{max}$.
\end{enumerate}
Note that the role of the player here is only to submit the message $\mathsf{Message}$ and exploratory queries for depths $p'>p$ to the oracle, in addition to providing a $M$-bit memory algorithm $alg$. In the second part of Game~\ref{game:feasibility_game}, the oracle directly emulates a run of a depth-$p$ period of Procedure~\ref{proc:feasibility} (without needing input from the player). We will also use the term depth-$q$ period for $q\in [p]$ for this game. One of the interests of the third point is to make the problem symmetric in terms of the objectives at depth $p\in[P]$, which will help implement our recursive query lower bound argument.

\subsection{Properties of the probing subspaces}

The probing subspaces $V_i^{(q)}$ at any level $q\in[P]$ are designed to ``probe'' for the query $\mb x$ being close to the perpendicular space to $E_q$. Precisely, the first step of the proof is to show that if a query $\mb x_t$ passed probes at level $q$, then with high probability it satisfied $\|\proj_{E_p}(\mb x_t)\| \leq \eta_p$. To prove this formally, we introduce the Probing Game~\ref{game:probing_game}. It mimics Procedure~\ref{proc:feasibility} but focuses on exploratory queries at a single layer. We recall its definition here for the sake of exposition.

\begin{game}[h!]

\repeatAlgoCaption{alg:probing_game}{Probing Game}

\setcounter{AlgoLine}{0}

\SetAlgoLined
\LinesNumbered

\everypar={\nl}

\hrule height\algoheightrule\kern3pt\relax
\KwIn{dimension $d$, response dimension $l$, number of exploratory queries $k$, objective $\rho>0$}

\textit{Oracle:} Sample independent uniform $l$-dimensional subspaces $V_1,\ldots,V_k$ in $\Rbb^d$.

\For{$i\in[k]$}{
    \textit{Player:} Based on responses $V_j,j<i$, submit query $\mb y_i\in\Rbb^d$

    \textit{Oracle:} \Return $V_i$ to the player
}

Player wins if for any $i\in[k]$ there exists a vector $\mb z\in Span(\mb y_j,j\in[i])$ such that
\begin{equation*}
    \|\proj_{Span(V_j,j\in[i])}(\mb z) \| \leq \rho \quad \text{and} \quad \|\mb z\|=1.
\end{equation*}

\hrule height\algoheightrule\kern3pt\relax
\end{game}

Our goal is to give a bound on the probability of success of any strategy for Game~\ref{game:probing_game}. To do so, we need to prove the \cref{thm:random_triangular_matrix} which essentially bounds the smallest singular value of random matrices for which the upper triangle components are all i.i.d. standard normal random variables. We start by proving the result in the case when the lower triangle is identically zero, which may be of independent interest.

\begin{corollary}\label{cor:M0_singular_value}
    Let $C\geq 2$ be an integer and $m=Cn$. Let $\mb M\in\Rbb^{n\times m}$ be the random matrix such that all coordinates $M_{i,j}$ in the upper-triangle $j>(i-1)C$ are together i.i.d. Gaussians $\Ncal(0,1)$ and the lower triangle is zero, that is $M_{i,j}=0$ for $j\leq (i-1)C$. Then for any $\alpha\in(0,1]$, if $C\geq C_\alpha\ln n$, we have
    \begin{equation*}
        \Pbb \paren{\sigma_1(\mb M) < \frac{1}{6} \sqrt{\frac{C}{n^{\alpha}}} } \leq 3e^{-C/16}.
    \end{equation*}
    Here $C_\alpha = (C_2/\alpha)^{\ln 2/\alpha}$ for some universal constant $C_2\geq 8$.
\end{corollary}

\begin{proof}
    We use an $\epsilon$-net argument to prove this result. However, we will need to construct the argument for various scales of $\epsilon$ because of the non-homogeneity of the triangular matrix $\mb M$. First, we can upper bound the maximum singular value of $\mb M$ directly as follows (see e.g. \cite[Theorem 4.4.5]{vershynin2020high} or \cite[Exercise 2.3.3]{tao2023topics}),
    \begin{equation*}
        \Pbb(\|\mb M\|_{op} > C_1\sqrt {Cn}) \leq e^{-Cn}
    \end{equation*}
    for some universal constant $C_1\geq 1$. For convenience, let us write $\mb M(i):=(M_{u,v})_{u\in[i, n],v\in[1+C(i-1),Cn]}$ for all $i\in[n]$. Applying the above result implies that
    \begin{equation*}
        \Pbb(\|\mb M(i)\|_{op} > C_1\sqrt {C(n-i+1)}) \leq e^{-C(n-i+1)},\quad i\in[n].
    \end{equation*}
     In particular, the event
    \begin{equation}\label{eq:upper_bound_op_norm}
        \Ecal = \bigcap_{i\in[n]}\set{\|\mb M( i)\|_{op}\leq C_1\sqrt{C(n-i+1)}}
    \end{equation}
    satisfies $\Pbb(\Ecal^c)\leq \sum_{i\in[n]} e^{-C(n-i+1)}\leq 2e^{-C}$ because $C\geq \ln(2)$. Next, let $\epsilon=1/(6nC_1)$ and fix a constant $\alpha\in(0,1)$. 
    For any $i\in[n]$, we construct an $\epsilon$-net $\Sigma_i$ of all unit vectors which have non-zero coordinates in $[i,n]$, that is $S_{d-1}\cap\{\mb x: \forall j<i,x_j=0\} $. For any $\mb x\in S_{d-1}$ note that
    \begin{equation*}
         \|\mb M^\top \mb x\|^2 \sim \sum_{i=1}^n Y_i \cdot \|(x_j)_{j\in[i]}\|^2,
    \end{equation*}
    where $Y_1,\ldots,Y_n\overset{i.i.d.}{\sim} \chi^2(C)$ are i.i.d. chi-squared random variables. Hence, for
    and $i\in[n]$, together with \cref{lemma:concentration_gaussian} this shows that
    \begin{equation*}
        \Pbb\paren{\|{\mb M}^\top \mb x\|<  \frac{1}{2}\sqrt{C(n-i+1) \sum_{j\in[i]}x_j^2} } \leq \Pbb\paren{ \sum_{j=i}^n Y_j < \frac{C(n-i+1)}{4}} \leq e^{-C(n-i+1)/8}.
    \end{equation*}
    In the last inequality, we used the fact that $\sum_{j=i}^n Y_i\sim\chi^2(C(n-i+1))$ hence has the distribution of the squared norm of a Gaussian $\Ncal(0,\mb I_{n-i+1})$.
    Next, because $\Sigma_i$ is an $\epsilon$-net of a sphere of dimension $n-i+1$ (restricted to the last $n-i+1$ coordinates) there is a universal constant $C_2\geq 1$ such that $|\Sigma^{(i)}| \leq \paren{C_2/\epsilon}^{n-i+1}$ (e.g. see \cite[Lemma 2.3.4]{tao2023topics}. For any $i\in[n]$, we write $l(i) = 1\lor (n-L_\alpha(n-i+1)+1)$ for a constant integer $L_\alpha\geq 2 (C_1/\alpha)^{\ln 2/\alpha}$ to be fixed later. Note that by construction we always have $n-l(i)+1\leq L_\alpha(n-i+1)$. Finally, we define the event
    \begin{equation*}
        \Fcal = \bigcap_{i\in[n]}\bigcap_{\mb x\in\Sigma_{l(i)}}\set{\|{\mb M}^\top \mb x\| \geq  \frac{1}{2}\sqrt{C(n-i+1) \sum_{j\in[i]}x_j^2}}.
    \end{equation*}
    Provided that $e^{C/16} \geq 2(C_2/\epsilon)^{L_\alpha}$ the union bound implies
    \begin{equation*}
        \Pbb(\Fcal^c) \leq \sum_{i=1}^n |\Sigma_{l(i)}|e^{-C(n-i+1)/8}\leq  \sum_{i=1}^n \paren{\frac{C_2^{L_\alpha}}{\epsilon^{L_\alpha} e^{C/8}}}^{n-i+1} \leq e^{-C/16}.
    \end{equation*}
    We note that the equation $e^{C/16} \geq 2(C_2/\epsilon)^{L_\alpha}$ is equivalent to $C/16\geq {L_\alpha}\ln(6C_1C_2\sqrt n) + \ln 2$, which is satisfied whenever $C\geq C_3 {L_\alpha}\ln n$ for some universal constant $C_3\geq 16$. We suppose that this is the case from now on and that the event $\Ecal\cap\Fcal$ is satisfied.

    We now construct a sequence of growing indices as follows. For all $k\leq \floor{\log_2(n)}-1:=k_0$, we let $i_k:=n-2^k+1$ for $k<k_0$. For any $\mb x\in S_{d-1}$, we define
    \begin{equation*}
        k(\mb x):=\arg\max_{k\in\{0,\ldots,k_0\}} e^{\alpha k} f(k;\mb x)\quad \text{where}\quad f(k;\mb x):=(n-i_k+1)\sum_{j\in[i_k]}x_j^2.
    \end{equation*}
    Note that because $i_0=n$, the inner maximization problem must have value at least $\sum_{j\in[n]}x_j^2=1$. Also, for any $k\geq k(\mb x)$, we have
    \begin{equation}\label{eq:decay_norm}
         f(k;\mb x) \leq e^{-\alpha(k-k(\mb x))} f(k(\mb x);\mb x) .
    \end{equation}
    We will also use the shortcut $l(\mb x) = l(i_{k(\mb x)})$. Let $\hat{\mb x}$ be the nearest neighbor of $\mb x$ in $\Sigma_{l(\mb x)}$, that is the nearest neighbor of the vector $\mb y = (x_j \1_{j\geq l(\mb x)})_{j\in[n]}$. By construction of the $\epsilon$-nets we have $\|\mb y-\hat{\mb x}\|\leq \epsilon$. Now observe that $\mb x - \mb y$ only has non-zero values for the first $l(\mb x)-1$ coordinates. We let $r(\mb x)$ be the index such that $i_{r(\mb x)+1}< l(\mb x)-1\leq i_{r(\mb x)}$ (with the convention $i_{k_0+1}=0$). We decompose $\mb x - \mb y$ as the linear sum of vectors for $r\in[r(\mb x), k_0]$ such that the $r$-th vector only has non-zero values for coordinates in $(i_{r+1},i_r]$.
    Then, using the triangular inequality
    \begin{align}
        \|{\mb M}^\top \mb x\| &\geq \|{\mb M}^\top \hat{\mb x}\| - \|{\mb M}^\top (\mb y-\hat{\mb x})\| - \sum_{r=r(\mb x)}^{k_0} \|{\mb M}^\top (x_j \1_{i_{r+1}<j\leq i_r\land (l(\mb x)-1)})_{j\in[n]}\|\\
        &\geq \|{\mb M}^\top \hat{\mb x}\| - \epsilon \|\mb M(l(\mb x))^\top \|_{op}  
        - \sum_{r=r(\mb x)}^{k_0} \|\mb M(i_r)\|_{op}\sqrt{\sum_{j\in [i_r]}x_j^2}\\
        &\geq \|{\mb M}^\top \hat{\mb x}\| - C_1\epsilon \sqrt{Cn}  
        - C_1 \sum_{r=r(\mb x)}^{k_0}\sqrt{C\cdot f(r;\mb x)}.\label{eq:triangular_inequality}
    \end{align}
    In the last inequality we used Eq~\eqref{eq:upper_bound_op_norm}.
    We start by treating the second term containing $\epsilon$. We recall that the inner maximization problem defining $k(\mb x)$ has value at least $1$, so that
    \begin{equation}\label{eq:lower_bound_objective}
        f(k(\mb x);\mb x) \geq e^{-\alpha k(\mb x)} \geq e^{-\alpha k_0} \geq n^{-\alpha}.
    \end{equation}
    As a result, recalling that $\alpha\in(0,1]$, we have
    \begin{equation}\label{eq:upper_bound_epsilon}
          C_1\epsilon\sqrt{Cn} \leq \frac{1}{6}\sqrt{\frac{C}{n}} \leq \frac{\sqrt{C\cdot f(k(\mb x);\mb x)}}{6}.
    \end{equation}

    Next, by Eq~\eqref{eq:decay_norm} we have
    \begin{align*}
        \sum_{r=r(\mb x)}^{k_0}\sqrt{f(r;\mb x)}\leq \sqrt{f(k(\mb x);\mb x)} 
        \sum_{r=r(\mb x)}^{k_0} e^{-\alpha(r-k(\mb x))/2} 
        &\leq \frac{e^{-\alpha(r(\mb x)-k(\mb x))/2}}{1-e^{-\alpha/2}}\sqrt{f(k(\mb x);\mb x)} \\
        &\leq \frac{4}{\alpha}e^{-\alpha(r(\mb x)-k(\mb x))/2} \sqrt{f(k(\mb x);\mb x)}.
    \end{align*}
    In the last inequality, we used the fact that $\alpha\in(0,1)$ and that $e^{-x}\leq 1-x/2$ for all $x\in[0,1]$.
    Next, recall that $i_{r(\mb x)+1}<l(\mb x)-1$. Hence, either $l(\mb x)=1$ in which case the sum above is empty, or we have
    \begin{equation*}
        2^{r(\mb x)+1} = n-i_{r(\mb x)+1}+1 \geq n-l(\mb x)+1 = {L_\alpha}(n-i_{k(\mb x)}+1) = {L_\alpha} 2^{k(\mb x)}.
    \end{equation*}
    As a result, $r(\mb x)-k(\mb x) \geq \log_2 {L_\alpha} -1 .$ We now select $L_\alpha$ to be the minimum integer for which $\log_2 L_\alpha-1\geq \frac{2\ln(24C_1/\alpha)}{\alpha}.$
    In turn, this implies
    \begin{equation}\label{eq:bound_tail}
        C_1\sum_{r=r(\mb x)}^{k_0}\sqrt{C\cdot f(r;\mb x)}\leq \frac{\sqrt{C\cdot f(k(\mb x);\mb x)}}{6}.
    \end{equation}
    Putting together Eq~\eqref{eq:triangular_inequality}, \eqref{eq:upper_bound_epsilon}, \eqref{eq:bound_tail} we obtained
    \begin{align*}
        \|{\mb M}^\top \mb x\|  \geq \|{\mb M}^\top \hat{\mb x}\| - \frac{1}{3}\sqrt{C \cdot f(k(\mb x);\mb x)} \geq \frac{1}{6}\sqrt{C \cdot f(k(\mb x);\mb x)} \geq \frac{1}{6}\sqrt{\frac{C}{n^\alpha}}.
    \end{align*}
    where in the second inequality, we used the event $\Fcal$ and the last inequality used Eq~\eqref{eq:lower_bound_objective}. In summary, we showed that if $C\geq C_2L_\alpha \ln n$, then
    \begin{equation*}
        \Pbb\paren{\sigma_1(\mb M) \geq \frac{1}{6}\sqrt{\frac{C}{n^\alpha}}} \geq\Pbb(\Ecal\cap\Fcal) \geq 1-3e^{-C/16}.
    \end{equation*}
    This ends the proof of the result.
\end{proof}

We are now ready to prove \cref{thm:random_triangular_matrix} which generalizes \cref{cor:M0_singular_value} to some cases when the lower triangle can be dependent on the upper triangle.

\vspace{3mm}

\begin{proof}[of \cref{thm:random_triangular_matrix}]
    The main part of the proof reduces the problem to the case when the lower-triangle is identically zero. To do so, let $\mb x\sim\Ucal(S_{n-1})$ be a random unit vector independent of $\mb M$. We aim to lower bound $\|\mb M ^\top \mb x\|$. Let us define some notations for sub-matrices and subvectors of $\mb M$ as follows for any $i\in[n]$,
    \begin{align*}
        \mb M^{(i)}&:= (M_{u,v})_{u\in[i],v\in[Ci]}\\
        \mb N^{(i)}&:= (M_{u,v})_{u\in[i],v\in[C(i-1)]}\\
        \mb a^{(i)}&:= (M_{i,v})_{v\in[C(i-1)]},\\
        \mb A^{(i)}&:= (M_{u,v})_{u\in[i],v\in[C(i-1)+1,C_i]}.
    \end{align*}
    For a visual representation of these submatrices, we have the following nested construction
    \begin{equation*}
        \mb M^{(i)} = \begin{tabular}{cccc|}
            \hline
\multicolumn{3}{|c|}{\multirow{3}{*}{$\mb M^{(i-1)}$}} & \multirow{5}{*}{$\mb A^{(i)}$} \\
\multicolumn{3}{|c|}{}                   &  \\
\multicolumn{3}{|c|}{}                   &  \\ \cline{1-3}
\multicolumn{3}{|c|}{\multirow{2}{*}{${\mb a^{(i)}}^\top$}} & \\
\multicolumn{3}{|c|}{}                   &  \\ \hline
            \end{tabular} = [\mb N^{(i)},\mb A^{(i)}]\;,\quad i\in[n]
    \end{equation*}
    and $\mb M^{(n)}=\mb M$. For now, fix $i\in [2,n]$. We consider any realization of the matrix $\mb N^{(i)}$, that is, of the matrix $M^{(i-1)}$ and the vector $\mb a^{(i)}$. Then,
    \begin{equation}\label{eq:simple_dominance}
        {\mb N^{(i)}}^\top \mb N^{(i)} = {\mb M^{(i-1)}}^\top \mb M^{(i-1)}  + \mb a^{(i)} {\mb a^{(i)}}^\top \succeq  {\mb M^{(i-1)}}^\top \mb M^{(i-1)}.
    \end{equation}
    Next, let $0\leq \sigma_1\leq \ldots\leq \sigma_i$ be the singular values of the matrix $\mb \Nbb^{(i)}$. We also define
    \begin{equation*}
        {\mb {\tilde N}}^{(i)}:= [{\mb M^{(i-1)}}^\top,\mb 0_{C(i-1),1}]^\top,
    \end{equation*}
    which is exactly the matrix $\mb N^{(i)}$ if we had $\mb a^{(i)}=\mb 0$, and let $0\leq \tilde \sigma_1\leq \ldots\leq \tilde \sigma_i$ be the singular values of $\mb {\tilde N}^{(i)}$. Then, Eq~\eqref{eq:simple_dominance} implies that for all $j\in[i]$, one has $\tilde \sigma_j \geq \sigma_j$.
    After constructing the singular value decomposition of the two matrices $\mb N^{(i)}$ and ${\mb {\tilde N}}^{(i)}$, this shows that there exists an orthogonal matrix $\mb U^{(i)}\in\Ocal_i$ such that
    \begin{equation*}
        \| {\mb N^{(i)}}^\top \mb x^{(i)} \| \geq \| {{\mb {\tilde N}}^{(i)} }^\top (\mb U^{(i)}\mb x^{(i)})\|,\quad \forall \mb x^{(i)}\in\Rbb^i.
    \end{equation*}
    Finally, defining the matrix
    \begin{equation}\label{eq:def_tilde_M}
         \mb {\tilde M}^{(i)} = [\mb{\tilde N}^{(i)}, \mb U^{(i)} \mb{A}^{(i)}] = \begin{tabular}{|c|c|}
            \hline
$\mb M^{(i-1)}$ & \multirow{2}{*}{$ \mb U^{(i)} \mb { A}^{(i)}$} \\ \cline{1-1}
$0$ &  \\ \hline
    \end{tabular}\;,
    \end{equation}
    we obtained that
    \begin{equation*}
        \|{\mb M^{(i)}}^\top \mb x^{(i)}\| \geq \|{\mb {\tilde M}^{(i)} }^\top (\mb U^{(i)} \mb x^{(i)})\|,\quad \forall \mb x^{(i)}\in\Rbb^i.
    \end{equation*}
    We now make a few simple remarks. First, the uniform distribution on $S_{i-1}$ is invariant by the rotation $\mb U^{(i)}$, hence $\mb x^{(i)}\sim\Ucal(S_{i-1})$ implies $\mb U^{(i)}\mb x^{(i)} \sim\Ucal(S_{i-1})$. Next, by isometry of Gaussian vectors, conditionally on $\mb N^{(i)}$, the matrix $\mb{\tilde A}^{(i)}:=\mb U^{(i)} \mb A^{(i)}$ is still distributed exactly as a matrix with i.i.d. $\Ncal(0,1)$ entries. In turn, $\|{\mb {\tilde A}^{(i)}}^\top (\mb U^{(i)}\mb x^{(i)})\|^2$ is still distributed as a chi-squares $\chi^2(C)$ independent from $\mb N^{(i)}$, just as for $\|{\mb A^{(i)}}^\top \mb x^{(i))}\|^2$. As a summary of the past arguments, using the notation $\succeq_{st}$ for stochastic dominance we obtained for $\mb x^{(i)}\sim\Ucal(S_{i-1})$ sampled independently from $\mb M$,
    \begin{equation*}
        \|{\mb M^{(i)}}^\top \mb x^{(i)}\|^2 \succeq_{st} \|{\mb {\tilde M}^{(i)}}^\top \mb x^{(i)}\|^2 = \|{\mb M^{(i-1)}}^\top (x^{(i)}_j)_{j\in[i-1]}\|^2 + \|{\mb {\tilde A}^{(i)}}^\top \mb x^{(i)}\|^2.
    \end{equation*}
    As a result,
    \begin{equation}\label{eq:recursion}
        \|{\mb M^{(i)}}^\top \mb x^{(i)}\|^2 \succeq_{st} (1-(x^{(i)}_i)^2) \|{\mb M^{(i-1)}}^\top \mb x^{(i-1)}\|^2 + Y_i,
    \end{equation}
    where $\mb x^{(i-1)}\sim\Ucal(S_{i-2})$ (can be resampled independently from $\mb x^{(i)}$ if wanted) and $Y_i\sim\chi^2(C)$ is independent from $\mb M^{(i-1)}$, $x_i^{(i)}$ and $\mb x^{(i-1)}$.

    Using Eq~\eqref{eq:recursion} recursively constructs a sequence of random independent vectors $\mb x^{(i)}\sim\Ucal(S_{i-1})$ for $i\in[n]$, as well as an independent sequence of i.i.d. $Y_1,\ldots,Y_n\overset{i.i.d.}{\sim}\chi^2(C)$ such that for $\mb x\sim\Ucal(S_{n-1})$ sampled independently from $\mb M$ and $Y_1,\ldots,Y_n$,
    \begin{equation}\label{eq:final}
        \|\mb M^\top \mb x\|^2 \succeq_{st} \sum_{i=1}^n Y_i \prod_{j=i+1}^n (1-(x_j^{(j)})^2)\sim \sum_{i=1}^n Y_i \cdot \|(x_j)_{j\in[i]}\|^2.
    \end{equation}
    The last inequality holds because of the observation that if $\mb x\sim\Ucal(S_{i-1})$, then the first $i-1$ coordinates of $\mb x$ are distributed as a uniform vector $\Ucal(S_{i-2})$ rescaled by $\sqrt{1-x_i^2}$. Next, we construct the matrix $\mb M^0\in\Rbb^{n\times Cn}$ that corresponds exactly to $\mb M$ had the vectors $\mb a^{(i)}$ been all identically zero: all coordinates $M^0_{i,j}$ for $j>(i-1)C$ are i.i.d. $\Ncal(0,1)$ and for $j<(i-1)C$ we have $M^0_{i,j}=0$. We can easily check that
    \begin{equation}\label{eq:exploded_eq}
         \|{\mb M^0}^\top \mb x\|^2 \sim \sum_{i=1}^n Y_i \cdot \|(x_j)_{j\in[i]}\|^2.
    \end{equation}
    Combining this equation together with Eq~\eqref{eq:final} shows that 
    \begin{equation*}
        \|\mb M^\top \mb x\| \succeq_{st} \|{\mb M^0}^\top \mb x\|,\quad \mb x\sim\Ucal(S_{n-1}).
    \end{equation*}
    Hence, intuitively the worst case to lower bound the singular values of $\mb M$ corresponds exactly to the case when all the vectors $\mb a^{(i)}$ were identically zero. While the previous equation concisely expresses this idea, we need a slightly stronger statement that also characterizes the form of the coupling between $(\mb M,\mb x)$ and $(\mb M^0,\tilde{\mb x})$ such that almost surely, $\|\mb M^\top \mb x\| \geq \|{\mb M^0}^\top \tilde {\mb x}\|$. Going back to the construction of these couplings, we note that $\tilde {\mb x}$ is obtained from $\mb x$ by applying a sequence of rotations to some of its components---the matrices $\mb U^{(i)}$---and these rotations only depend on $\mb M$. Similarly, the matrix $\mb M^0$ is coupled to $\mb M$ but is independent from $\mb x$ (see Eq~\eqref{eq:def_tilde_M}). The last remark is crucial because it shows that having fixed a realization for $\mb M$ and $\mb M^0$, for $\mb x\sim\Ucal(S_{d-1})$ sampled independently from $\mb M$ we have
    \begin{equation*}
        \|\mb M^\top \mb x\| \geq \|{\mb M^0}^\top \tilde{\mb x}\| \geq \sigma_1(\mb M^0).
    \end{equation*}
    As a result, the above equation holds for almost all $\mb x\in S_{d-1}$, which is sufficient to prove that almost surely, $\sigma_1(\mb M) \geq \sigma_1(\mb M^0)$. Together with the lower bound on $\sigma_1(\mb M^0)$ from \cref{cor:M0_singular_value} we obtain the desired result.
\end{proof}

Having these random matrices results at hand, we can now give query lower bounds on the Probing Game~\ref{game:probing_game}.

\begin{theorem}
\label{thm:no_small_vectors}
    Suppose that $4kl\leq d$ and $l\geq C_\alpha \ln k$ for a fixed $\alpha\in(0,1]$, where $C_\alpha$ is as defined in \cref{thm:random_triangular_matrix}. Let $\rho\leq \frac{1}{12}\sqrt{\frac{l}{dk^\alpha}}$. Then, no algorithm wins at the Probing Game~\ref{game:probing_game} with probability at least $4de^{-l/16}$.
\end{theorem}

\begin{proof}[of \cref{thm:no_small_vectors}]
    We start with defining some notations. We observe that for any $i\in[k]$, sampling $V_i$ uniformly within $l$-dimensional subspaces of $\Rbb^d$ is stochastically equivalent to sampling some Gaussian vectors $\mb v_i^{(1)},\ldots,\mb v_i^{(l)}\overset{i.i.d.}{\sim}\Ncal(0,\mb I_d)$ and constructing the subspace $Span(\mb v_i^{(r)},r\in[l])$. Without loss of generality, we can therefore suppose that the probing subspaces were sampled in this way. We then define a matrix $\Pi$ summarizing all probing subspaces as follows:
    \begin{equation*}
        \mb \Pi:= [\mb v_1^{(1)},\ldots,\mb v_1^{(l)}, \mb v_2^{(1)},\ldots,\mb v_2^{(l)},\ldots ,\mb v_k^{(1)},\ldots, \mb v_k^{(l)}].
    \end{equation*}
    Next, we let $\mb x_1,\ldots,\mb x_k$ be the sequence resulting from doing the Gram-Schmidt decomposition from $\mb y_1,\ldots,\mb y_k$, that is
    \begin{equation*}
        \mb x_i = \begin{cases}
            \frac{\proj_{Span(\mb y_j,j<i)^\perp}(\mb y_i)}{\|\proj_{Span(\mb y_j,j<i)^\perp}(\mb y_i)\|} &\text{if }\mb y_i\notin Span(\mb y_j,j<i)\\
            \mb 0 &\text{otherwise.}
        \end{cases}
    \end{equation*}
    Note that it is always advantageous for the player to submit a vector $\mb y_i\notin Span(\mb y_j,j<i)$ so that the space $Span(\mb y_j,j\in[i])$ is as large as possible. Without loss of generality, we will therefore suppose that $\mb y_i\notin Span(\mb y_j,j<i)$ for all $i\in[k]$. In particular, $\mb x_1,\ldots,\mb x_k$ form an orthonormal sequence. Last, we construct the Gram matrix $\mb M\in\Rbb^{k\times lk}$ as
    \begin{equation*}
        \mb M := [\mb x_1,\ldots,\mb x_k]^\top \Pi.
    \end{equation*}
    The next part of this proof is to lower bound the smallest singular value of $\mb M$. To do so, we show that it satisfies the assumptions necessary for \cref{thm:random_triangular_matrix}. 
    
    For $i\in[k]$, conditionally on all $\mb x_j$ for $j\in[i]$ as well as all $\mb v_j^{(r)}$ for $j<i$ and $r\in[l]$, the vectors $\mb v_i^{(r)}$ for $r\in[l]$ are still i.i.d. Gaussian vectors $\Ncal(0,\mb I_d)$. Now by construction, $\mb x_1,\ldots,\mb x_i$ form an orthonormal sequence. Hence, by isometry of the Gaussian distribution $\Ncal(0,\mb I_d)$, the random variables $(\mb x_j^\top \mb v_i^{(r)})_{j\in[i],r\in[l]}$ are together i.i.d. and distributed as normal random variables $\Ncal(0,1)$. Because this holds for all $i\in[k]$, this proves that the upper-triangular components of $\mb M$ are all i.i.d. normal $\Ncal(0,1)$. Next, for any $i\in[k]$, recall that the vectors $\mb x_1,\ldots,\mb x_i$ are independent from all the vectors $\mb v_j^{(r)}$ for $j\geq i$ and $r\in[l]$. In particular, this shows that all components $M_{a,b}=\mb x_a^\top \mb v_b^{(r)}$ for $a\leq i$, $b<i$ and $r\in[l]$ are all independent from the normal variables $M_{a,b}=\mb x_a^\top \mb v_b^{(r)}$ for $a\in[n]$, $b\geq a$ and $b\geq i$ and $r\in[l]$. This exactly shows that $\mb M$ satisfies all properties for \cref{thm:random_triangular_matrix}. In particular, letting $\mb M^{(i)} = (M_{u,v})_{u\in[i],v\in[Ci]}$ for all $i\in[k]$, this also proves that $\mb M^{(i)}$ satisfies the required conditions.
    
    Now fix $\alpha\in(0,1]$ and suppose that $l\geq C_\alpha \ln k$. In particular, without loss of generality $l\geq 4$. \cref{thm:random_triangular_matrix} shows that the event
    \begin{equation*}
        \Ecal:= \bigcap_{i\in[k]}\set{ \sigma_1(\mb M^{(i)}) \geq \frac{1}{6}\sqrt{\frac{l}{i^\alpha}} }
    \end{equation*}
    has probability at least $1-3ke^{-l/16}$. In the last part of the proof, we show how to obtain upper bounds on the probability of success of the player for Game~\ref{game:probing_game} given these singular values lower bounds. Recall that the vectors $\mb v_i^{(r)}$ for $i\in[k]$ and $r\in[l]$ are all i.i.d. Gaussians $\Ncal(0,\mb I_d)$. Hence, letting $\mb \Pi^{(i)}$ be the matrix that contains the first $li$ columns of $\mb \Pi$ (that is, vectors $\mb v_j^{(r)}$ for $r\in[l]$ and $j\in[i]$), by \cref{thm:random_rectangular_matrices} we have
    \begin{equation*}
        \Pbb\paren{ \|\mb \Pi^{(i)}\|_{op}\geq \sqrt d \paren{3/2+\sqrt{\frac{li}{d}}} } \leq e^{-d/8},\quad i\in[k].
    \end{equation*}
    We note that for any $i\in[k]$, one has $li/d\leq 1/4$. We then define the event
    \begin{equation*}
        \Fcal := \set{ \|\mb \Pi^{(i)}\|_{op} < 2\sqrt d  ,i\in[k]},
    \end{equation*}
    which has probability at least $1-ke^{-d/8}$. On the event $\Ecal\cap \Fcal$, for any $i\in[k]$ and $\mb z\in Span(\mb y_j,j\in[i])$, writing $\mb z = \sum_{j\in[i]}\lambda_j \mb x_j$ we have $\|\mb z\|=\|\mb \lambda\|$ since the sequence $\mb x_1,\ldots,\mb x_k$ is orthonormal. As a result,
    \begin{align*}
        \|\proj_{Span(V_j,j\in[i])}(\mb z)\| \geq \frac{\|{\mb \Pi^{(i)}}^\top \mb z\|}{ \|\mb\Pi^{(i)}\|_{op}}
        &= \frac{\|(\mb M^{(i)})^\top \mb \lambda\|}{ \|\mb\Pi^{(i)}\|_{op}}
        >  \frac{\sigma_1(\mb M^{(i)})}{2\sqrt d} \|\mb z\|
        \geq \frac{1}{12} \sqrt{\frac{l}{dk^{\alpha}}} \|\mb z\|.
    \end{align*}
    That is, under $\Ecal\cap\Fcal$ the player does not win for the parameter $\rho = \frac{1}{12} \sqrt{\frac{l}{dk^{\alpha}}}$. Last, by the union bound, $\Pbb(\Ecal\cap\Fcal) \geq 1-3ke^{-l/16} -ke^{-d/8} \geq 1-4de^{-l/16}$.
    This ends the proof.
\end{proof}

We next define the notion of \emph{proper} period for which the probing subspaces were a good proxy for the projection onto $E_q$.

\begin{definition}[Proper periods]
\label{def:proper_period}
    Let $q\in[P]$. We say that a depth-$q$ period starting at time $t_1$ is \emph{proper} when for any time $t\geq t_1$ during this period, if $\mb x_t$ is a depth-$p'$ query for $p'>q$, then
    \begin{equation*}
        \|\proj_{E_q}(\mb x_t)\| \leq \eta_q.
    \end{equation*}
\end{definition}

This definition can apply equally to periods from the 
 Procedure~\ref{proc:feasibility} or the Depth-$p$ Game~\ref{game:feasibility_game} (provided $q\leq p$). By proving a reduction from running a given period to the Probing Game~\ref{game:probing_game}, we can show that \cref{thm:no_small_vectors} implies most periods are proper.

\begin{lemma}\label{lemma:most_periods_proper}
    Let $q\in[P]$ ($q\in[p]$ for Depth-$p$ Game~\ref{game:feasibility_game}). Suppose that $4l_qk \leq \tilde d$ and that $l_q\geq C_\alpha\ln k$ for some fixed constant $\alpha\in(0,1]$. For any index $j\geq 1$, define the event $\Ecal_q(j) = \{\text{$j$ depth-$q$ periods were started}\}$. If $\Pbb(\Ecal_q(j))>0$, then,
    \begin{equation*}
        \Pbb\paren{\text{$j$-th depth-$q$ period is proper} \mid \Ecal_q(j)} \geq 1-4de^{-l_q/16}.
    \end{equation*}
\end{lemma}

\begin{proof}
We prove the result in the context of Game~\ref{game:feasibility_game} which is the only part that will be needed for the rest of the paper. The exact same arguments will yield the desired result for Procedure~\ref{proc:feasibility}.
We fix a strategy for Game~\ref{game:feasibility_game}, a period depth $q\in[p]$, and a period index $J$ such that $\Pbb(\Ecal_q(J))>0$. Using this strategy, we construct a learning algorithm for the Probing Game~\ref{game:probing_game} with dimension $\tilde d$, response dimension $l_q$, $k$ exploratory queries.

This player strategy is detailed in \cref{alg:strategy_for_vector_exploration}. It works by simulating a run of Game~\ref{game:feasibility_game}, sampling quantities similarly as what the oracle in that game would sample for $E_1,\ldots,E_P$ and the probing subspaces. More precisely, we proceed conditionally on a run of the Game~\ref{game:feasibility_game} starting $J$ depth-$q$ periods. For instance, this can be done by simulating the game until this event happens (Part 1 of \cref{alg:strategy_for_vector_exploration}). The algorithm then continues the run of the $J$-th depth-$q$ period using the probing subspaces $V_1,\ldots,V_k$ provided by the oracle of Game~\ref{game:probing_game} (Part 2 of \cref{alg:strategy_for_vector_exploration}). These oracle subspaces and depth-$q$ probing subspaces for Game~\ref{game:feasibility_game} live in different spaces: in $E_q\subset \Rbb^d$ for $V_i^{(q)}$ and $\Rbb^{\tilde d}$ for $V_i$. Hence, we use an isometric mapping $R:\Rbb^d\to\Rbb^d$ whose image of $E_q$ is $\Rbb^{\tilde d}\otimes \{0\}^{d-\tilde d}$. Letting $\tilde R = \pi_{\tilde d}\circ R$ where $\pi_{\tilde d}$ is the projection onto the first $\tilde d$ coordinates, we map any vector $\mb x\in E_q$ to the vector $\tilde R(\mb x)$ in $\Rbb^{\tilde d}$. The natural inverse mapping $\tilde R^{<-1>}$ such that $\tilde R^{<-1>}\circ \tilde R = \proj_{E_q}$ is used to make the transfer
\begin{equation*}
    V_i^{(q)} := \tilde R^{<-1>}(V_i),\quad i\in[k].
\end{equation*}

\begin{algorithm}[ht!]
        
\caption{Strategy of the Player for the Probing Game~\ref{game:probing_game}}\label{alg:strategy_for_vector_exploration}

\setcounter{AlgoLine}{0}
\SetAlgoLined
\LinesNumbered

\everypar={\nl}

\hrule height\algoheightrule\kern3pt\relax
\KwIn{Period index $J$, Depth $q\in[p]$; dimensions $d$, $\tilde d$, $l_1,\ldots,l_P$; number of exploratory queries $k$; maximum number of queries $T_{max}$; algorithm $alg$ for Game~\ref{game:feasibility_game}}

\vspace{5pt}

{\nonl \textbf{Part 1:}} Initializing run of Procedure~\ref{proc:feasibility} conditionally on $\Ecal_q(J)$

$\mathsf{EventNotSatisfied}\gets$\textbf{true}

\While{$\mathsf{EventNotSatisfied}$}{
    With fresh randomness, sample independently $E_1,\ldots,E_P$, uniform $\tilde d$-dim. subspaces in $\Rbb^d$ and for $i\in[P-p]$ sample $k$ uniform $l_{p+i}$-dim. subspaces of $E_{p+i}$: $V_j^{(p+i)}$ for $j\in[k]$ 

    Given all previous information, set the memory of $alg$ to $M$-bit message $\mathsf{Message}$ and set $n_{p+i}\in[k]$ and vectors $\mb y_j^{(p+i)}$ for $j\in[n_{p+i}]$, for all $i\in[P-p]$; as in Game~\ref{game:feasibility_game}

    Set $n_{p'}\gets 0$ for $p'\in[p]$ and reset probing subspaces $V_j^{(p+i)}$ for $i\in[P-p]$ and $j>n_{p+i}$

    \For{$t\in[T_{max}]$}{
        Run $alg$ with current memory to get $\mb x_t$. Update exploratory queries $\mb y_i^{(p')}$ and probing subspaces $V_i^{(p')}$ for $p'\in[P],i\in[n_{p'}]$ with fresh randomness as in Procedure~\ref{proc:feasibility}, and \Return $\mb g_t=\Ocal_{\mb V}(\mb x_t)$ as response to $alg$. \lIf{$n_p$ was reset because a deeper period was completed}{\textbf{break}}

        \uIf{$n_q$ was reset for the $J$-th time}{

            Rewind the state of the Procedure~\ref{proc:feasibility} variables to the exact moment when $n_q$ was reset for the $J$-th time, in particular, before resampling $V_1^{(q)}$ 

            $\mathsf{EventNotSatisfied}\gets$\textbf{false} and $T\gets t$; \textbf{break}
        }
    }
}

\vspace{5pt} 

{\nonl \textbf{Part 2:}} Continue the $J$-th depth-$q$ period using oracle probing subspaces

Let $R:\Rbb^d\to \Rbb^d\in\Ocal_d(\Rbb)$ be an isometry such that $R(E_q) = \Rbb^{\tilde d}\otimes \{0\}^{d-\tilde d}$. Denote $\tilde R= \pi_{\tilde d}\circ R$ which only keeps the first $\tilde d$ coordinates of $R(\cdot)$, and let $\tilde R^{<-1>}:\Rbb^{\tilde d}\to\Rbb^d$ be the linear map for which $\tilde R^{<-1>}\circ \tilde R = \proj_{E_q}$

\For{$t\in \{T,\ldots,T_{max}\}$}{
    \lIf{$t\neq T$}{
        Run $alg$ with current memory to get $\mb x_t$
    }

    Update (For $t=T$, continue updating) exploratory queries $\mb y_i^{(p')}$ and probing subspaces $V_i^{(p')}$ for $p'\in[P],i\in[n_{p'}]$ as in Procedure~\ref{proc:feasibility}, and \Return $\mb g_t=\Ocal_{\mb V}(\mb x_t)$ as response to $alg$. Whenever needed to sample a new depth-$q$ probing subspace $V_i^{(q)}$ for $i\in[k]$, submit vector $\mb y=\tilde R(\mb x_t)$ to the oracle. Define $V_i^{(q)} := \tilde R^{<-1>}(V_i)$ where $V_i$ is the oracle response.  \lIf{$n_p$ was reset because a deeper period was completed}{\textbf{break}}

    \lIf{$n_q$ was reset during this iteration (depth-$q$ period completed)}{\textbf{break}}
}

\lIf{not all $k$ queries to the oracle were performed}{Query $\mb y=\mb 0$ for remaining queries}

\hrule height\algoheightrule\kern3pt\relax
\end{algorithm}

One can easily check that the constructed subspaces $V_i^{(q)}$ for $i\in[k]$ are i.i.d. uniform $l_q$-dimensional subspaces of $E_q$, which is consistent with the setup in the original Game~\ref{game:feasibility_game}: using this construction instead of resampling probing subspaces is stochastically equivalent. Note that the run of Procedure~\ref{proc:feasibility} stops either when $k$ depth-$q$ exploratory queries were found, or a period of larger depth was completed. As a result, during lines 14-18 of \cref{alg:strategy_for_vector_exploration}, one only needs to construct at most $k$ depth-$q$ probing subspaces. In summary, the algorithm never runs out of queries for Game~\ref{game:probing_game}, and queries during the run are stochastically equivalent to those from playing the initial strategy on Game~\ref{game:feasibility_game}.

As a result, we can apply \cref{thm:no_small_vectors} to the strategy from \cref{alg:strategy_for_vector_exploration}. It shows that on an event $\Ecal$ of probability $1-4de^{-l_q/16}$, the strategy loses at the Probing Game~\ref{game:probing_game} for parameter $\rho_q := \frac{1}{12}\sqrt{\frac{l_q}{\tilde d k^{\alpha}}}$. We now show that on the corresponding event $\tilde \Ecal$ (replacing the construction of the probing subspaces using oracle subspaces line 16 of \cref{alg:strategy_for_vector_exploration} by a fresh uniform sample in $l_q$-dimensional subspaces of $E_q$), the depth-$q$ period of Game~\ref{game:feasibility_game} is proper. Indeed, note that in Part 2 of \cref{alg:strategy_for_vector_exploration}, the times $t$ when one queries the oracle are exactly depth-$q$ exploration times. Let $\hat n_q$ be the total number of depth-$q$ exploratory times during the run. The queries are exactly $\mb y_i:=\tilde R(\mb y_i^{(q)})$ for $i\in[\hat n_q]$. On $\tilde \Ecal$, for any $i\in[\hat n_q]$, there are no vectors $\mb z\in Span(\mb y_j,j\in[i])$ such that $\|\mb z\|=1$ and $\|\proj_{Span(V_j,j\in[i])}(\mb z)\| \leq \rho_q.$ Applying the mapping $\tilde R^{<-1>}$ shows that for any vectors $\mb z\in Span(\proj_{E_q}(\mb y_j^{(q)}),j\in[i])$, with $\|\mb z\|=1$, we have
\begin{equation*}
    \|\proj_{Span(V_j^{(q)},j\in[i])}(\mb z)\| \geq \rho_q.
\end{equation*}
In other words,
\begin{equation}\label{eq:property_probing_game}
    \rho_q\|\mb z\| \leq  \|\proj_{Span(V_j^{(q)},j\in[i])}(\mb z)\|,\quad \mb z\in Span(\proj_{E_q}(\mb y_j^{(q)}),j\in[i]).
\end{equation}

Now consider any depth-$p'$ query $\mb x_t$ with $p'> q$, during the interval of time between the $i$-th depth-$q$ exploratory query and the following one. By definition (see Eq~\eqref{eq:def_separation_oracle}), it passed all probes $V_1^{(q)},\ldots,V_i^{(q)}$. That is,
\begin{equation}\label{eq:passed_probes}
    \|\proj_{Span(V_j^{(q)},j\in[i])}(\mb x_t)\| \leq \delta_q.
\end{equation}
We also have $\mb e^\top \mb x_t\leq -\frac{1}{2}$ which implies $\|\mb x_t\|\geq \frac{1}{2}$. Next, either $\mb x_t$ was a depth-$q$ exploratory query, in which case, $\mb x_t = \mb y_i^{(q)}$; or $\mb x_t$ does not satisfy the robustly-independent condition Eq~\eqref{eq:robustly-independent_query} (otherwise $\mb x_t$ would be a depth-$q$ exploratory query), that is
\begin{equation}\label{eq:non_robustly_indep}
    \|\proj_{Span(\mb y_j^{(q)},j\in[i])^\perp}(\mb x_t) \| < \delta_q.
\end{equation}
In both cases, Eq~\eqref{eq:non_robustly_indep} is satisfied. For convenience, denote $\mb u:=\proj_{Span(\mb y_j^{(p)},j\in[i])}(\mb x_t)$. Applying Eq~\eqref{eq:property_probing_game} with $\mb z = \proj_{E_q}(\mb u)$ yields
\begin{align*}
    \rho_q\|\proj_{E_q}(\mb u)\| &\leq \|\proj_{Span(V_j^{(q)},j\in[i])}\circ\proj_{E_q}(\mb u)\|\\
    &= \|\proj_{Span(V_j^{(q)},j\in[i])}(\mb u)\|\\
    &\leq  \|\mb u-\mb x_t\| + \|\proj_{Span(V_j^{(q)},j\in[i])}(\mb x_t)\| \\
    &\leq \delta_q + \delta_q = 2\delta_q
\end{align*}
In the last inequality, we used Eq~\eqref{eq:passed_probes} and \eqref{eq:non_robustly_indep}. As a result, on $\tilde\Ecal$, all depth-$p'$ queries $\mb x_t$ with $p' > q$ satisfy
\begin{equation*}
    \|\proj_{E_q}(\mb x_t)\| \leq \|\proj_{E_q}(\mb u)\| + \|\mb x_t-\mb u\| \leq \frac{2\delta_q}{\rho_q} +\delta_p \leq \frac{3\delta_q}{\rho_q} \leq \eta_q.
\end{equation*}
In the last inequality, we used the definition of $\delta_q$ from Eq~\eqref{eq:def_delta}. This shows that under $\tilde\Ecal$, the $J$-th depth-$q$ period was proper.
We recall that the $J$-th depth-$q$ run was generated conditionally on $\Ecal_q(J)$ (see Part 1 of \cref{alg:strategy_for_vector_exploration}). Hence, we proved the desired result
\begin{equation*}
    \Pbb\paren{ \text{$J$-th depth-$q$ period is proper} \mid \Ecal_q(J)} \geq \Pbb(\tilde \Ecal) \geq 1-4de^{-l_q/16}.
\end{equation*}
This ends the proof.
\end{proof}

\subsection{Query lower bounds for the Orthogonal Subspace Game}

We next show that during a depth-$p$ period for $p\geq 2$, provided that the player won and the depth-$(p-1)$ periods during that interval of time were \emph{proper} (which will be taken care of via \cref{lemma:most_periods_proper}), then a memory-constrained algorithm needs to have received $\Omega(\tilde d)$ vectors from the previous depth $p-1$. This uses techniques from previous papers on memory lower bounds for such orthogonal vector games, starting from \cite{marsden2022efficient}. For our purposes, we need a specific variant of these games, which we call the Orthogonal Subspace Game~\ref{game:orthogonal_subspace_game}. This simulates a generic run of Procedure~\ref{proc:feasibility} during a period at any given depth $p\geq 2$. For the sake of presentation, we recall its definition here.

\begin{game}[h!]

\caption{Orthogonal Subspace Game}
\repeatAlgoCaption{alg:orthogonal_subspace_game}{Orthogonal Subspace Game}
\setcounter{AlgoLine}{0}

\SetAlgoLined
\LinesNumbered

\everypar={\nl}

\hrule height\algoheightrule\kern3pt\relax

\KwIn{dimensions $d$, $\tilde d$; memory $M$; number of robustly-independent vectors $k$; number of queries $m$; parameters $\beta$, $\gamma$}

\textit{Oracle:} Sample a uniform $\tilde d$-dimensional subspace $E$ in $\Rbb^d$ and $\mb v_1,\ldots,\mb v_m\overset{i.i.d.}{\sim} \Ucal(S_d\cap E)$ 

\textit{Player:} Observe $E$ and $\mb v_1,\ldots,\mb v_m$, and store an $M$-bit message $\mathsf{Message}$ about these

\textit{Oracle:} Send samples $\mb v_1,\ldots,\mb v_m$ to player

\textit{Player:} Based on $\mathsf{Message}$ and $\mb v_1,\ldots,\mb v_m$ only, return unit norm vectors $\mb y_1,\ldots, \mb y_k$

The player wins if for all $i\in[k]$
\begin{enumerate}
    \item $\| \proj_E(\mb y_i)\| \leq \beta$
    \item $\|\proj_{Span(\mb y_1,\ldots,\mb y_{i-1})^\perp}(\mb y_i)\| \geq \gamma$.
\end{enumerate}
\hrule height\algoheightrule\kern3pt\relax
\end{game}

Our end goal is to prove a query lower bound $\Omega(\tilde d)$ on Game~\ref{game:orthogonal_subspace_game} for the player to succeed with reasonable probability. To do so, we simplify the game further by deleting the queries $\mb v_1,\ldots,\mb v_m$ altogether. This yields Game~\ref{game:simplified_orthogonal_subspace_game}.

\begin{game}[h!]

\caption{Simplified Orthogonal Subspace Game}\label{game:simplified_orthogonal_subspace_game}

\setcounter{AlgoLine}{0}

\SetAlgoLined
\LinesNumbered

\everypar={\nl}

\hrule height\algoheightrule\kern3pt\relax

\KwIn{dimensions $d$, $\tilde d$; memory $M$; number of vectors $k$; parameters $\beta$, $\gamma$}

\textit{Oracle:} Sample a uniform $\tilde d$-dimension linear subspace $E$ in $\Rbb^d$

\textit{Player:} Observe $E$ and store an $M$-bit message $\mathsf{Message}$ about $E$

\textit{Player:} Based on $\mathsf{Message}$ only, return unit norm vectors $\mb y_1,\ldots, \mb y_k$

The player wins if for all $i\in[k]$
\begin{enumerate}
    \item $\| \proj_E(\mb y_i)\| \leq \beta$
    \item $\|\proj_{Span(\mb y_1,\ldots,\mb y_{i-1})^\perp}(\mb y_i)\| \geq \gamma$.
\end{enumerate}
\hrule height\algoheightrule\kern3pt\relax
\end{game}

Precisely, we show that a strategy to play the Orthogonal Subspace Game~\ref{game:orthogonal_subspace_game} yields a strategy for the Simplified Orthogonal Subspace Game~\ref{game:simplified_orthogonal_subspace_game} for the new dimension $d'=d-m$. The following lemma formalizes this reduction.

\begin{lemma}\label{lemma:reduction_simplified_ortho_game}
    If there is an algorithm for the Orthogonal Subspace Game~\ref{game:orthogonal_subspace_game} with parameters $(d,\tilde d, M, k, m, \beta,\gamma)$, then there is a strategy for the Simplified Subspace Game~\ref{game:simplified_orthogonal_subspace_game} with parameters $(d,\tilde d - m,M,k,\beta,\gamma)$ that wins with at least the same probability.
\end{lemma}

\begin{proof}
    Fix a strategy for the Orthogonal Subspace Game~\ref{game:orthogonal_subspace_game}. We define in \cref{alg:strategy_simplified_game} a strategy for Game~\ref{game:simplified_orthogonal_subspace_game} for the desired parameters $(d,\tilde d - m)$.

    The intuition is the following: $E$ is sampled as a uniform $\tilde d$-dimensional subspace of $\Rbb^d$. On the other hand, if $V$ (resp. $F'$) is a uniform $m$-dimensional (resp. $(\tilde d-m)$-dimensional) subspace of $\Rbb^d$ then the subspace $V\oplus F'$ is also distributed as a uniform $\tilde d$-dimensional subspace of $\Rbb^d$. Now note that the subspace $F$ from the oracle of Game~\ref{game:simplified_orthogonal_subspace_game} is distributed exactly as a $(\tilde d-m)$-dimensional subspace of $\Rbb^d$. Therefore, we can simulate the subspace $E$ from Game~\ref{game:orthogonal_subspace_game} via $E=V\oplus F$. It remains to simulate $\mb v_1,\ldots,\mb v_m$. To do so, note that conditionally on $E:=F\oplus V$, the subspace $V$ is a uniformly random $m$-dimensional subspace of $E$. Similarly, for i.i.d. sampled vectors $\mb w_1,\ldots,\mb w_m\overset{i.i.d.}{\sim} \Ucal(S_{d-1}\cap E)$, the space $Span(\mb w_1,\ldots,\mb w_m)$ is also a uniform $m$-dimensional subspace of $E$ (on an event $\Ecal$ of probability one). Last, for any $m$-dimensional subspace $V$ , denote by $\Dcal(V)$ the conditional distribution of $(\mb w_1,\ldots,\mb w_m)$ conditionally on $Span(\mb w_i,i\in[m])=V$. (Note that $\Dcal(V)$ does not correspond to $\Ucal(S_d\cap V)^{\otimes m}$ because the vectors $\mb v_i$ will tend to be more ``orthogonal'' since they were initially sampled in a $\tilde d$-dimensional subspace $E$ while $V$ has only dimension $m$). As a summary of the previous discussion, by sampling vectors $(\mb v_1,\ldots,\mb v_m)\sim \Dcal(V)$, conditionally on $E=V\oplus F$, these are distributed exactly as i.i.d. uniform $\Ucal(S_{d-1}\cap E)$ samples. We can then use the following procedure to sample $E$ and $\mb v_1,\ldots,\mb v_m$:
    \begin{enumerate}
        \item Let $F$ be the $(\tilde d-m)$-dimensional subspace provided by the oracle of Game~\ref{game:simplified_orthogonal_subspace_game}.
        \item Sample a uniform $m$-dimensional subspace $V$ of $\Rbb^d$. Sample $(\mb v_1,\ldots,\mb v_m)\sim \Dcal(V)$ and define $E = V\oplus F$.
    \end{enumerate}
    The procedure is stochastically equivalent to the setup line 1 of Game~\ref{game:orthogonal_subspace_game}. The complete strategy for the simplified Game~\ref{game:simplified_orthogonal_subspace_game} is given in \cref{alg:strategy_simplified_game}.

    \begin{algorithm}[ht!]
        
\caption{Strategy for the Simplified Game~\ref{game:simplified_orthogonal_subspace_game} given a strategy for Game~\ref{game:orthogonal_subspace_game}}\label{alg:strategy_simplified_game}

\setcounter{AlgoLine}{0}
\SetAlgoLined
\LinesNumbered

\everypar={\nl}

\hrule height\algoheightrule\kern3pt\relax
\KwIn{dimensions $d$, $\tilde d$; memory $M$; number of robustly-independent vectors $k$; number of queries $m$; strategy for Game~\ref{game:orthogonal_subspace_game}}

\vspace{5pt}

{\nonl\textbf{Part 1}: Construct $\mathsf{Message}$}

Receive $F$ a $(\tilde d-m)$-dimensional subspace provided by the oracle

Sample a uniform $m$-dimensional subspace $V$ of $\Rbb^d$. Sample $(\mb v_1,\ldots,\mb v_m)\sim \Dcal(V)$ and define $E = V\oplus F$. Store message $\mathsf{Message}$ given $E$ and $\mb v_1,\ldots,\mb v_m$ as in Game~\ref{game:orthogonal_subspace_game}

\vspace{5pt}

{\nonl\textbf{Part 2}: Output solution vectors}

Observe $\mathsf{Message}$ and resample $V$, $\mb v_1,\ldots,\mb v_m$ using the same randomness as in Part 1

\Return $\mb y_1,\ldots,\mb y_k$, the same outputs given by the strategy for Game~\ref{game:orthogonal_subspace_game}

\hrule height\algoheightrule\kern3pt\relax
\end{algorithm}

    Now suppose that the player won at Game~\ref{game:orthogonal_subspace_game}. Then, the outputs $\mb y_1,\ldots,\mb y_k$ are normalized and satisfy the desired robust-independence property. Further, for all $i\in[k]$, we have
    \begin{equation*}
        \|\proj_{F}(\mb y_i)\| \leq \|\proj_E(\mb y_i)\|\leq \beta.
    \end{equation*}
    Hence, \cref{alg:strategy_simplified_game} wins at Game~\ref{game:simplified_orthogonal_subspace_game} on the same event.
\end{proof}

In Game~\ref{game:simplified_orthogonal_subspace_game}, provided that the message does not contain enough information to store the outputs $\mb y_1,\ldots,\mb y_k$ directly, we will show that the player cannot win with significant probability. This should not be too surprising at this point, because the message is all the player has access to output vectors that are roughly orthogonal to the complete space $E$.

\begin{lemma}\label{lemma:query_lower_bound_simplified_game}
    Let $d\geq 8$, $k\leq\frac{\tilde d}{2}$ and $0<\beta,\gamma\leq 1$ such that $\gamma/\beta \geq 3e  \sqrt{kd/\tilde d}$. Then, the success probability of a player for Game~\ref{game:simplified_orthogonal_subspace_game} for memory $M$ satisfies
    \begin{equation*}
        \Pbb(\text{player wins}) \leq 25\cdot \frac{M+2}{\tilde dk} \ln\frac{\sqrt d}{\gamma}.
    \end{equation*}
\end{lemma}

For this, we use a lemma that constructs an orthonormal sequence of vectors from robustly-independent vectors. Versions of this property were already observed in \cite[Lemma 34]{marsden2022efficient}, \cite[Lemma 22]{blanchard2023quadratic}. To ontain query lower bounds that reach the query complexity of gradient descent $1/\epsilon^2$, we need the tightest form of that result. The proof is included in appendix.

\begin{lemma}
\label{lemma:gram-schmidt_marsden}
    Let $\delta\in(0,1]$ and $\mb y_1,\ldots,\mb y_r\in \Rbb^d$ some $r\leq d$ unit norm vectors. Suppose that for any $i\leq k$,
    \begin{equation*}
        \|P_{Span(\mb y_j,j<i)^\perp}(\mb y_i)\|\geq \delta.
    \end{equation*}
    Let $\mb Y=[\mb y_1,\ldots,\mb y_r]$ and $s\geq 2$. There exists $\lceil r/s\rceil$ orthonormal vectors $\mb Z=[\mb z_1,\ldots,\mb z_{\lceil r/s \rceil}]$ such that for any $\mb a\in\Rbb^d$,
    \begin{equation*}
        \|\mb Z^\top \mb a\|_\infty\leq  \paren{\frac{\sqrt r}{\delta}}^{s/(s-1)}\|\mb Y^\top\mb a\|_\infty.
    \end{equation*}
    Further, these can be constructed as the singular vectors of the singular value decomposition of $\mb Y$ associated with the $\lceil r/s\rceil$ largest singular values.
\end{lemma}

We are now ready to prove \cref{lemma:query_lower_bound_simplified_game}.

\begin{proof}[of \cref{lemma:query_lower_bound_simplified_game}]
    Fix a strategy for Game~\ref{game:simplified_orthogonal_subspace_game} and $s= 1 + \ln \frac{\sqrt d}{\gamma}$. For simplicity, without loss of generality assume that the message $\mathsf{Message}$ is deterministic in $E$ (by the law of total probability, there is a choice of internal randomness such that running the strategy with that randomness yields at least the same probability of success). Now denote by $\Ecal$ the event when the player wins and let $\mb Y=[\mb y_1,\ldots,\mb y_k]$ be the concatenation of the vectors output by the player. By \cref{lemma:gram-schmidt_marsden}, we can construct an orthonormal sequence $\mb Z = [\mb z_1,\ldots,\mb z_r]$ with $r=\ceil{k/s}$ such that on the event $\Ecal$, for all $i\in[r]$ and $\mb a\in E$ with unit norm,
    \begin{equation*}
        \|\mb z_i^\top \mb a\| \leq \paren{\frac{\sqrt k}{\gamma}}^{s/(s-1)} \|\mb Y^\top\mb a\|_\infty \leq  \paren{\frac{\sqrt k}{\gamma}}^{s/(s-1)} \max_{j\in[k]} \|\proj_E(\mb y_j)\| \leq \beta  \paren{\frac{\sqrt k}{\gamma}}^{1+ \frac{1}{s-1}} \leq \frac{e\beta\sqrt k}{\gamma}.
    \end{equation*}
    In other words, we have for all $i\in[r]$,
    \begin{equation}\label{eq:small_proj}
        \|\proj_E(\mb z_i)\| \leq \frac{e\beta\sqrt k}{\gamma} \leq \frac{1}{3} \sqrt{\frac{\tilde d}{d}}.
    \end{equation}

    We now give both an upper and lower bound on $I(E,\mb Y)$ -- this will lead to the result. First, because $\mb Y$ is constructed from $\mathsf{Message}$, the data processing inequality gives
    \begin{equation}\label{eq:upper_bound_information}
        I(E;\mb Y) \leq I(E;\mathsf{Message}) \leq H(\mathsf{Message}) \leq M\ln 2.
    \end{equation}
    To avoid continuous/discrete issues with the mutual information, we provide the formal justification of the computation above. For any $\Mcal\in\{0,1\}^M$, let $\Scal(\Mcal) = \{E:\mathsf{Message}(E)=\Mcal\} $ the set of subspaces that lead to message $\Mcal$. This is well defined because we supposed $\mathsf{Message}$ to be deterministic in $E$. Because $\mb Y$ depends only on $\mathsf{Message}$, we also define $p_{\mb Y,\Mcal}$ the probability mass function for $\mb Y$ given message $\Mcal$, and $p_M(\Mcal)$ the probability to have $\mathsf{Message}=\Mcal$. Then,
    \begin{align*}
        I(E;\mb Y) &:= \int_e\int_{\mb y} p_{E, \mb Y}(e, \mb y) \ln\frac{p_{E, \mb Y}(e, \mb y)}{p_E(e) p_{\mb Y}(\mb y)}de d\mb y\\
        &=\sum_{\Mcal\in\{0,1\}^M} \int_{e\in\Scal(\Mcal)}\int_{\mb y} p_E(e) p_{\mb Y,\Mcal}(\mb y) \ln\frac{p_{\mb Y,\Mcal}(\mb y) }{p_{\mb Y}(\mb y)}de d\mb y\\
        &= \sum_{\Mcal\in\{0,1\}^M} p_M(\Mcal)\int_{\mb y} p_{\mb Y,\Mcal}(\mb y) \ln\frac{p_{\mb Y,\Mcal}(\mb y) }{p_{\mb Y}(\mb y)} d\mb y\\
        &= \int_{\mb y} d\mb y \sum_{\Mcal\in\{0,1\}^M} p_M(\Mcal) p_{\mb Y,\Mcal}(\mb y) \ln\frac{p_{\mb Y,\Mcal}(\mb y) }{\sum_{\Mcal'\in\{0,1\}^M} p_M(\Mcal')p_{\mb Y,\Mcal'}(\mb y)} \\
        &\leq \int_{\mb y} d\mb y \sum_{\Mcal\in\{0,1\}^M} p_M(\Mcal)p_{\mb Y,\Mcal}(\mb y) \ln\frac{1}{p_M(\Mcal)} =H(\Mcal)\leq M \ln 2.
    \end{align*}
    In the last inequality, we used the standard inequality $H(X)\geq 0$ where $X$ is the discrete random variable with $p_X(\Mcal)\propto p_M(\Mcal)p_{\mb Y,\Mcal}(\mb y)$.

    We now turn to the lower bound. Recall that from \cref{lemma:gram-schmidt_marsden}, the vectors $\mb z_1,\ldots,\mb z_r$ are constructed explicitely from $\mb Y=[\mb y_1,\ldots,\mb y_k]$ as the vectors from the $r$ largest singular values for the singular value decomposition of $\mb Y$. As a result, by the data processing inequality, we have
    \begin{equation*}
        I(E;\mb Y) = I(E;\mb Y,\mb Z) \geq I(E;\mb Z).
    \end{equation*}
    By the chain rule,
    \begin{align*}
        I(E;\mb Z)  &= I(E;\mb Z,\1[\Ecal]) + I(E;\1[\Ecal]) - I(E;\1[\Ecal]\mid \mb Z)\\
        &\geq I(E;\mb Z\mid \1[\Ecal]) - H(\1[\Ecal])\\
        &\geq \Pbb(\Ecal)\Ebb_{\Ecal}\sqb{I(E;\mb Z \mid \Ecal)} - \ln 2.
    \end{align*}
    Now fix a possible realization for $\mb Z$. We recall that provided that the player won ($\Ecal$ is satisfied), the matrix $\mb Z=[\mb z_1,\ldots,z_r]$ satisfies Eq~\eqref{eq:small_proj}. Let $\Ccal$ denote the set of all subspaces $E$ compatible with these:
    \begin{equation*}
        \Ccal:=\Ccal(\mb Z) = \set{\tilde d\text{-dimensional subspace $F$ of }\Rbb^d: \|\proj_F(\mb z_i)\|\leq \frac{1}{3}\sqrt{\frac{\tilde d}{d}},i\in[r] }.
    \end{equation*}
    This set is measurable as the intersection of measurable sets. By the data processing inequality,
    \begin{align*}
        I(E;\mb Z \mid \Ecal) \geq I(E;\Ccal\mid \Ecal) = \Ebb_{\Ccal\mid \Ecal}\sqb{ D(p_{E\mid \Ccal,\Ecal}\parallel p_{E\mid \Ecal})  } 
    \end{align*}
    Note that $p_{E\mid\Ecal} = p_{E,\Ecal}/\Pbb(\Ecal) \leq p_E/\Pbb(\Ecal)$.
    Since $\Ecal$ is satisfied, the support of $p_{E\mid \Ccal,\Ecal}$ is included in $\Ccal$. Now let $F$ be a uniformly sampled subspace of $\Ccal$. We obtain
    \begin{equation}\label{eq:dk_lower_bound}
        D(p_{E\mid \Ccal,\Ecal}\parallel p_{E\mid \Ecal}) = \Ebb_{E\mid \Ccal ,\Ecal}\sqb{\ln\frac{p_{E\mid \Ccal,\Ecal}}{p_{E\mid \Ecal}}} \geq \Ebb_{E\mid \Ccal ,\Ecal}\sqb{\ln\frac{\Pbb(\Ecal) p_{F\mid \Ccal}}{p_{E} }} = \ln\frac{1}{\Pbb(E\in\Ccal)}+ \ln \Pbb(\Ecal). 
    \end{equation}
    
    The last step of the proof is to upper bound this probability $\Pbb(E\in \Ccal)$ where here $\Ccal$ and $\mb Z$ were fixed. Without loss of generality (since the distribution of $E$ is rotation-invariant), we can assume that $\mb z_i = \mb e_i$ for $i\in[r]$, the first $r$ vectors of the natural basis of $\Rbb^d$. Equivalently, we can consider the setup where $E=\Rbb^{\tilde d}\otimes\{0\}^{d-\tilde d}$ and we sample a random orthonormal sequence $\mb Z$ of $r$ vectors uniformly in $\Rbb^d$, which does not affect the quantity
    \begin{equation*}
        \Pbb(E\in\Ccal) = \Pbb\paren{\|\proj_E(\mb z_i)\|\leq \frac{1}{3}\sqrt{\frac{\tilde d}{ d}},\;i\in[r]}.
    \end{equation*}
    We take this perspective from now. For $i\in[r]$, we introduce $G_i = Span(\mb z_j,\proj_E(\mb z_j),j\in[i])$. We then define
    \begin{equation*}
        F_i = E \cap G_{i-1}^\perp \quad \text{ and }\quad \mb a_i = \proj_{G_{i-1}^\perp}(\mb z_i).
    \end{equation*}
    In particular, we have $E=Span(\proj_E(\mb z_j),j<i) \oplus F_i$.
    Recall that because $\mb z_1,\ldots,\mb z_r$ was sampled as a uniformly rotated orthonormal sequence, conditionally on $\mb z_j$ for $j<i$ (and also on $\proj_E(\mb z_j)$ for $j<i$, which do not bring further information on $\mb z_i$), the variable $\mb z_i$ is exactly distributed as a random uniform unit vector in $Span(\mb z_j,j<i)^\perp$. Within this space is included $F_i\subset G_{i-1}^\perp$, hence we can apply \cref{lemma:concentration_projection} to obtain
    \begin{equation*}
        \Pbb\paren{ \|\proj_{F_i}(\mb z_i)\| \leq \sqrt{\frac{\dim(F_i)}{d-i+1} \paren{1-\frac{1}{\sqrt 2}}} \mid \mb z_j,\proj_E(\mb z_j),j<i} \leq e^{-\dim(F_i)/8}.
    \end{equation*}
    Because $r\leq k\leq \frac{\tilde d}{4}$, we have $\dim(F_i) \geq \dim(E) - \dim(G_{i-1}) \geq \tilde d-2(r-1) \geq \frac{\tilde d}{2}$. Hence, the previous equation implies
    \begin{equation*}
        \Pbb\paren{ \|\mb a_i\| \leq \frac{1}{3}\sqrt{\frac{\tilde d}{d}} \mid \mb z_j,\proj_E(\mb z_j),j<i} \leq e^{-\tilde d/16}.
    \end{equation*}
    Combining this equation together with the fact that
    \begin{equation*}
        \|\proj_E(\mb z_i)\| \geq \|\proj_{F_i}(\mb z_i)\| = \|\proj_{F_i}(\mb a_i)\|,
    \end{equation*}
    we obtained
    \begin{equation*}
        \Pbb\paren{\|\proj_E(\mb z_i)\| \leq \frac{1}{3}\sqrt{\frac{\tilde d}{d}} \mid \mb z_j,\proj_E(\mb z_j),j<i} \leq e^{-\tilde d/16}.
    \end{equation*}
    Because this holds for all $i\in[r]$, this implies that
    \begin{equation*}
        \Pbb(E\in\Ccal(\mb Z)) \leq e^{-\tilde d r/16}.
    \end{equation*}

    We can then plug this bound into Eq~\eqref{eq:dk_lower_bound}. This finally yields
    \begin{equation*}
        M\ln 2\geq I(E;\mb Y) \geq \Pbb(\Ecal) \frac{\tilde d r}{16} - \Pbb(\Ecal)\ln\frac{1}{\Pbb(\Ecal)} -\ln 2 \geq \Pbb(\Ecal)\frac{\tilde d k}{16s} -\ln 2-\frac{1}{e}.
    \end{equation*}
    In the left inequality, we recalled the information upper bound from Eq~\eqref{eq:upper_bound_information}. Because we assumed $d \geq 8$, we have $s \leq 2\ln\frac{\sqrt d}{\gamma}$. Rearranging and simplifying ends the proof.
\end{proof}

As a result, combining the reduction from the Orthogonal Subspace Game~\ref{game:orthogonal_subspace_game} to the simplified Game~\ref{game:simplified_orthogonal_subspace_game} from \cref{lemma:reduction_simplified_ortho_game}, together with the query lower bound in \cref{lemma:query_lower_bound_simplified_game}, we obtain the desired query lower bound for Game~\ref{game:orthogonal_subspace_game}.

\begin{proof}[of \cref{thm:memory_query_lower_bound}]
    The proof essentially consists of putting together \cref{lemma:reduction_simplified_ortho_game,lemma:query_lower_bound_simplified_game}. Suppose that the player uses at most $m< \frac{\tilde d}{2}$ queries. By \cref{lemma:reduction_simplified_ortho_game}, we can use this strategy to solve Game~\ref{game:simplified_orthogonal_subspace_game}, where the dimension of the subspace $E$ is $\dim(E) = \tilde d-m > \frac{\tilde d}{2}$. Using this bound, we can check that the parameters satisfy the conditions from \cref{lemma:query_lower_bound_simplified_game}, that is
    \begin{equation*}
        k\leq \frac{\dim(E)}{2} \quad \text{and}\quad \frac{\gamma}{\beta} \geq 3e \sqrt{\frac{kd}{\dim(E)}}.
    \end{equation*}
    As a result, we have
    \begin{equation*}
        \Pbb(\text{player wins}) \leq 25\frac{M+2}{\dim(E) k} \ln\frac{\sqrt d}{\gamma} < 50\frac{M+2}{\tilde d k} \ln\frac{\sqrt d}{\gamma} \leq \frac{1}{C}.
    \end{equation*}
    This gives a contradiction and ends the proof.
\end{proof}

\subsection{Recursive query lower bound argument for the feasibility game}

It remains to relate the Orthogonal Subspace Game~\ref{game:orthogonal_subspace_game} to the feasibility Game~\ref{game:feasibility_game} to obtain query lower bounds for the latter using \cref{thm:memory_query_lower_bound}. We briefly give some intuition as to how this reduction works. Intuitively, during a run of a period of depth $p$ from Game~\ref{game:feasibility_game} the algorithm needs to find $k$ exploratory queries that are by definition robustly-independent (Eq~\eqref{eq:robustly-independent_query}). Using \cref{lemma:most_periods_proper} we also show that most of the periods at depth $p-1$ are proper, hence the exploratory queries also need to be roughly orthogonal to $E_{p-1}$. We can therefore emulate a run of Game~\ref{game:orthogonal_subspace_game} by taking $E_{p-1}=E$ to be the hidden random subspace. This gives the following result.

\begin{lemma}\label{lemma:many_depth_p-1}
Let $p\in\{2,\ldots,P\}$. Suppose that $\frac{\tilde d}{4l} \geq k \geq 50\cdot(4P) \frac{M + \frac{d}{8} \log_2(T_{max})+3}{\tilde d} \ln \frac{\sqrt d}{\delta_p}$. If there exists a strategy for Game~\ref{game:feasibility_game} for depth $p$ that wins with probability at least $q$, then during a run of the strategy,
\begin{equation*}
    \Pbb\paren{\text{at least } \frac{\tilde d}{4lk} \text{ periods of depth }p-1\text{ are completed}} \geq q-\frac{3}{8P}.
\end{equation*}
\end{lemma}

\begin{proof}
    Fix $p\in\{2,\ldots,P\}$ and a strategy for the Depth-$p$ Game~\ref{game:feasibility_game}. We construct in \cref{alg:strategy_orthogonal_subspace_game} a strategy for Game~\ref{game:orthogonal_subspace_game} for $m=\ceil{\tilde d/2}-1$. It simulates a run of Game~\ref{game:feasibility_game} by sampling subspaces $E_{p'}$ for $p'\neq p-1$ and using for $E_{p-1}$ the subspace $E$ sampled by the oracle. The message $\mathsf{Message}$, constructed in Part 1 on \cref{alg:strategy_orthogonal_subspace_game}, contains the initialization for the memory of the underlying algorithm $alg$, as well as indications of when are the exploratory queries for periods of depth $p'>p$. In Part 2 of \cref{alg:strategy_orthogonal_subspace_game}, the run of Game~\ref{game:feasibility_game} is simulated once again, but without having direct access to $E$. Fortunately, to compute the feasibility separation oracle from Game~\ref{game:feasibility_game} (or Procedure~\ref{proc:feasibility}), one only needs to:
    \begin{enumerate}
        \item Construct uniformly sampled subspaces $V_i^{(p-1)}$ of $E=E_{p-1}$. This can be done directly thanks to the vectors $\mb v_1,\ldots,\mb v_m$ provided by the oracle in line 3 of Game~\ref{game:orthogonal_subspace_game}. Indeed, for any vectors $\mb z_1,\ldots,\mb z_l\overset{i.i.d.}{\sim} \Ucal(S_d\cap E)$, the distribution of $Span(\mb z_1,\ldots,\mb z_l)$ is the same as for a uniformly sampled $l$-dimensional subspace of $E_{p-1}=E$. We therefore use $l$ new vectors within the list $\mb v_1,\ldots,\mb v_m$ whenever a new probing subspace of $E_{p-1}$ is needed. (Recall that $l_{p-1}=l$ since $p-1<P$.)
        \item Know when queries are exploratory queries. This is important to update the number of exploratory queries $n_{p'}$ for $p'\in[P]$ which dictates the number of probing subspaces needed. For $p'\in[p]$, this can be done directly since all depth-$p'$ periods start with no exploratory queries ($n_{p'}\gets 0$ in line 3 of Game~\ref{game:feasibility_game}). Hence all previous depth-$p'$ exploratory queries are queried during the run Part 2, and we can test for robust-independence (Eq~\eqref{eq:robustly-independent_query}) directly.
        This is more problematic for depths $p'>p$ because these depend on the vectors $\mb y_j^{(p')}$ for $j\in[n_{p'}]$ defined in line 2 of Game~\ref{game:feasibility_game} with knowledge of $E_{p-1}=E$. These contain too many bits to be included in $\mathsf{Message}$. Fortunately, we only need to store the times of these depth-$p'$ exploratory queries, which sidesteps checking for robust-independence (Eq~\eqref{eq:robustly-independent_query}). To know when these times occur, we need to simulate the complete Game~\ref{game:feasibility_game} in Part 1, lines 3-11 of \cref{alg:strategy_orthogonal_subspace_game}, also using the oracle samples $\mb v_1,\ldots,\mb v_m$ that would be used in Part 2 to construct depth-$(p-1)$ probing subspaces. 
    \end{enumerate}

        \begin{algorithm}[ht!]
        
\caption{Strategy of the Player for the Orthogonal Subspace Game~\ref{game:orthogonal_subspace_game}}\label{alg:strategy_orthogonal_subspace_game}

\setcounter{AlgoLine}{0}
\SetAlgoLined
\LinesNumbered

\everypar={\nl}

\hrule height\algoheightrule\kern3pt\relax
\KwIn{depth $p$, dimensions $d$, $\tilde d$, number of vectors $k$, $M$-bit algorithm $alg$ for Game~\ref{game:feasibility_game} at depth $p$; $T_{max}$}

\vspace{5pt}

{\nonl \textbf{Part 1:}} Constructing the message\,

Sample independently $E_{p'}$ for $p'\in[P]\setminus\{p-1\}$, uniform $\tilde d$-dimensional linear subspaces in $\Rbb^d$ and for $i\in[P-p]$ sample $k$ uniform $l$-dimensional subspaces of $E_{p+i}$: $V_j^{(p+i)}$ for $j\in[k]$ 

Observe $E$ and set $E_{p-1}=E$. Given all previous information, set the memory of $alg$ to $M$-bit message $\mathsf{Memory}$ and set $n_{p+i}\in[k]$ and vectors $\mb y_j^{(p+i)}$ for $j\in[n_{p+i}]$, for all $i\in[P-p]$; as in Game~\ref{game:feasibility_game}

Set $n_{p'}\gets 0$ for $p'\in[p]$

Receive samples $\mb v_1,\dots, \mb v_m$, and set sample index $i\gets 1$

\For{$t\in[T_{max}]$}{
    Run $alg$ with current memory to get $\mb x_t$. Update exploratory queries $\mb y_i^{(p')}$ and probing subspaces $V_i^{(p')}$ for $p'\in[P],i\in[n_{p'}]$ as in Procedure~\ref{proc:feasibility}. If $n_p$ was reset because a deeper period was completed, strategy fails: \textbf{end} procedure. Whenever needed to sample a depth-$(p-1)$ probing subspace $V_{n_{p-1}}^{(p-1)}$ of $E_{p-1}$ (lines 9 or 12 of Procedure~\ref{proc:feasibility}):
    
    \lIf{$i+l-1>m$}{Strategy fails: \textbf{end} procedure}
    \lElse{use oracle samples, set $V_{n_{p-1}}^{(p-1)}:=Span(\mb v_i,\ldots,\mb v_{i+l-1})$ and $i\gets i+l$}

    \Return $\mb g_t=\Ocal_{\mb V}(\mb x_t)$ as response to $alg$
    
    \lIf{$n_p=k$}{\textbf{break}}
}
For $i\in[P-p]$ denote $t_1^{(p+i)},\ldots,t_{n_{p+i}}^{(p+i)}$ the times of depth-$(p+i)$ exploratory queries. If these were done before $t=1$ set them to $0$. Store message $\mathsf{Message}=(\mathsf{Memory};t_j^{(p+i)}, i\in[P-p],j\in[k])$ (let $t_j^{(p+i)}=-1$ if $j>n_{p+i}$)

\vspace{5pt}

{\nonl \textbf{Part 2:}} Simulate run of Game~\ref{game:feasibility_game}

Receive samples $\mb v_1,\ldots,\mb v_m$ from Oracle

Resample $E_{p'}$ for $p'\in[P]\setminus\{p-1\}$ using same randomness as in Part 1. Initialize memory of $alg$ to $\mathsf{Memory}$. Set $n_{p'}\gets 0$ for $p'\in[p]$, and sample index $i\gets 1$

\For{$t \in[T_{max}]$}{
    Run $alg$ with current memory to get $\mb x_t$. Update exploratory queries and probing subspaces exactly as in line 6, with the same randomness as in Part 1 for sampling probing subspaces. To know whether $\mb x_t$ is a depth-$(p+i)$ exploratory query for $i\in[P-p]$ (line 6 of Procedure~\ref{proc:feasibility}), check whether $t$ is within the message times $t_j^{(p+i)}$ for $j\in[n_{p+i}]$ 
    
    Return $\mb g_t = \Ocal_{\mb V}(\mb x_t)$ as response to $alg$. \lIf{$n_p=k$}{\textbf{break}}
}

\Return normalized depth-$p$ exploratory queries $\frac{\mb y_1^{(p)}}{\|\mb y_1^{(p)}\|},\ldots,\frac{\mb y_k^{(p)} }{\|\mb y_k^{(p)}\|}$

\hrule height\algoheightrule\kern3pt\relax
\end{algorithm}

    We then define the message as $\mathsf{Message}=(\mathsf{Memory};t_j^{(p+i)}, i\in[P-p],j\in[k])$ where $t_j^{(p+i)}$ is the time of the $j$-th exploratory query of depth $p+i$, with the convention $t_j^{(p+i)}=-1$ if there were no $j$ depth-$(p+i)$ exploratory queries (line 12 of \cref{alg:strategy_orthogonal_subspace_game}). With this convention, we can also know exactly what was $n_{p+i}$ for $i\in[P-p]$ from the message, and it can be stored with a number of bits of
    \begin{equation*}
        M+\ceil{(P-p)k \log_2(T_{max}+2)} \leq M + \frac{d}{8} \log_2(T_{max}+2)+1.
    \end{equation*}
    Here we used $Pk \leq P\tilde d \leq \frac{d}{2}$.
    In Part 2, after simulating the run from Game~\ref{game:feasibility_game}, the strategy returns the (normalized) depth-$p$ exploratory queries to the oracle (line 19).

    We now show that this strategy wins with significant probability. Note that a run of Game~\ref{game:feasibility_game} ends whenever the depth-$p$ period is finished. In particular, this ensures that there is no overwriting of the depth-$p$ exploratory queries, hence, there is no ambiguity when referring to some exploratory query $\mb y_i^{(p)}$. We next let $\Ecal$ be the event when the strategy for Game~\ref{game:feasibility_game} wins. Because $E$ is also sampled uniformly as a $\tilde d$-dimensional subspace by the oracle, under the corresponding event $\tilde\Ecal$ (only changing the dependency in $E_{p-1}$ by the dependency in $E$), the strategy succeeds, that is, the algorithm makes $k$ depth-$p$ exploratory queries.
    
    We apply \cref{lemma:most_periods_proper} checking that the assumptions are satisfied: $l\geq C_\alpha \ln n$ from Eq~\eqref{eq:def_l} and $4lk \leq \tilde d$, where $\tilde d$ is the dimension of the problem for Game~\ref{game:probing_game} here. Hence, by the union bound, on an event $\Fcal$ of probability at least $1-4d(k)de^{-l/16}$, the first $d(k):=\floor {\frac{\tilde d}{k}}$ periods of depth $p-1$ are proper. Indeed,
    \begin{align*}
        \Pbb(\Fcal^c) &=\Pbb (\exists j\in[d(k)]:j\text{-th depth-$(p-1)$ period was started and is not proper})\\
        &\leq \sum_{j\in[d(k)]}\Pbb(\Ecal_{p-1}(j) \cap \{\text{$j$-th depth-$(p-1)$ period is not proper} \}\\
        &\leq \sum_{j\in[d(k)]} 4de^{-l/16}\Pbb(\Ecal_{p-1}(j)) \leq 4d^2e^{-l/16}.
    \end{align*}
    Last, let $\Gcal$ be the event that there are at most $\frac{m}{lk}$ periods of depth $p-1$. On each period of depth $p-1$, \cref{alg:strategy_orthogonal_subspace_game} only uses at most $lk$ samples $\mb v_i$ from the oracle: $k$ for each probing subspace $V_i^{(p-1)}$. We note that because the algorithm stops as soon as a depth-$p'$ period for $p'\geq p$ ends, the probing subspaces $V_i^{(p-1)}$ are never reset because deeper periods ended (see line 12 of Procedure~\ref{proc:feasibility}). Thus, on $\Gcal$, \cref{alg:strategy_orthogonal_subspace_game} does not run out of oracle samples.

    On $\Ecal\cap \Fcal\cap \Gcal$, all $k$ depth-$p$ exploratory queries were made during proper periods of depth $p-1$. Here we used the fact that on $\Gcal$, there are at most $\frac{m}{lk}\leq d(k)$ periods of depth $p-1$. By definition of proper periods (\cref{def:proper_period}), we have
    \begin{equation*}
        \|\proj_{E_{p-1}}(\mb y_i^{(p)})\| \leq \eta_{p-1},\quad i\in[k].
    \end{equation*}
    Now recall that exploratory queries also satisfy $\mb e^\top \mb y_i^{(p)} \leq -\frac{1}{2}$, so that $\frac{1}{2}\leq \|\mb y_i^{(p)}\|\leq 1$ for $i\in[k]$. As a result, the output normalized vectors $\mb u_i=\mb y_i^{(p)} / \|\mb y_i^{(p)}\|$ for $i\in[k]$ satisfy
    \begin{equation*}
        \|\proj_{E_{p-1}}(\mb u_i^{(p)})\| \leq 2\|\proj_{E_{p-1}}(\mb y_i^{(p)})\| \leq 2\eta_{p-1},\quad i\in[k].
    \end{equation*}
    Also, by construction exploratory queries are robustly-independent (Eq~\ref{eq:robustly-independent_query}). Hence,
    \begin{equation*}
        \|\proj_{Span(\mb u_j^{(p)},j<i)^\perp}(\mb u_i^{(p)})\| \geq \|\proj_{Span(\mb y_j^{(p)},j<i)^\perp}(\mb y_i^{(p)})\| \geq \delta_p,\quad j\in[k].
    \end{equation*}
    This shows that on $\Ecal\cap\Fcal\cap\Gcal$, the algorithm wins at Game~\ref{game:orthogonal_subspace_game} with memory at most $M + \frac{d}{8} \log_2(T_{max}+2)+1$, using $m<\tilde d/2$ queries and for the parameters $(\beta,\gamma) = (2\eta_{p-1},\delta_p)$. We now check that the assumptions for applying \cref{thm:memory_query_lower_bound} are satisfied. The identity $d\geq 8P$ is satisfied by assumption throughout the paper. The only assumption that needs to be checked is that for $\gamma/\beta$. Using the notation $\mu_p=\mu$ for all $p\in[P-1]$, we now note that
    \begin{equation*}
        \frac{\gamma}{\beta} = \frac{\delta_p}{2\eta_{p-1}} = \frac{\mu_p}{72}\sqrt{\frac{l_p}{\tilde d k^{\alpha}}} = \frac{600}{72} \sqrt{\frac{kd}{\tilde d}} \geq 12\sqrt{\frac{k d}{\tilde d}}.
    \end{equation*}
    Here we used the definitions Eq~\eqref{eq:def_eta} and \eqref{eq:def_delta}. The lower bound $k$ from the hypothesis is exactly the bound needed to apply \cref{thm:memory_query_lower_bound} with the probability $1/(4P)$. Precisely, we have
    \begin{equation*}
        \Pbb(\Ecal\cap\Fcal\cap \Gcal) \leq \Pbb (\text{\cref{alg:strategy_orthogonal_subspace_game} wins at Game}~\ref{game:orthogonal_subspace_game} ) \leq \frac{1}{4P}.
    \end{equation*}
    Combining the previous statements shows that
    \begin{align*}
        \Pbb\paren{\text{more than }\frac{m}{lk} \text{ periods of depth }p-1} =\Pbb(\Gcal^c)&\geq \Pbb(\Ecal\cap\Fcal\cap\Gcal^c)=\Pbb(\Ecal\cap\Fcal) - \Pbb(\Ecal\cap\Fcal\cap\Gcal)\\
        &\geq \Pbb(\Ecal) - \Pbb(\Fcal^c) - \Pbb(\Ecal\cap\Fcal\cap\Gcal)\\
        &\geq q-\frac{1}{4P} - 4d^2e^{-l/16} .
    \end{align*}
    Because we also have $l\geq 16\ln (32d^2P)$ from Eq~\eqref{eq:def_l}, this shows that
    \begin{equation*}
        \Pbb\paren{\text{more than }\frac{m}{lk} \text{ periods of depth }p-1} \geq q-\frac{3}{8P}.
    \end{equation*}
    In this event there are at least $\frac{\tilde d/2}{lk}$ periods of depth $p-1$, hence at least $\frac{\tilde d}{2lk}-1 \geq \frac{\tilde d}{4lk}$ are complete.    
\end{proof}

We are now ready to state the main recursion lemma, which enables us to construct an algorithm for Game~\ref{game:feasibility_game} at depth $p-1$ from an algorithm for depth $p$.

\begin{lemma}\label{lemma:recursion}
    Let $p\in\{2,\ldots,P\}$ and suppose that the assumptions on $k$ from \cref{lemma:many_depth_p-1} are satisfied. Suppose that there is a strategy for Game~\ref{game:feasibility_game} for depth $p$ with $T_{max}^{(p)}$ queries, that uses $M$ bits of memory and that wins with probability at least $q\in[0,1]$.
    Then, there is a strategy for Game~\ref{game:feasibility_game} for depth $p-1$ with
    \begin{equation*}
        T_{max}^{(p-1)}:= \frac{32Plk}{\tilde d}T_{max}^{(p)}
    \end{equation*}
    queries, that uses the same memory and wins with probability at least $q - \frac{1}{2P}$.
\end{lemma}

\begin{proof}
    Fix $p\in\{2,\ldots,P\}$ and the strategy for depth $p$. By \cref{lemma:many_depth_p-1}, on an event $\Ecal$ of probability at least $q-\frac{3}{8P}$, there are at least $K_0= \ceil{\frac{\tilde d}{4lk}}$ completed depth-$(p-1)$ periods. 

    The main remark is that given $E_1,\ldots,E_p$, the separation oracle from Procedure~\ref{proc:feasibility} is defined exactly similarly for each new depth-$(p-1)$ period. This is the reason why we reset all information from depths $q\leq p'$ whenever a period of depth $p'$ ends (line 12 of Procedure~\ref{proc:feasibility}). In fact, the distribution of outcomes for the run of a depth-$(p-1)$ period is completely characterized by the memory state at the beginning of that period, as well as the exploratory queries for depths $p'>p-1$. Under $\Ecal$, in average, the depth-$(p-1)$ periods are completed using relatively few iterations, hence we aim to simulate a depth-$(p-1)$ run to solve Game~\ref{game:feasibility_game} for depth $p-1$. The strategy for the player is as follows:
    \begin{enumerate}
        \item Draw an index $\mathsf{it}\sim \Ucal([K_0])$.
        \item Run the strategy for depth $p$ using the separating oracle from Procedure~\ref{proc:feasibility} until the beginning of the $\mathsf{it}$-th period of depth $p-1$. If the procedure never finishes $\mathsf{it}-1$ depth-$(p-1)$ periods, the strategy fails. When needed to sample new probing spaces for depths $p'>p-1$, use those provided by the oracle in line 1 of Game~\ref{game:feasibility_game} (this was already done for $p'>p$ before, now we also use these for depth $p'=p$).
        \item Set $\mathsf{Memory}$ to be the memory of the algorithm, and the exploratory queries for depths $p'>p-1$ to be exactly as in the beginning of the $\mathsf{it}$-th period of depth $p-1$.
    \end{enumerate}
    The complete strategy for Game~\ref{game:orthogonal_subspace_game} at depth $p-1$ is described in \cref{alg:strategy_recursion}.
    
\begin{algorithm}[ht!]
        
\caption{Strategy of the Player for Depth-$(p-1)$ Game~\ref{game:feasibility_game} given a strategy for depth $p$}\label{alg:strategy_recursion}

\setcounter{AlgoLine}{0}
\SetAlgoLined
\LinesNumbered

\everypar={\nl}

\hrule height\algoheightrule\kern3pt\relax
\KwIn{dimensions $d$, $\tilde d$, number of vectors $k$, depth $p$, $M$-bit memory algorithm $alg$ for Game~\ref{game:feasibility_game} at depth $p$}

\KwOut{strategy for Game~\ref{game:feasibility_game} at depth $p-1$}

\vspace{5pt}

Receive subspaces $E_1,\ldots,E_P$ and probing subspaces $V_j^{(p-1+i)}$ for $i\in[P-p+1]$ and $j\in[k]$

Initialize memory of $alg$ with the same $M$-bit message $\mathsf{Message}$ and define $n_{p+i}\in[k]$ and exploratory vectors $\mb y_j^{(p+i)}$ for $i\in[n_{p+i}]$ for all $i\in[P-p]$ as in the Depth-$p$ Game~\ref{game:feasibility_game} (ignoring probing subspaces $V_j^{(p)}$) 

Set $n_{p'}\gets 0$ for $p'\in[p]$, reset probing subspaces $V_j^{(p+i)}$ for $i\in[P-p]$ and $j>n_{p+i}$

Sample $\mathsf{it}\sim \Ucal([K_0])$. Initialize $\mathsf{EnoughPeriods}\gets$\textbf{false}

\For{$t\in[T_{max}^{(p)}]$}{
    Run $alg$ with current memory to get query $\mb x_t$. Update exploratory queries and probing subspaces as in Procedure~\ref{proc:feasibility} and \Return $\mb g_t = \Ocal_{\mb V}(\mb x_t)$ as response to $alg$. When needed to sample a depth-$p$ probing subspace $V_j^{(p)}$, use one provided by the oracle in line 1 that was not yet used, and with smallest index $j$.
    
    \lIf{$n_{p-1}$ was reset because of deeper periods}{Strategy fails, \textbf{end} procedure}

    \lIf{$n_{p-1}$ was reset for $\mathsf{it}$-th time (counting $t=1$s)}{$\mathsf{EnoughPeriods}\gets$\textbf{true}, $t(\mathsf{it})\gets t$}
}

\uIf{$\mathsf{EnoughPeriods}$}{
    Submit to the oracle the memory state $\mathsf{Memory}$ of $alg$, as well as all values of $n_{p-1+i}$ for $i\in[P-p+1]$ and exploratory queries $\mb y_j^{(p-1+i)}$ for $j\in[n_{p-1+i}],i\in[P-p+1]$, just before starting iteration $t(\mathsf{it})$
}
\lElse{Strategy fails, \textbf{end} procedure}

\hrule height\algoheightrule\kern3pt\relax
\end{algorithm}

We now estimate the probability of success of this strategy for the Depth-$(p-1)$ Game~\ref{game:feasibility_game}. Notice that after the strategy submits an $M$-bit message and exploratory queries to the oracle, in lines 4-9 of the Depth-$(p-1)$ Game~\ref{game:feasibility_game}, the oracle proceeds to simulate the run of the $\mathsf{it}$-th depth-$(p-1)$ period of Depth-$p$ Game~\ref{game:feasibility_game}. Indeed, the responses obtained by $alg$ in the run initialized by \cref{alg:strategy_recursion} are stochastically equivalent to those that were obtained during that $\mathsf{it}$-th depth-$(p-1)$ period because they were generated using the same process. As a result, the final run in lines 4-9 of the Depth-$(p-1)$ Game~\ref{game:feasibility_game} is stochastically equivalent to the last step of the following procedure:

\begin{enumerate}
    \item Run the complete depth-$p$ strategy for Depth-$p$ Game~\ref{game:feasibility_game}.
    \item Sample $\mathsf{it}\sim\Ucal([K_0])$ independently from the previous run.
    \item If there were no $\mathsf{it}-1$ finished depth-$(p-1)$ periods in the previous run, strategy fails.
    \item Otherwise, Re-run the $\mathsf{it}$-th depth-$(p-1)$ period with the exact same randomness as in item 1, for at most $T_{max}^{(p-1)}$ iterations.
\end{enumerate}

In the rest of the proof, we will prove success probabilities for this construction.
Note that on $\Ecal$, because $K_0$ depth-$(p-1)$ periods were completed, the strategy does not fail at step 3 and step 4 exactly implements the $\mathsf{it}$-th depth-$(p-1)$ period of Depth-$p$ Game~\ref{game:feasibility_game}. Further, this period will be complete, given enough iterations to be finished. That is, if it uses at most $T_{max}^{(p-1)}$ iterations, the player wins at the Depth-$(p-1)$ Game~\ref{game:feasibility_game}.
We therefore aim to bound the number of iterations $\mathsf{length}(i)$ needed for the $i$-th depth-$(p-1)$ period, with the convention $\mathsf{length}(i):=\infty$ if this period was never finished. We have
\begin{align*}
    \Ebb[\mathsf{length}(\mathsf{it}) \mid \Ecal] &= \frac{1}{K_0}\sum_{i=1}^{K_0} \mathsf{length}(i) \leq \frac{T_{max}^{(p)}}{K_0}.
\end{align*}
As a result, letting $\Fcal = \set{\mathsf{length}(\mathsf{it}) \leq T_{max}^{(p-1)}:=\frac{8PT_{max}^{(p)}}{K_0} }$, we have
\begin{equation*}
    \Pbb(\Fcal \mid \Ecal) \geq 1-\frac{1}{8P}.
\end{equation*}
In summary, on $\Ecal\cap\Fcal$, \cref{alg:strategy_recursion} wins at Game~\ref{game:feasibility_game} at depth $p-1$. Thus,
\begin{equation*}
    \Pbb(\text{\cref{alg:strategy_recursion} wins}) \geq \Pbb(\Ecal\cap\Fcal) \geq \paren{q- \frac{3}{8P}}\paren{1- \frac{1}{8P}} \geq q-\frac{1}{2P}.
\end{equation*}
Because $T_{max}^{(p-1)} \leq \frac{32Plk}{\tilde d}T_{max}^{(p)}$, this ends the proof of the result.
\end{proof}

We now apply \cref{lemma:recursion} recursively to progressively reduce the depth of Game~\ref{game:feasibility_game}. This gives the following result.

\begin{theorem}\label{thm:final_lower_bound_game}
    Let $P\geq 2$, $d\geq 40P$. Suppose that 
    \begin{equation}\label{eq:assumption_on_k}
         c_2 \frac{M + dP\ln d}{d}P^3\ln d \leq k \leq  c_1\frac{d}{C_\alpha P \ln d} 
    \end{equation}
    for some universal constants $c_1,c_2>0$. If a strategy for Game~\ref{game:feasibility_game} for depth $P$ uses $M$ bits of memory and wins with probability at least $\frac{1}{2}$, then it performed at least
    \begin{equation*}
        T_{max} \geq \frac{k}{2}\paren{\frac{d}{100 P^2lk}}^{P-1}
    \end{equation*}
    queries.
\end{theorem}

\begin{proof}
    Define for any $p\in[P]$,
    \begin{equation*}
        T_{max}^{(P)} = \frac{k}{2}\paren{\frac{\tilde d}{32 Plk}}^{p-1}.
    \end{equation*}
    Suppose for now that the parameter $k$ satisfies all assumptions from \cref{lemma:many_depth_p-1} for all $p\in\{2,\ldots,P\}$. Then, starting from a strategy for Game~\ref{game:feasibility_game} at depth $P$ and that wins with probability $q\geq \frac{1}{2}$, \cref{lemma:recursion} iteratively constructs strategies for $p\in[P]$ for Game~\ref{game:feasibility_game} at depth $p$ with $T_{max}^{(p)}$ iterations that wins with probability $q-\frac{1}{2P}(P-p)$. Now to win Game~\ref{game:feasibility_game} with depth $1$, one needs to make at least $k$ queries (the exploratory queries). Hence, no algorithm wins with such probability $q-\frac{1}{2P}(P-1)$ using $T_{max}^{(1)}$ queries. Recall that $\tilde d \geq d/(3P)$, hence
    \begin{equation*}
        T_{max}^{(P)} \geq  \frac{k}{2}\paren{\frac{d}{100 P^2lk}}^{P-1}.
    \end{equation*}

    The only remaining step is to check that all assumptions from \cref{lemma:many_depth_p-1} are satisfied. It suffices to check that
    \begin{equation}\label{eq:assumption_needed}
        \frac{\tilde d}{4l} \geq k \geq 50\cdot(4P) \frac{M + \frac{d}{8} \log_2(T_{max}^{(P)})+3}{\tilde d} \ln \frac{\sqrt d}{\delta_1}.
    \end{equation}
    We start with the upper bound. Recalling the definition of $l$ in Eq~\eqref{eq:def_l}, we have that
    \begin{equation*}
        \frac{\tilde d}{4l} \geq \frac{d}{12Pl} = \Omega\paren{\frac{d}{C_\alpha P\ln d}}.
    \end{equation*}
    Now for the upper bound,
    \begin{equation*}
        50\cdot(4P) \frac{M + \frac{d}{8} \log_2(T_{max}^{(P)})+3}{\tilde d} \ln \frac{\sqrt d}{\delta_1} =\Ocal\paren{ \frac{M+dP\ln d}{d}P^3 \ln d}.
    \end{equation*}
    Hence, for a choice of constants $0<c_1<c_2$ that we do not specify, Eq~\eqref{eq:assumption_needed} holds.
\end{proof}

\subsection{Reduction from the feasibility procedure to the feasibility game}

The only step remaining is to link Procedure~\ref{proc:feasibility} to the Game~\ref{game:feasibility_game} at depth $P$. Precisely, we show that during a run of Procedure~\ref{proc:feasibility}, many depth-$P$ periods are completed -- recall that Game~\ref{game:feasibility_game} at depth $P$ exactly corresponds to the run of a depth-$P$ period of Procedure~\ref{proc:feasibility}. For this, we first prove a simple query lower bound for the following game, that emulates the discovery of the subspace $E_P$ at the last layer $P$.

\begin{game}[h!]

\caption{Kernel discovery game}\label{game:kernel_discovery}

\setcounter{AlgoLine}{0}

\SetAlgoLined
\LinesNumbered

\everypar={\nl}

\hrule height\algoheightrule\kern3pt\relax
\KwIn{$d$, subspace dimension $\tilde d$, number of samples $m$}

\textit{Oracle:} Sample a uniformly random $\tilde d$-dimensional linear subspace $E$ of $\Rbb^d$\,

\textit{Oracle:} Send i.i.d. samples $\mb v_1,\ldots,\mb v_m \overset{i.i.d.}{\sim} \Ucal(S_{d-1}\cap E)$ to player

\textit{Player:} Based on $\mb v_1,\ldots,\mb v_m$ output a unit vector $\mb y$\,

The learner wins if $\|\proj_{E}(\mb y)\| < \sqrt{\frac{\tilde d}{20d}}$.
\hrule height\algoheightrule\kern3pt\relax
\end{game}

We show that to win the Kernel discovery Game~\ref{game:kernel_discovery} with reasonable probability, one needs $\Omega(\tilde d)$ queries. This is to be expected since finding orthogonal vectors to $E$ requires having information on the complete $\tilde d$-dimensional space.

\begin{lemma}
\label{lemma:kenel_discovery_query_lower_bound}
    Let $m\leq \frac{\tilde d}{2}\leq\frac{d}{4}$. No algorithm wins at Game~\ref{game:kernel_discovery} with probability more than $e^{-\tilde d/10}$.
\end{lemma}

\begin{proof}
    Suppose that $m\leq \frac{\tilde d}{2}$. First note that with probability one, the samples $\mb v_1,\ldots,\mb v_m$ are all linearly independent. Conditional on $\mb v_1,\ldots,\mb v_m$, the space $E$ can be decomposed as
    \begin{equation*}
        E = Span(\mb v_1,\ldots,\mb v_m) \oplus F
    \end{equation*}
    where under $\Ecal$, $F = E\cap  Span(\mb v_i,i\in[m])^\perp$ is a uniform $(\tilde d-m)$-dimensional subspace of $Span(\mb v_i,i\in[m])^\perp$. Now letting $\mb z = \proj_{Span(\mb v_i,i\in[m])^\perp}(\mb y)$, one has
    \begin{equation}\label{eq:simple_decomposition}
        \|\mb z\|^2 = \|\mb y\|^2 - \|\proj_{Span(\mb v_i,i\in[m])}(\mb y)\|^2 = 1- \|\mb y-\mb z\|^2.
    \end{equation}
    Further, provided that $\mb z\neq 0$, from the point of view of $F$, the vector $\frac{\mb z}{\|\mb z\|}$ is a random uniform unit vector in $Span(\mb v_1,\ldots,\mb v_m)^\perp$. Formally, \cref{lemma:concentration_projection} shows that
    \begin{equation*}
        \Pbb\paren{\|\proj_F(\mb z)\| < \|\mb z\|\sqrt{\frac{\tilde d-m}{10(d-m)}}} \leq e^{-(\tilde d-m)/5} \leq e^{-\tilde d/10}
    \end{equation*}
    In the last inequality, we used $m\leq \frac{\tilde d}{2}$. We denote by $\Fcal$ the complement of this event. Then, under $\Ecal\cap\Fcal$,
    \begin{align*}
        \|\proj_E(\mb y)\|^2 &= \|\mb y-\mb z\|^2 + \|\proj_F(\mb z)\|^2\\
        &\geq \|\mb y-\mb z\|^2 + \frac{\tilde d-m}{10(d-m)} \|\mb z\|^2\\
        &\geq \frac{\tilde d-m}{10(d-m)}  + \|\mb y-\mb z\|^2 \paren{1-\frac{\tilde d-m}{10(d-m)}}  \geq \frac{\tilde d}{20d}.
    \end{align*}
    In the second inequality we used Eq~\eqref{eq:simple_decomposition}. In summary, the player loses with probability $\Pbb(\Ecal\cap\Fcal) \geq 1-e^{-\tilde d/10}$. This ends the proof.
\end{proof}

Using a reduction from Procedure~\ref{proc:feasibility} to the Kernel Discovery Game~\ref{game:kernel_discovery}, we use the previous query lower bound to show that to solve Procedure~\ref{proc:feasibility}, the algorithm needs to complete $\Omega(d/(Plk))$ depth-$P$ periods.

\begin{lemma}\label{lemma:final_depth_exploration}
    Let $alg$ be an algorithm for Procedure~\ref{proc:feasibility}. Suppose that $4l_Pk \leq \tilde d$ and $l_P\geq l$. Then, with probability at least $1-\frac{1}{8P} - e^{-\tilde d/10}$, during the run of Procedure~\ref{proc:feasibility}, there were are least $\frac{\tilde d}{2l_Pk}$ completed periods of depth $P$.
\end{lemma}

\begin{proof}
    Similarly as in the proof of \cref{lemma:many_depth_p-1}, by \cref{lemma:most_periods_proper}, we know that on an event $\Ecal$ of probability at least $1-4d^2e^{-l_P/16}$, the first $d$ periods of depth $P$ are proper. In our context, this means that under $\Ecal$, if the algorithm $alg$ was successful for Procedure~\ref{proc:feasibility} within the first $d$ depth-$P$ periods, then the final query $\mb x_t$ for which $\Ocal_{\mb V}(\mb x_t)=\mathsf{Success}$ satisfied $\|\proj_{E_P}(\mb x_t)\| \leq \eta_P$. We can therefore construct a strategy for the Kernel Discovery Game~\ref{game:kernel_discovery} as follows: we simulate a run of Procedure~\ref{proc:feasibility} using $E_P=E$ the subspace provided by the oracle. When needed to construct a new depth-$P$ probing subspace, we use $l$ vectors of the sequence $\mb v_1,\ldots,\mb v_m$ similarly as in \cref{alg:strategy_orthogonal_subspace_game} for the proof of \cref{lemma:many_depth_p-1}. The strategy is formally defined in \cref{alg:strategy_discovery_game}.

        \begin{algorithm}[ht!]
        
\caption{Strategy of the Player for the Kernel Discovery Game~\ref{game:kernel_discovery}}\label{alg:strategy_discovery_game}

\setcounter{AlgoLine}{0}
\SetAlgoLined
\LinesNumbered

\everypar={\nl}

\hrule height\algoheightrule\kern3pt\relax
\KwIn{depth $P$, dimensions $d$, $\tilde d$, $l$, number of exploratory queries $k$, $M$-bit algorithm $alg$, number of samples $m=\floor{\tilde d/2}$}

\vspace{5pt}

Sample independently $E_1,\ldots, E_{P-1}$, uniform $\tilde d$-dimensional subspaces of $\Rbb^d$

Initialize $n_p\gets 0$ for $p\in[P]$ and set memory of $alg$ to $\mb 0$

Receive samples $\mb v_1,\ldots,\mb v_m$ and set sample index $i\gets 1$

\While{$alg$ did not receive a response $\mathsf{Success}$}{
    Run $alg$ with current memory to obtain query $\mb x$. Update exploratory queries and probing subspaces as in lines 5-13 of Procedure~\ref{proc:feasibility}. Whenever needed to sample a depth-$P$ probing subspace $V_{n_P}^{(P)}$ of $E_P$:

    \lIf{$i+l-1>m$}{Strategy fails: \textbf{end} procedure}
    \lElse{use oracle samples, set $V_{n_P}^{(P)}:=Span(\mb v_i,\ldots,\mb v_{i+l-1})$ and $i\gets i+l$}

    \Return $\Ocal_{\mb V}(\mb x)$ as response to $alg$

    \lIf{$\Ocal_{\mb V}(\mb x)=\mathsf{Success}$}{ \Return $\frac{\mb x}{\|\mb x\|}$ to oracle, \textbf{break}}
}

\hrule height\algoheightrule\kern3pt\relax
\end{algorithm}

    From the previous discussion, letting $\Fcal$ be the event that at most $n_0:=\floor{\frac{\tilde d}{2l_Pk}}$ depth-$P$ periods were completed, we have that under $\Ecal\cap\Fcal$ \cref{alg:strategy_discovery_game}, needed at most $l_Pk n_0 \leq \frac{\tilde d}{2}$ samples from the oracle and the last vector vector $\mb x$ satisfies $\|\proj_{E_P}(\mb x)\|= \|\proj_{E}(\mb x)\|\leq \eta_P$. Now recall that because $\mb x$ was successful for $\Ocal_{\mb V}$ we must have $\mb e^\top \mb x \leq -\frac{1}{2}$ and as a result $\|\mb x\|\geq \frac{1}{2}$. Hence, the output vector $\mb y = \frac{\mb x}{\|\mb x\|}$ satisfies
    \begin{equation*}
        \|\proj_E(\mb y)\| \leq 2\eta_P < \sqrt{\frac{\tilde d}{20d}}.
    \end{equation*}
    In the last inequality, we used Eq~\eqref{eq:def_eta}. Hence, using \cref{lemma:kenel_discovery_query_lower_bound}, we have
    \begin{equation*}
        \Pbb(\Ecal\cap\Fcal)\leq \Pbb(\cref{alg:strategy_discovery_game}\text{ wins})\leq e^{-\tilde d/10}.
    \end{equation*}
    In particular,
    \begin{equation*}
        \Pbb(\Fcal) \leq \Pbb(\Ecal^c)+\Pbb(\Ecal\cap\Fcal) \leq 4d^2e^{-l_P/16} + e^{-\tilde d/10} \leq \frac{1}{8P} + e^{-\tilde d/10}.
    \end{equation*}
    In the last inequality we used $l_P\geq l\geq 16\ln(32d^2P)$ from Eq~\eqref{eq:def_l}.
    Because on $\Fcal^c$ there are at last $n_0+1\geq \frac{\tilde d}{2lk}$ completed depth-$P$ periods, this ends the proof.
\end{proof}

\cref{lemma:final_depth_exploration} shows that many depth-$P$ periods are performed during a run of Procedure~\ref{proc:feasibility}. Because the depth-$P$ Game~\ref{game:feasibility_game} exactly simulates a depth-$P$ period of Procedure~\ref{proc:feasibility}, we can combine this with our previous lower bound to obtain the following.

\begin{theorem}\label{thm:query_lower_bound_proc}
    Let $P\geq 2$ and $d\geq 20P$. Suppose that $k$ satisfies Eq~\eqref{eq:assumption_on_k} as in \cref{thm:final_lower_bound_game}. Also suppose that $4l_pk\leq \tilde d$ and $l_P\geq l$. If an algorithm for Procedure~\ref{proc:feasibility} uses $M$ bits of memory and wins making at most $T_{max}$ queries with probability at least $\frac{3}{4}$, then
    \begin{equation*}
        T_{max} \geq \frac{kl}{l_P} \paren{\frac{d}{100 P^2lk}}^P.
    \end{equation*}
\end{theorem}

\begin{proof}
    \cref{lemma:final_depth_exploration} plays exactly the same role as \cref{lemma:many_depth_p-1}. In fact, we can easily check that the exact same proof as for \cref{lemma:recursion} shows that if there is an algorithm for Procedure~\ref{proc:feasibility} that uses at most $T_{max}$ queries and wins with probability at least $q\geq \frac{3}{4}$, then there is a strategy for Game~\ref{game:feasibility_game} for depth $P$, that uses the same memory and at most
\begin{equation*}
    T_{max}^{(P)}:= \frac{8P}{\tilde K_0}T_{max}
\end{equation*}
queries, where $\tilde K_0=\frac{\tilde d}{2l_Pk}$ is the number of depth-$P$ periods guaranteed by \cref{lemma:final_depth_exploration}. Further, it wins with probability at least $q-\frac{1}{8P}-e^{-\tilde d/10}$. The failure probability $\frac{1}{8P}+e^{-\tilde d/10}$ corresponds to the failure probability of \cref{lemma:final_depth_exploration}. Hence this win probability is more than $\frac{1}{2}$ since $P\geq 2$ and $\tilde d\geq \frac{d}{2}\geq 20$. By \cref{thm:final_lower_bound_game} we must have
    \begin{equation*}
        T_{max}= \frac{\tilde d}{16Pl_Pk}T_{max}^{(P)} \geq  \frac{\tilde d}{16Pl_Pk} \frac{k}{2} \paren{\frac{d}{100 P^2lk}}^{P-1} \geq \frac{kl}{l_P} \paren{\frac{d}{100 P^2lk}}^P.
    \end{equation*}
    This ends the proof.
\end{proof}

From \cref{lemma:check_procedure_is_feasibility_pb}, we know that with high probability on $E_1,\ldots,E_P$, Procedure~\ref{proc:feasibility} implements a valid feasibility problem for accuracy $\epsilon=\delta_1/2$. Combining this with the previous query lower bound for Procedure~\ref{proc:feasibility} gives the desired final result for deterministic algorithms.

\begin{theorem}\label{thm:deterministic_alg_lower_bound}
    Fix $\alpha\in(0,1]$, $d\geq 1$ and an accuracy $\epsilon\in(0,\frac{1}{\sqrt d}]$ such that
    \begin{equation*}
        d\ln^2 d \geq \paren{\frac{c_3}{\alpha}}^{\frac{\ln 2}{\alpha}} \cdot  \frac{M}{d}  \ln^4 \frac{1}{\epsilon}
    \end{equation*}
    for some universal constant $c_3$. Then, any $M$-bit deterministic algorithm that solves feasibility problems up to accuracy $\epsilon$ makes at least
    \begin{equation*}
        \paren{\frac{d}{M}}^\alpha 
 \frac{1}{\epsilon^{2\psi(d,M,\epsilon)}}
    \end{equation*}
    queries, where $\psi(d,M,\epsilon) =  \frac{1- \ln\paren{ \frac{M}{d}} / \ln d}{1+(1+\alpha)\ln\paren{ \frac{M}{d}} / \ln d} - \Ocal\paren{\frac{\ln \frac{\ln(1/\epsilon)}{\ln d}}{\ln d} + \frac{\ln\ln d}{\ln d}}$.
\end{theorem}

\begin{proof}
    Roughly speaking, the proof consists in finding parameters for $P$ and $k$ that (1) satisfy the assumptions to apply the query lower bound \cref{thm:query_lower_bound_proc} on Procedure~\ref{proc:feasibility}, (2) for which this procedure emulates $\epsilon$-accuracy feasibility problems with high accuracy, and (3) that maximizes the query lower bound provided in \cref{thm:query_lower_bound_proc}. 

    First, we recall that $d\log_2\frac{1}{2\epsilon}$ bits of memory are necessary to solve the feasibility problem because this is already true for optimizing $1$-Lipschitz functions on the unit ball \cite[Theorem 5]{woodworth2019open}. Without loss of generality, we therefore suppose that $M\geq d\ln\frac{1}{\epsilon}$. For convenience, let us define the following quantity
    \begin{equation*}
        \tilde P := \frac{2\ln \frac{1}{\epsilon} - (1+\alpha)\ln\frac{M}{d} - (4+3\alpha) \ln \paren{\frac{\ln(1/\epsilon)}{\ln d}+1} -2\ln(c_2\ln d) - 24}{\ln d + (1+\alpha)\ln \frac{M}{d} + 3(1+\alpha)\ln \paren{\frac{\ln(1/\epsilon)}{\ln d}+1} + \ln(c_2\ln d) + 12}.
    \end{equation*}
    We now define the parameters $P$, $k$, and $l_P$ as
    \begin{align*}
        P:= \ceil{ \tilde P}, \quad
        k:=\ceil{3c_2 \frac{M}{d}P^3\ln d },\quad \text{and}\quad 
        l_P:= l\lor  \floor{\paren{\frac{\tilde d}{4k}}^{P-\tilde P}} .
    \end{align*}
    In particular, note that $P< P_{max}:=\frac{2\ln\frac{1}{\epsilon}}{\ln d}+1$ and we directly have $l_P\geq l$. We also recall that under the assumptions from \cref{thm:query_lower_bound_proc}, we showed that $4lk \leq \tilde d$. As a result, we also have $4l_Pk\leq \tilde d$ provided that Eq~\eqref{eq:assumption_on_k} is satisfied.
    Now fix an algorithm $alg$ for the $\epsilon$-accuracy feasibility problem that uses at most $M$ bits of memory and uses at most $T_{max}$ separation oracle queries. By assumption, we have
    \begin{equation*}
        d\ln d \geq \sqrt{c_3} \ln\frac{1}{\epsilon}.
    \end{equation*}
    Hence, setting $c_3\geq 20^2$ we have that $P\leq P_{max} \leq \frac{d}{20}$. Assuming now that $P\geq 2$, we can apply \cref{lemma:check_procedure_is_feasibility_pb} which shows that on an event $\Ecal$ of probability at least $1-e^{-d/40}\geq 1-e^{-2}>\frac{3}{4}$, Procedure~\ref{proc:feasibility} using $alg$ emulates a valid $(\delta_1/2)$-accuracy feasibility problem. Note that
    \begin{equation*}
        \frac{\delta_1}{2}=\frac{1}{720\mu^{P-2}\mu_P} \sqrt{\frac{l}{ dk^{\alpha}}}  = \frac{5}{6\mu^P}\sqrt{\frac{kl_P}{l}}  
        \geq \frac{1}{\mu^{\tilde P}}\sqrt{\frac{k}{2l}\paren{\frac{\tilde d}{4k\mu^2}}^{P-\tilde P}} .
    \end{equation*}
    In the inequality, we used $k\geq 3$. Furthering the bounds and using $P-\tilde P<1$, we obtain
    \begin{equation*}
        \frac{\delta_1}{2} \geq \frac{1}{\mu^{\tilde P}} \sqrt{\frac{\tilde d}{8l\mu^2}} \geq \frac{1}{3000 \mu^{\tilde P}\sqrt{Pk^{1+\alpha}}} \geq \epsilon.
    \end{equation*}
    In the last inequality, we used the definition of $\tilde P$, which using $\alpha\leq 1$, implies in particular
    \begin{equation*}
        \tilde P\leq \frac{2\ln \frac{1}{\epsilon} -  \ln P_{max} - (1+\alpha)\ln(6c_2\frac{M}{d}P_{max}^3 \ln d) -  2\ln (3000)}{\ln d + (1+\alpha)\ln \paren{6c_2\frac{M}{d}P_{max}^3 \ln d} - \ln(16\ln d)+ 2\ln (600)} \leq \frac{\ln\frac{1}{\epsilon} - \ln(3000\sqrt{Pk^{1+\alpha}})}{\ln \mu}.
    \end{equation*}
    
    As a result, because $\delta_1/2\geq \epsilon$, under that same event $\Ecal$, Procedure~\ref{proc:feasibility} terminates with at most $T_{max}$ queries. We now check that the choice of $k$ satisfies Eq~\eqref{eq:assumption_on_k}. Note that $dP\ln d \leq 2d\ln\frac{1}{\epsilon}$. As a result, $M+dP\ln d \leq 3M$ and hence $k$ directly satisfies the left-hand side inequality of Eq~\eqref{eq:assumption_on_k}. The assumption gives
    \begin{equation*}\label{eq:assumption_on_d_P}
        d\ln^2 d \geq \frac{100C_\alpha c_2}{c_1} \frac{M}{d}  \ln^4 \frac{1}{\epsilon}.
    \end{equation*}
    As a result, $c_1\frac{d}{C_\alpha P\ln d}  \geq c_1\frac{d}{C_\alpha P_{max}\ln d} 
 \geq 6 c_2 (M/d) P_{max}^3\ln d \geq k$. As a result,
    the right-hand side is of Eq~\eqref{eq:assumption_on_k} is also satisfied. We can now apply \cref{thm:query_lower_bound_proc} which gives
    \begin{align*}\label{eq:actual_bound_deterministic}
        T_{max} \geq \frac{kl}{l_P}\paren{\frac{d}{100 P^2 l k}}^P  \geq \frac{k}{(13 P l)^{P-\tilde P}} \paren{\frac{d}{100 P^2 l k}}^{\tilde P}& \geq \frac{c_2P^2(M/d)\ln d}{5l} \paren{\frac{d}{100 P^2 l k}}^{\tilde P}\\
        &=: \frac{1}{\epsilon^{2\tilde \psi(d,M,\epsilon)}},
    \end{align*}
    where we defined $\tilde \psi(d,M,\epsilon)$ through the last equality.
    Note that the above equation always holds, even if $P<2$ (that is, $\tilde P\leq 1$) because in that case it is implied by $T_{max}\geq d$ (which is necessary even for convex optimization). We wrote the above equation for the sake of completeness; the above computations can be simplified to
    \begin{align*}
        \tilde \psi(d,M,\epsilon) &= \frac{\ln \frac{M}{d}}{2\ln \frac{1}{\epsilon}} + \Ocal\paren{\frac{\ln\ln\frac{1}{\epsilon}}{\ln\frac{1}{\epsilon}}} + \frac{\tilde P \ln \frac{d}{100P^2lk}}{2 \ln \frac{1}{\epsilon}}\\
        &=\frac{\ln \frac{M}{d}}{2\ln \frac{1}{\epsilon}} 
        + \frac {
        \ln d
        -\ln\frac{M}{d}
        - \frac{1+\alpha}{2}\frac{\ln d \cdot \ln\frac{M}{d}}{\ln \frac{1}{\epsilon}}
        - 4 \ln \frac{\ln (1/\epsilon)}{\ln d} 
        }
        { \ln d+ (1+\alpha) \ln \frac{M}{d} + 3(1+\alpha) \ln \frac{\ln (1/\epsilon)}{\ln d} }+ \Ocal\paren{\frac{\ln\ln d}{\ln d}}\\
        &\geq -\frac{\alpha\ln \frac{M}{d}}{2\ln \frac{1}{\epsilon}} 
        + \frac {
        \ln d
        -\ln\frac{M}{d}
        - 4 \ln \frac{\ln (1/\epsilon)}{\ln d} 
        }
        { \ln d+ (1+\alpha)\ln \frac{M}{d} + 3(1+\alpha) \ln \frac{\ln (1/\epsilon)}{\ln d} }+ \Ocal\paren{\frac{\ln\ln d}{\ln d}}.
    \end{align*}
    This ends the proof of the theorem.
\end{proof}

In the standard regime when $\ln\frac{1}{\epsilon} \leq d^{o(1)}$, the query lower bound from \cref{thm:deterministic_alg_lower_bound} can be greatly simplified, and directly implies \cref{thm:main_result_deterministic}.

\section{Query complexity/memory tradeoffs for randomized algorithms}
\label{sec:randomized}

The feasibility procedure defined in Procedure~\ref{proc:feasibility} is adaptive in the algorithm queries. As a result, this approach fails to give query lower bounds for randomized algorithms, which is the focus of the present section.

Although Procedure~\ref{proc:feasibility} is adaptive, note that the generated subspaces $E_1,\ldots, E_P$ and probing subspaces $V_i^{(p)}$ for $i\in[k]$ and $p\in[P]$ are not. In fact, the only source of adaptivity comes from deciding when to add a new probing subspace for any depth $p\in[P]$. In Procedure~\ref{proc:feasibility} this is done when the algorithm performs a depth-$p$ exploratory query. We now present an alternative feasibility procedure for which the procedure oracle does not need to know when exploratory queries are performed, at the expense of having worse query lower bounds.

\subsection{Definition of the hard class of feasibility problems}

The subspaces $E_1,\ldots,E_P$ are sampled exactly as in Procedure~\ref{proc:feasibility} as independent uniform $\tilde d$-dimensional subspaces of $\Rbb^d$ where $\tilde d=\floor{d/(2P)}$. As before, for each space $E_p$ for $p\in[P]$ we construct $l$-dimensional probing subspaces.
However, given such probing subspaces say $V_1,\ldots,V_r$ for $r\in[k]$, we define a different depth-$p$ oracle. We will always assume that $k\geq 3$. Precisely, given parameters $\mb\delta=(\delta_1,\ldots,\delta_k)\in(0,\infty)^k$, we first define the set
\begin{equation*}
    \Ical_{V_1,\ldots,V_r}(\mb x;\mb \delta) = \set{i\in[k]: \|\proj_{V_i}(\mb x)\| > \delta_i}.
\end{equation*}
The oracle is then defined as
\begin{equation}\label{eq:def_layer_oracle_randomized}
    \tilde{\mb g}_{V_1,\ldots,V_r}(\mb x;\mb \delta) := \begin{cases}
        \frac{\proj_{V_i}(\mb x)}{\|\proj_{V_i}(\mb x)\|} &\text{if } \Ical_{V_1,\ldots,V_r}(\mb x;\mb \delta)\neq \emptyset \text{ and } i=\min \Ical_{V_1,\ldots,V_r}(\mb x;\mb \delta),\\
        \mathsf{Success} &\text{otherwise}.
    \end{cases}
\end{equation}
Compared to $\mb g_{V_1,\ldots,V_r}$, this oracle does not combine probing subspaces by taking their span, and prioritizes separation hyperplanes constructed from probing subspaces with the smallest index $i\in[k]$. In the oracle, only functions $\mb {\tilde g}_{V_1,\ldots,V_k}$ which have exactly $k$ subspaces are used -- the definition for $r<k$ subspaces will only be useful for the proof.

For each depth $p\in[P]$, we will sample $k$ probing subspaces $V_1^{(p)},\ldots,V_k^{(p)}$ as before: these are i.i.d. $l_p$-dimensional subspaces of $E_p$. This time, we set
\begin{equation}\label{eq:def_l_randomized}
    l_p=l:= \ceil{C k^3\ln d},\quad p\in[P-1],
\end{equation}
for a universal constant $C\geq 1$ introduced in \cref{lemma:properties_oracle_construction}. We let $l_P\in[\tilde d]$ with $l_P\geq l$ be a parameter as in the deterministic case. Also, these probing subspaces will be resampled regularly throughout the feasibility procedure. We use the notation $\mb V^{(p)} =(V_1^{(p)},\ldots,V_k^{(p)})$, noting that here $\mb V^{(p)}$ always contains all $k$ probing subspaces contrary to Procedure~\ref{proc:feasibility}. Given these subspaces, the format of the oracle is similar as in Eq~\eqref{eq:def_separation_oracle}:
\begin{equation}\label{eq:def_multi_oracle_randomized}
    \tilde \Ocal_{\mb V^{(1)},\ldots, \mb V^{(P)}}(\mb x):= \begin{cases}
        \mb e &\text{if } \mb e^\top \mb x > -\frac{1}{2}\\
        \tilde{\mb g}_{\mb V^{(p)}}(\mb x;\mb \delta^{(p)} ) &\text{if } \tilde{\mb g}_{\mb V^{(p)}}( \mb x;\mb \delta^{(p)} ) \neq \mathsf{Success} \\
        &\text{and } p=\min\set{q\in[P], \; \tilde{\mb g}_{\mb V^{(q)}} (\mb x;\mb \delta^{(q)} ) \neq \mathsf{Success}}\\
        \mathsf{Success} &\text{otherwise}.
    \end{cases}
\end{equation}
Before defining the parameters $\mb\delta^{(1)},\ldots,\mb\delta^{(P)}$, we first define $\eta_1,\ldots,\eta_P$ as follows. First, let
\begin{equation*}
    \eta_P:= \frac{1}{10}\sqrt{\frac{\tilde d}{d}} \quad \text{and}\quad  \mu_P:= 1200 k\sqrt{\frac{kd}{l_P}},
\end{equation*}
and for all $p\in[P-1]$ we let
\begin{equation}\label{eq:def_eta_randomized}
    \eta_p:= \frac{\eta_P/\mu_P}{\mu^{P-p-1}}   \quad \text{where}\quad  \mu:= 1200 k\sqrt{\frac{kd}{l}}
\end{equation}
We then define the orthogonality tolerance parameters as follows
\begin{equation}\label{eq:def_delta_randomized}
    \delta_i^{(p)} := \frac{\eta_p(1-2/k)^{k-i}}{2}\sqrt{\frac{l_p}{\tilde d}},\quad p\in[P],i\in[k].
\end{equation}

As for the deterministic case, whenever the output of the oracle is not $\mb e$ nor $\mathsf{Success}$, we say that $\mb x$ is a depth-$p$ query where $p=\min\set{q\in[P], \tilde{\mb g}_{\mb V^{(q)}} (\mb x;\mb \delta^{(q)} ) \neq \mathsf{Success}}$. Except for the probing subspaces, all other parameters are kept constant throughout the feasibility problem. The depth-$p$ probing subspaces are resampled independently from the past every $T_p$ iterations, where the sequence $T_1,\ldots,T_P$ is defined as
\begin{equation*}
    T_p := \floor{\frac{k}{2}}  N^{p-1},\quad p\in[P],\quad \text{where} \quad N:=\floor{\frac{\tilde d}{2lk}}.
\end{equation*}
We are now ready to formally define our specific separation oracle for randomized algorithms. As in the oracle of Procedure~\ref{proc:feasibility}, we use the fallback separation oracle $\Ocal_{E_1,\ldots,E_P}$ when the one from Eq~\eqref{eq:def_multi_oracle_randomized} returns $\mathsf{Success}$. Having sampled i.i.d. $\tilde d$-dimensional subspaces $E_1,\ldots,E_P$, we independently construct for each $p\in[P]$ an i.i.d. sequence $(\mb V^{(p,a)})_{a\geq 0}$ of lists $\mb V^{(p,a)}=(V_1^{(p,a)},\ldots,V_k^{(p,a)})$ containing i.i.d. uniform $l_p$-dimensional random subspaces of $E_p$. To make the notations cleaner we assume that the number of iterations starts from $t=0$. We define the separation oracle for all $t\geq 0$ and $\mb x\in\Rbb^d$ via
\begin{equation}\label{eq:def_final_oracle_randomized}
    \tilde\Ocal_t (\mb x) := \begin{cases}
        \tilde \Ocal_{\mb V^{(1,\floor{t/T_1})} ,\ldots,\mb V^{(P,\floor{t/T_P})}} (\mb x) &\text{if } \tilde \Ocal_{\mb V^{(1,\floor{t/T_1})} ,\ldots,\mb V^{(P,\floor{t/T_P})}} (\mb x) \neq \mathsf{Success}\\
        \Ocal_{E_1,\ldots,E_P}(\mb x) &\text{otherwise}.
    \end{cases}
\end{equation}
This definition is stochastically equivalent to simply resampling the depth-$p$ probing subspaces $\mb V^{(p)}$ every $T_p$ iterations. We take this perspective from now on. In this context, a depth-$p$ period is simply a interval of time of the form $[aT_p,(a+1)T_p)$ for some integer $a\geq 0$. Here, the feasible set is defined as
\begin{equation*}
    \tilde Q_{E_1,\ldots,E_P} := B_d(\mb 0,1) \cap \set{ \mb x : \mb e^\top \mb x \leq -\frac{1}{2} } \cap \bigcap_{p\in[P]} \set{\mb x: \|\proj_{E_p}(\mb x)\| \leq \delta_1^{(p)}}.
\end{equation*}

We now prove query lower bounds for algorithms under this separation oracle using a similar methodology as for Procedure~\ref{proc:feasibility}. To do so, we first need to slightly adjust the notion of exploratory queries in this context. At the beginning of each depth-$p$ period, these are reset and we consider that there are no exploratory queries.

\begin{definition}[Exploratory queries, randomized case]
\label{def:explo_query_randomized}
    Let $p\in[P]$ and fix a depth-$p$ period $[aT_p,(a+1)T_p)$ for $a\in\{0,\ldots,N^{P-p}-1\}$. Given previous exploratory queries $\mb y_1^{(p)},\ldots,\mb y_{n_p}^{(p)}$ in this period, we say that $\mb x\in B_d(\mb 0,1)$ is a depth-$p$ exploratory query if
    \begin{enumerate}
        \item $\mb e^\top \mb x \leq -\frac{1}{2}$,
        \item the query passed all probes from levels $q < p$, that is $\mb g_{\mb V^{(q)}}(\mb x;\mb \delta^{(q)}) = \mathsf{Success}$,
        \item and it is robustly-independent from all previous depth-$p$ exploratory queries in the period
        \begin{equation}\label{eq:def_gamma_randomized}
            \|\proj_{Span(\mb y_r^{(p)},r\leq n_p)^\perp}(\mb x)\| \geq \gamma_p:=\frac{\delta_1^{(p)}}{4k}.
        \end{equation}
    \end{enumerate}
\end{definition}

We now introduce the feasibility game associated to the oracle which differs from the feasibility problem with the oracles $(\tilde\Ocal_t)_{t\geq 0}$ in the following ways.
\begin{itemize}
    \item The player can access some initial memory about the subspaces $E_1,\ldots,E_P$.
    \item Their goal is to perform $k$ depth-$P$ exploratory queries during a single depth-$P$ period of $T_P$ iterations. In particular, they only have a budget of $T_P$ calls to the oracle.
    \item They have some mild influence on the sequences of probing subspaces $\Vbb^{(p)}:=(\mb V^{(p,a)})_{a\geq 0}$ for $p\in[P]$ that will be used by the oracle. Precisely, for each $p\in[P]$, the oracle independently samples $J_P$ i.i.d. copies of these sequences $\Vbb^{(p,1)},\ldots,\Vbb^{(p,J_P)}$ for a pre-specified constant $J_P$. The player then decides on an index $\hat j\in [J_P]$ and the oracle uses the sequences $\Vbb^{(p,\hat j)}$ for $p\in[P]$ to simulate the feasibility problem.
\end{itemize}
The details of the game are given in Game~\ref{game:feasibility_game_randomized}. Note that compared to the feasibility Game~\ref{game:feasibility_game}, we do not need to introduce specific games at depth $p\in[P]$ nor introduce exploratory queries. The reason is that because they are non-adaptive, the oracles at depth $p'>p$ do not provide any information about $E_p$. Hence, the game at depth $p\in[P]$ can simply be taken as the original game but with $p$ layers instead of $P$.

\begin{game}[ht!]
        
\caption{Feasibility Game for randomized algorithms}\label{game:feasibility_game_randomized}

\setcounter{AlgoLine}{0}
\SetAlgoLined
\LinesNumbered

\everypar={\nl}

\hrule height\algoheightrule\kern3pt\relax
\KwIn{depth $P$; dimensions $d$, $\tilde d$, $l_1,\ldots,l_P$; $k$; $M$-bit memory randomized algorithm $alg$; resampling horizons $T_1,\ldots,T_P$; maximum index $J_P$}
\vspace{5pt}

\textit{Oracle:} Sample independently $E_1,\ldots, E_P$, uniform $\tilde d$-dimensional subspaces of $\Rbb^d$

\textit{Oracle:} For all $p\in[P]$ sample independently $J_P$ i.i.d. sequences $\Vbb^{(p,1)},\ldots,\Vbb^{(p,J_P)}$. Each sequence $\Vbb^{(p,j)}=(\mb V^{(p,j,a)})_{a\in[0,N^{P-p}) }$ contains $N^{P-p}$ i.i.d. $k$-tuples $\mb V^{(p,j,a)} = (V_1^{(p,j,a)},\ldots,V_k^{(p,j,a)})$ of i.i.d. $l_p$-dimensional subspaces of $E_p$.

\textit{Player:} Observe $E_1,\ldots,E_P$ and all sequences $\Vbb^{(p,1)},\ldots,\Vbb^{(p,J_P)}$ for $p\in[P]$. Based on these, submit to oracle an $M$-bit message $\mathsf{Message}$ and an index $\hat j \in[J_P]$

\textit{Oracle:} Initialize memory of $alg$ to $\mathsf{Message}$

\textit{Oracle:} \For{$t\in \{0,\ldots,T_P-1\}$}{
    Run $alg$ with current memory to get query $\mb x_t$

    \textbf{if} $t=0\bmod{T_p}$ \textbf{for} any $p\in[P]$ \textbf{then} $\mb V^{(p)}\gets \mb V^{(p,\hat j , t/T_p)}$
    
    \Return $\mb g_t = \tilde\Ocal_{\mb V^{(1)},\ldots,\mb V^{(P)}}(\mb x_t)$ as response to $alg$
}

Player wins if the player performed $k$ (or more) depth-$P$ exploratory queries

\hrule height\algoheightrule\kern3pt\relax
\end{game}

Roughly speaking, the first step of the proof is to show that because of the construction of the oracle $\tilde {\mb g}_{\mb V^{(p)}}$ for $p\in[P]$ in Eq~\eqref{eq:def_layer_oracle_randomized}, we have the following structure during any depth-$p$ period with high probability. (1) The algorithm observes $V_1^{(p)},\ldots,V_k^{(p)}$ in this exact order and further, (2) the algorithm needs to query a new robustly independent vector to observe a new probing subspace. These are exactly the properties needed to replace the step from Procedure~\ref{proc:feasibility} in which the adaptive oracle adapts to when exploratory queries are performed. Proving this property is one of the main technicality to extend our query lower bounds for deterministic algorithms to randomized algorithms.

\begin{lemma}\label{lemma:properties_oracle_construction}
    Fix $p\in[P]$ and $a\in\{0,\ldots,N^{P-p}-1\}$. Suppose that $l_p\geq Ck^3\ln d$ for some universal constant $C>0$ and $k\geq 3$. Then, with probability at least $1-k^2 J_P e^{-k\ln d}$, for all times $t\in [aT_p,(a+1)T_p)$ during Game~\ref{game:feasibility_game_randomized}, the following hold. Let $r_p(t)$ be the number of depth-$p$ exploratory queries in $[aT_p,t]$. If $r_p(t) \leq k$,
    \begin{itemize}
        \item the response $\mb g_t$ is consistent to the oracle without probing subspaces $V_i^{(p)}$ for $i>r_p(t)$, that is, replacing $\tilde{\mb g}_{\mb V^{(p)}}(\cdot;\mb\delta^{(p)})$ with $\tilde{\mb g}_{V_1^{(p)},\ldots,V_{r_p(t)}^{(p)}} (\cdot;\mb\delta^{(p)}) $ in Eq~\eqref{eq:def_multi_oracle_randomized},
        \item if $\mb x_t$ is a depth-$p'$ query for $p'>p$, then $\|\proj_{E_p}(\mb x_t)\| \leq \eta_p$.
    \end{itemize}
\end{lemma}

Here, the first bullet point exactly proves the behavior that was described above for the sequential discovery of probing subspaces. Roughly speaking, the second bullet point shows that periods are still mostly proper, at least before $k$ exploratory queries are performed during the period.

\vspace{3mm}

\begin{proof}[of \cref{lemma:properties_oracle_construction}]
    Fix $p\in[P]$ and the period index $a<N^{P-p}$. We will use the union bound to take care of the degree of liberty $\hat j\in[J_P]$. For now, fix $j\in[J_P]$ and suppose that we had $\hat j=j$, that is, all probing subspaces were constructed from the sequences $\Vbb^{(1,j)},\ldots,\Vbb^{(P,j)}$.

    Let $i\in\{1,\ldots,k\}$ and consider the game for which the oracle responses are constructed exactly as in Game~\ref{game:feasibility_game_randomized} except during the considered period $[aT_p,(a+1)T_p)$ as follows. For the oracle response at time $t\in [aT_p,(a+1)T_p) $: if $r_p(t)\geq i$ we use the same oracle as in Game~\ref{game:feasibility_game_randomized}; but if $r_p(t) < i$, we replace $\mb V^{(p)}$ with $(V_1^{(p)},\ldots,V_{r_p(t)}^{(p)})$ -- that is we ignore all depth-$p$ probing subspaces $V_j^{(p)}$ with index $j>r_p(t)$. For convenience, let us refer to this as $\text{Game}(p,a;i)$. Note that $\text{Game}(p,a;1)$ is exactly Game~\ref{game:feasibility_game_randomized}. Indeed, at a given time $t$ in the period, if there were no previous depth-$p$ exploratory queries either (1) $\mb x_t$ does not pass probes at level $q<p$ or $\mb e^\top \mb x>-\frac{1}{2}$: in this case no depth-$p$ probing subspaces are needed to construct the response; or (2) we have in particular $\|\mb x\| \geq \frac{1}{2} \geq \gamma_p$ so $\mb x_t$ is exploratory. As a result, before the first depth-$p$ exploratory query, having access to the depth-$p$ subspaces is irrelevant.

    Now fix $i\in [k-1]$. Our goal is to show that with high probability, the responses returned by $\text{Game}(p,a;i)$ are equivalent to those of $\text{Game}(p,a;i+1)$. Let $\Ecal_i$ be the event that there were at least $i$ depth-$p$ exploratory queries during the period $[aT_p,(a+1)T_p)$. We recall the notations $\mb y_1^{(p)},\mb y_2^{(p)},\ldots$ for these exploratory queries. Note that we slightly abuse of notations because these are the exploratory queries for $\text{Game}(p,a;i)$ (not for Game~\ref{game:feasibility_game_randomized}). However, our goal is to show that their responses coincide so under this event, all these exploratory queries will coincide. Importantly, by construction of the oracle for $\text{Game}(p,a;i)$, all queries $\mb y_1^{(p)},\ldots,\mb y_i^{(p)}$ are independent from the subspaces $V_s^{(p)} = V^{(p,j,a)}_s$ for all $s\geq i$. Indeed, during the period $[aT_p,(a+1)T_{p+1})$, before receiving the response for the query $\mb y_i^{(p)}$, the oracle only used probing subspaces $V_1^{(p)},\ldots,V_{i-1}^{(p)}$. Hence, $\mb y_1^{(p)},\ldots,\mb y_i^{(p)}$ are dependent only on those probing subspaces. As a result, from the point of view of the spaces $V_j^{(p)}$ for $j\geq i$, the subspace $F_i:=Span(\proj_{E_p}(\mb y_s^{(p)}),s\in[i])$ is uniformly random. Formally, we define the following event
    \begin{equation*}
        \Fcal_i = \bigcap_{j=i}^k \set{ \paren{1-\frac{1}{2k}}\|\mb x\| \leq  \sqrt{\frac{\tilde d}{l_p} } \|\proj_{V_j}(\mb x)\| \leq \paren{1+\frac{1}{2k}} \|\mb x\|,\; \forall \mb x\in F_i  }.
    \end{equation*}
    By \cref{lemma:concentration_projecting_subspaces} and the union bound, we have
    \begin{equation*}
        \Pbb(\Fcal_i \mid \Ecal_i) \geq 1-k\exp\paren{i\ln\frac{2C\tilde dk}{l_p} - \frac{l_p}{2^7 k^2}}\geq 1- k\exp\paren{k\ln d - \frac{l_p}{2^7k^2}} \geq 1-ke^{-k\ln d}.
    \end{equation*}
    Here we used $l_p\geq 2^8Ck^2\ln d$ and the fact that $\dim(F_i)\leq i\leq k$. For convenience let $\zeta= \frac{1-1/(2k)}{1+1/(2k)}$. Using the previous bounds, under $\Ecal_i\cap \Fcal_i$, we have for any $\mb y\in Span(\mb y_s^{(p)},p\in[i])$,
    \begin{equation*}
        \|\proj_{V_i}(\mb y)\| \geq  \zeta\|\proj_{V_j}(\mb y)\|, \quad j\in\{i+1,\ldots,k\}.
    \end{equation*}
    Now consider any time during the period $t\in[aT_p,(a+1)T_{p+1})$ such that $n_p(t) = i$. If $\mb x_t$ was not a depth-$p'$ query with $p'\geq p$, knowing the depth-$p$ probing subspaces is irrelevant to construct the oracle response. Otherwise, we have
    \begin{equation}\label{eq:non_robustly_independent}
        \|\proj_{Span(\mb y_s^{(p)},s\in[i])^\perp}(\mb x_t)\| <\gamma_p.
    \end{equation}
    Indeed, either $\mb x_t$ was a depth-$p$ exploratory query and hence $\mb x_t = \mb y_i^{(p)}$, or it was not, in which case the third property from \cref{def:explo_query_randomized} must not be satisfied (because the two first are already true since $\mb x_t$ is a depth-$p'$ query with $p'\geq p$). For simplicity, write $\mb y_t = \proj_{Span(\mb y_s^{(p)},s\in[i])}(\mb x_t)$ and note that $\|\mb x_t-\mb y_t\|<\gamma_p$ by Eq~\eqref{eq:non_robustly_independent}. Then, for any $i<j\leq k$,
    \begin{align*}
        \|\proj_{V_i}(\mb x_t)\| > \|\proj_{V_i}(\mb y_t)\| - \gamma_p \geq \zeta  \|\proj_{V_j}(\mb x_t)\| - \gamma_p > \zeta \|\proj_{V_j}(\mb x_t)\| - (1+\zeta) \gamma_p.
    \end{align*}
    In particular, if one has $j\in\Ical_{V_1,\ldots,V_k}(\mb x_t;\mb \delta^{(i)})$ for some $i+1<j\leq k$, we have
    \begin{equation*}
        \|\proj_{V_i}(\mb x_t)\| > \zeta \delta_j^{(p)} - 2\gamma_p \geq \zeta \delta_{i+1}^{(p)} - \frac{\delta_1^{(p)}}{2k} \geq \delta_i^{(p)} \paren{\frac{\zeta}{1-2/k}-\frac{1}{2k}} \geq \delta_i^{(p)}.
    \end{equation*}
    In the second inequality, we use the definition of $\gamma_p$ in Eq~\eqref{eq:def_gamma_randomized} and in the last inequality, we used $k\geq 3$.
    As a result, we also have $i\in \Ical_{V_1,\ldots,V_k}(\mb x_t;\mb \delta^{(i)})$. Going back to the definition of the depth-$p$ oracle in Eq~\eqref{eq:def_layer_oracle_randomized} shows that in all cases (whether there is some $j\in\Ical_{V_1,\ldots,V_k}(\mb x_t;\mb \delta^{(i)})$ for $i+1<j\leq k$ or not),
    \begin{equation*}
        \tilde {\mb g}_{\mb V^{(p)}}(\mb x_t;\mb\delta^{(p)}) = \tilde {\mb g}_{V_1^{(p)},\ldots, V_i^{(p)}}(\mb x_t;\mb\delta^{(p)}),
    \end{equation*}
    that is, the oracle does not use the subspaces $V_j^{(p)}$ for $j>i$. In summary, under $\Ecal_i\cap\Fcal_i$ the responses provided in $\text{Game}(p,a;i)$ and $\text{Game}(p,a;i+1)$ are identical.

    Using the above result recursively shows that under
    \begin{equation*}
        \Gcal:=\bigcap_{i\in[k-1]}\Ecal_i^c \cup (\Ecal_i\cap\Fcal_i),
    \end{equation*}
    the responses in $\text{Game}(p,a;k)$ are identical to those in $\text{Game}(p,a;1)$ which is the original Game~\ref{game:feasibility_game_randomized} provided $\hat j=j$. Hence, under $\Gcal$ and assuming $\hat j=j$, the first claim of the lemma holds.
    The second point is a direct consequence of the previous property. For any fixed $t\in[aT_p,(a+1)T_{p+1})$ write $i=n_p(t)$. Under $\Gcal$, because the event $\Ecal_{n_p(t)}\cap\Fcal_{n_p(t)}$ holds, we have in particular
    \begin{equation*}
        \sqrt{\frac{\tilde d}{l_p}}\|\proj_{V_i}(\mb x_t)\| \geq \paren{1-\frac{1}{2k}} \|\proj_{E_p}(\mb x_t)\|.
    \end{equation*}
    Hence, if $\mb x_t$ is a depth-$p'$ query for $p'>p$, we must have $\|\proj_{V_i}(\mb x_t)\| \leq \delta_i^{(p)}$, which in turns gives
    \begin{equation*}
        \|\proj_{E_p}(\mb x_t)\| \leq \frac{\delta_i^{(p)}}{1-\frac{1}{2k}}\sqrt{\frac{\tilde d}{l_p}} \leq 2\delta_k^{(p)}\sqrt{\frac{\tilde d}{l_p}} = \eta_p.
    \end{equation*}
    Hence, under $\Gcal$, all claims hold and we have
    \begin{equation*}
        \Pbb(\Gcal) \geq 1-\sum_{i\in[k-1]}\Pbb(\Fcal_i^c\mid \Ecal_i) \geq 1-k^2e^{-k\ln d}.
    \end{equation*}

    We now recall that all the previous discussion was dependent on the choice of $j\in[J_P]$. Taking the union bound over all these choices shows that all claims from the lemma hold with probability at least $1-k^2J_P e^{-k\ln d}$.
\end{proof}

\subsection{Query lower bounds for an Adapted Orthogonal Subspace Game}

By \cref{lemma:properties_oracle_construction}, to receive new information about $E_p$, the algorithm needs to find robustly-independent queries. We next show that we can relate a run of Game~\ref{game:feasibility_game_randomized} to playing an instance of some orthogonal subspace game. However, the game needs to be adjusted to take into account the degree of liberty from $\hat j\in[J_P]$. This yields following Adapted Orthogonal Subspace Game~\ref{game:adapted_orthogonal_subspace_game}.

\begin{game}[h!]

\caption{Adapted Orthogonal Subspace Game}\label{game:adapted_orthogonal_subspace_game}

\setcounter{AlgoLine}{0}

\SetAlgoLined
\LinesNumbered

\everypar={\nl}

\hrule height\algoheightrule\kern3pt\relax

\KwIn{dimensions $d$, $\tilde d$; memory $M$; number of robustly-independent vectors $k$; number of queries $m$; parameters $\beta$, $\gamma$; maximum index $J$}

\textit{Oracle:} Sample a uniform $\tilde d$-dimension linear subspace $E$ in $\Rbb^d$ and i.i.d. vectors $\mb v_r^{(j)}\overset{i.i.d.}{\sim} \Ucal(S_d\cap E)$ for $r\in[m]$ and $j\in[J]$ 

\textit{Player:} Observe $E$ and $\mb v_r^{(j)}$ for all $r\in[m]$ and $j\in[J]$. Based on these, store an $M$-bit message $\mathsf{Message}$ and an index $\hat j\in[J]$

\textit{Oracle:} Send samples $\mb v_1^{(\hat j)},\ldots,\mb v_m^{(\hat j)}$ to player

\textit{Player:} Based on $\mathsf{Message}$ and $\mb v_1^{(\hat j)},\ldots,\mb v_m^{(\hat j)}$ only, return unit norm vectors $\mb y_1,\ldots, \mb y_k$

The player wins if for all $i\in[k]$
\begin{enumerate}
    \item $\| \proj_E(\mb y_i)\| \leq \beta$
    \item $\|\proj_{Span(\mb y_1,\ldots,\mb y_{i-1})^\perp}(\mb y_i)\| \geq \gamma$.
\end{enumerate}
\hrule height\algoheightrule\kern3pt\relax
\end{game}

We first prove a query lower bound for Game~\ref{game:adapted_orthogonal_subspace_game} similar to \cref{thm:memory_query_lower_bound}.

\begin{theorem}\label{thm:adapted_memory_query_lower_bound}
    Let $d\geq 8$, $C\geq 1$, and $0<\beta,\gamma\leq 1$ such that $\gamma/\beta \geq 3e \sqrt{kd/\tilde d}$. Suppose that $\frac{\tilde d}{4}\geq k \geq 50C\frac{M+2\log_2 J+3}{\tilde d}\ln\frac{\sqrt d}{\gamma}$. If the player wins at the Adapted Orthogonal Subspace Game~\ref{game:adapted_orthogonal_subspace_game} with probability at least $1/C$, then $m> \frac{\tilde d}{2}$.
\end{theorem}

\begin{proof}
    Fix parameters satisfying the conditions of the lemma and suppose $m \leq \frac{\tilde d}{2}$. In the previous proof for \cref{thm:memory_query_lower_bound}, we could directly reduce the Orthogonal Vector Game~\ref{game:orthogonal_subspace_game} to a simplified version (Game~\ref{game:simplified_orthogonal_subspace_game}) in which the query vectors $\mb v_1,\ldots,\mb v_m$ are not present anymore. We briefly recall the construction, which is important for this proof as well. Let $E$ be a uniform $\tilde d$-dimensional subspace of $\Rbb^d$ and $\mb w_1,\ldots,\mb w_m\overset{i.i.d.}{\sim} \Ucal(S_{d-1}\cap E)$. For any $m$-dimensional subspace $V$, let $\Dcal(V)$ be the distribution of $\mb w_1,\ldots,\mb w_m$ conditional on $V=Span(\mb w_i,i\in[m])$. We now propose an alternate construction for the subspace $E$ and the vectors $\mb w_1,\ldots,\mb w_m$. Instead, start by sampling a uniform $m$-dimensional subspace $V$, then sample $(\mb w_1,\ldots,\mb w_m)\sim\Dcal(V)$. Last, let $F$ be a uniform $(\tilde d-m)$-dimensional subspace of $\Rbb^d$ independent of all past random variables. We then pose $E:= V\oplus F$. We can easily check that the two distributions are equal, and as a result, we assumed without loss of generality that the samples $\mb v_1,\ldots,\mb v_m$ were sampled in that manner. This is convenient because the space $F$ is independent from $\mb v_1,\ldots,\mb v_m$. Unfortunately, this will not be the case in the present game for the samples $\mb v_1^{(\hat j)},\ldots,\mb v_k^{(\hat j)}$ because the samples also depend on the index $\hat j\in[J]$.
    
    We slightly modify the construction to account for the $J$ different batches of samples.
    \begin{enumerate}
        \item Sample $E$ a uniform $\tilde d$-dimensional subspace of $\Rbb^d$
        \item Independently for all $j\in[J]$, sample a $m$-dimensional subspace $V^{(j)}$ of $E$ and sample $(\mb v_1^{(j)},\ldots,\mb v_k^{(j)})\sim\Dcal(V^{(j)})$. independently from these random variables, also sample $F^{(j)}$ a uniform $(\tilde d-m)$-dimensional subspace of $E$.
    \end{enumerate}
    We can check that the construction of the vectors $\mb v_i^{(j)}$ for $i\in[k]$ and $j\in[J]$ is equivalent as that in line 1 of Game~\ref{game:adapted_orthogonal_subspace_game}. Further, note that for any fixed $j\in[J]$, the subspaces $V^{(j)},F^{(j)}$ are independent (because $E$ was also uniformly sampled) and sampled according to uniform $m$-dimensional (resp. $\tilde d-m$-dimensional) subspaces of $\Rbb^d$. Also, on an event $\Fcal_i$ of probability one, $E=V^{(i)}\oplus F^{(i)}$. In particular, their mutual information is zero. We now bound their mutual information for the selected index $\hat j\in[J]$ and aim to show that it is at most $\Ocal(\ln J)$. By definition,
    \begin{align*}
        I(\mb V^{(\hat j)};F^{(\hat j)}) &= \int_{\mb v}\int_f p_{\mb V^{(\hat j)},F^{(\hat j)}}(\mb v,f) \ln \frac{p_{\mb V^{(\hat j)},F^{(\hat j)}}(\mb v,f)}{p_{\mb V^{(\hat j)}}(\mb v) p_{F^{(\hat j)}}(f)} d\mb v df.
    \end{align*}
    Since for any $j\in[J]$, $\mb V^{(j)}$ and $F^{(j)}$ are independent, we have
    \begin{align*}
        p_{\mb V^{(\hat j)},F^{(\hat j)}}(\mb v,f) = \sum_{j\in[J]}p_{\mb V^{(j)},F^{(j)},\hat j}(\mb v,f,j) 
        & = \sum_{j\in[J]} p_{\mb V,F}(\mb v,f) \Pbb(\hat j=j\mid \mb V^{(j)}=\mb v,F^{(j)}=f)\\
        &=p_{\mb V}(\mb v) p_F(f)\sum_{j\in[J]}\Pbb(\hat j=j\mid \mb V^{(j)}=\mb v,F^{(j)}=f).
    \end{align*}
    Similarly, we have
    \begin{align}
        p_{\mb V^{(\hat j)}}(\mb v) &= p_{\mb V}(\mb v) \sum_{j\in[J]} \Pbb(\hat j=j\mid \mb V^{(j)}=\mb v)\label{eq:property_marginal_V_after_choice} \\ 
        p_{F^{(\hat j)}}(f) &= p_F(f) \sum_{j\in[J]} \Pbb(\hat j=j\mid F^{(j)}=f). \label{eq:property_marginal_F_after_choice}
    \end{align}
    For convenience we introduce the following notations
    \begin{align*}
        P_{\mb V, F}(\mb v, f) := \sum_{j\in[J]}\Pbb(\hat j=j\mid \mb V^{(j)}=\mb v,F^{(j)}=f)\;\; \text{and}\;\; \begin{cases}
            P_{\mb V}(\mb v):= \sum_{j\in[J]} \Pbb(\hat j=j\mid \mb V^{(j)}=\mb v)\\
            P_F(f) :=\sum_{j\in[J]} \Pbb(\hat j=j\mid F^{(j)}=f).
        \end{cases}
    \end{align*}
    Putting the previous equations together gives
    \begin{align*}
        I(\mb V^{(\hat j)};F^{(\hat j)}) =\int_{\mb v}\int_f p_{\mb V^{(\hat j)},F^{(\hat j)}}(\mb v,f)  \ln \frac{P_{\mb V,F}(\mb v, f)} { P_{\mb V}(\mb v)P_F(f)}  d\mb v df.
    \end{align*}
    We decompose the logarithmic term on the right-hand side and bound each corresponding term separately. We start with the term involving $P_{\mb V,F}(\mb v, f)$.
    Define the random variable $\tilde j(\mb v, f)$ on $[J]$ that has $\Pbb(\tilde j(\mb v, f) =j) = \Pbb(\hat j=j\mid \mb V^{(j)}=\mb v,F^{(j)}=f)/P_{\mb V,F}(\mb v, f)$. Because $H(\tilde j(\mb v, f))\leq \ln J$ for all choices of $\mb v, f$ we obtain
    \begin{align*}
    &\int_{\mb v}\int_f p_{\mb V^{(\hat j)},F^{(\hat j)}}(\mb v,f)  \ln P_{\mb V,F}(\mb v, f)  d\mb v df\\
    &=\int_{\mb v}\int_f p_{\mb V}(\mb v) p_F(f)\sum_{j\in[J]} \Pbb(\hat j=j\mid \mb V^{(j)}=\mb v,F^{(j)}=f)
        \ln P_{\mb V,F}(\mb v, f) d\mb v df \\
        &\leq \int_{\mb v}\int_f p_{\mb V}(\mb v) p_F(f)\sum_{j\in[J]} \Pbb(\hat j=j\mid \mb V^{(j)}=\mb v,F^{(j)}=f)
        \ln \frac{P_{\mb V,F}(\mb v, f)} { \Pbb(\hat j=j\mid \mb V^{(j)}=\mb v,F^{(j)}=f)}d\mb v df \\
        &\leq \ln J \int_{\mb v}\int_f p_{\mb V}(\mb v) p_F(f)\sum_{j\in[J]} \Pbb(\hat j=j\mid \mb V^{(j)}=\mb v,F^{(j)}=f)
       d\mb v df = \ln J.
    \end{align*}

    We next turn to the term involving $P_{\mb V}(\mb v)$. Since $x\ln\frac{1}{x}\leq 1/e$ for all $x\geq 0$, we have directly
    \begin{align*}
         \int_{\mb v}\int_f p_{\mb V^{(\hat j)},F^{(\hat j)}}(\mb v,f) 
        \ln \frac{1} { P_{\mb V}(\mb v)}d\mb v df &= \int_{\mb v} p_{\mb V^{(\hat j)}}(\mb v)  \ln \frac{1} { P_{\mb V}(\mb v)}d\mb v \\
        &= \int_{\mb v}p_{\mb V}(\mb v) P_{\mb V}(\mb v)   \ln \frac{1} { P_{\mb V}(\mb v)}d\mb v \leq \frac{1}{e}.
    \end{align*}
    We similarly get the same bound for the term involving $P_F(f)$. Putting these estimates together we finally obtain
    \begin{equation}\label{eq:low_mutual_information}
        I(\mb V^{(\hat j)};F^{(\hat j)}) \leq \ln J + \frac{2}{e}\leq \ln J +1.
    \end{equation}

    From now, we use similar arguments as in the proof of \cref{lemma:query_lower_bound_simplified_game}. We fix $s=1+\ln\frac{\sqrt d}{\gamma}$. Denote by $\Ecal$ the event when the player wins and denote by $\mb Y=[\mb y_1,\ldots,\mb y_k]$ the concatenation of the vectors output by the player. Using \cref{lemma:gram-schmidt_marsden} and the same arguments as in \cref{lemma:query_lower_bound_simplified_game}, we construct from $\mb Y$ an orthonormal sequence $\mb Z = [\mb z_1,\ldots,\mb z_r]$ with $r=\ceil{k/s}$ such that on the event $\Ecal$, for all $i\in[r]$,
    \begin{equation*}
         \|\proj_E(\mb z_i)\| \leq \frac{e\beta\sqrt k}{\gamma} \leq \frac{1}{3} \sqrt{\frac{\tilde d}{d}}.
    \end{equation*}
    In particular, we have
    \begin{equation}\label{eq:constraints_randomized}
        \|\proj_{F^{(\hat j)}}(\mb z_i)\| \leq \|\proj_E(\mb z_i)\| \leq \frac{1}{3} \sqrt{\frac{\tilde d}{d}},\quad i\in[r]
    \end{equation}
    
    We now give both an upper and lower bound on $ I(F^{(\hat j)};\mb Y)$ by adapting the arguments from \cref{lemma:query_lower_bound_simplified_game}. The data processing inequality gives
    \begin{align*}
        I(F^{(\hat j)};\mb Y) \leq I(F^{(\hat j)};\mathsf{Message},\mb V^{(\hat j)}) &= I(F^{(\hat j)};\mb V^{(\hat j)}) + I(F^{(\hat j)};\mathsf{Message}\mid \mb V^{(\hat j)})\\
        &\leq M\ln 2+ \ln J +1.
    \end{align*}
    In the last inequality, we used the fact that $\mathsf{Message}$ is encoded in $M$ bits (to avoid continuous/discrete issues with the mutual information, this step can be fully formalized as done in the proof of \cref{lemma:query_lower_bound_simplified_game}) and Eq~\eqref{eq:low_mutual_information}.

    We next turn to the lower bound. The same arguments as in \cref{lemma:query_lower_bound_simplified_game} give
    \begin{equation*}
         I(F^{(\hat j)};\mb Y) \geq \Pbb(\Ecal)\Ebb_{\Ecal}\sqb{I(F^{(\hat j)};\mb Z \mid \Ecal)} - \ln 2.
    \end{equation*}
    Next, denote by $\Ccal$ the set of $(\tilde d-m)$-dimensional subspaces $F$ compatible with Eq~\eqref{eq:constraints_randomized}, that is
    \begin{equation*}
        \Ccal:=\Ccal(\mb Z) = \set{(\tilde d-m)\text{-dimensional subspace $F$ of }\Rbb^d: \|\proj_F(\mb z_i)\|\leq \frac{1}{3}\sqrt{\frac{\tilde d}{d}},i\in[r] }.
    \end{equation*}
    The same arguments as in \cref{lemma:query_lower_bound_simplified_game} (see Eq~\eqref{eq:dk_lower_bound}) show that
    \begin{equation*}
        I(F^{(\hat j)};\mb Z \mid \Ecal) \geq \Ebb_{\Ccal \mid \Ecal} \sqb{\ln\frac{1}{\Pbb(\hat F^{(j)} \in\Ccal)}} + \ln\Pbb(\Ecal).
    \end{equation*}
    The rest of the proof of \cref{lemma:query_lower_bound_simplified_game} shows that letting $G$ be a random uniform $(\tilde d-m)$-dimensional subspace of $\Rbb^d$, for any realization of $\mb Z$, we have
    \begin{equation*}
        \Pbb(G \in\Ccal(\mb Z)) \leq e^{-(\tilde d-m)r/16}.
    \end{equation*}
    While $F^{(\hat j)}$ is not distributed as a uniform $(\tilde d-m)$-dimensional subspace of $\Rbb^d$, we show that is close from it. Indeed, for any choice of $(\tilde d-m)$-dimensional subspace $f$, by Eq~\eqref{eq:property_marginal_F_after_choice} we have
    \begin{equation*}
        p_{F^{(\hat j)}}(f) = p_F(f) \sum_{j\in[J]} \Pbb(\hat j=j\mid F^{(j)}=f) \leq J p_F(f),
    \end{equation*}
    where $F\sim G$ is distributed a uniform $(\tilde d-m)$-dimensional subspace of $\Rbb^d$. Thus, for any realization of $\mb Z$, we obtain
    \begin{equation*}
        \Pbb(F^{(\hat j)} \in \Ccal(\mb Z)) \leq J \cdot \Pbb(F \in \Ccal(\mb Z)) = J \cdot \Pbb(G \in \Ccal(\mb Z)) \leq J e^{-(\tilde d-m)r/16}.
    \end{equation*}
    Putting together all the previous equations, we obtained
    \begin{align*}
        M\ln 2 + \ln J + 1\geq I(F^{(\hat j)};\mb Y) &\geq \Pbb(\Ecal) \paren{\frac{(\tilde d-m)r}{16}  -\ln J} - \Pbb(\Ecal)\ln\frac{1}{\Pbb(\Ecal)} - \ln 2\\
        &\geq  \Pbb(\Ecal)\frac{\tilde d r}{32}  - \ln J- \ln 2 -\frac{1}{e}.
    \end{align*}
    As a result, using $d\geq 8$ so that $s\leq 2\ln\frac{\sqrt d}{\gamma}$, we have
    \begin{equation*}
        \Pbb(\Ecal) \leq 32\frac{M\ln 2+2\ln J +3\ln 2}{\tilde d r} \leq 50\frac{M+2\log_2 J +3}{\tilde d k} \ln\frac{\sqrt d}{\gamma}.
    \end{equation*}
    In the last inequality, we used the assumption on $k$.
    We obtain a contradiction, which shows that $m\geq \frac{\tilde d}{2}$.
\end{proof}

\subsection{Recursive lower bounds for the feasibility game and feasibility problems}

We can now start the recursive argument to give query lower bounds. Precisely, we relate the feasibility Game~\ref{game:feasibility_game_randomized} to the Adapted Orthogonal Subspace Game~\ref{game:adapted_orthogonal_subspace_game}. The intuition is similar as in the deterministic case except for a main subtlety. We do not restart depth-$(P-1)$ periods once $k$ exploratory queries have been performed. As a result, it may be the case that for such a period after having found $k$ exploratory queries, the algorithm gathers a lot of information on the next layer $P$ without having to perform new exploratory queries. This is taken into account in the following result.

\begin{lemma}\label{lemma:recursive_randomized}
    Let  $d\geq 8$ and $\zeta\geq 1$. Suppose that $\frac{\tilde d}{4l_{P-1}} \geq k \geq 100\zeta \frac{M+2\log_2 J_P + 3}{\tilde d} \ln\frac{\sqrt d}{\gamma_P}$, that $k\ln d \geq \ln(2\zeta k^2 J_P(N+1))$, and $k\geq 3$. Also, suppose that $l_P,l_{P-1} \geq Ck^3\ln d$. If that there exists a strategy for Game~\ref{game:feasibility_game_randomized} with $P$ layers and maximum index $J_P$ that wins with probability at least $q$, then there exists a strategy for Game~\ref{game:feasibility_game_randomized} with $P-1$ layers that wins with probability at least $q-\frac{1}{2\zeta}$ with same parameters as the depth-$P$ game for $(d,\tilde d,l,k,M,T_1,\ldots,T_{P-1})$ and maximum index $J_{P-1}=NJ_P$.
\end{lemma}

\begin{proof}
    Fix $P\geq 2$ and a strategy for Game~\ref{game:feasibility_game_randomized} with $P$ layers. Within this proof, to simplify the notations, we will write $l$ instead of $l_{P-1}$. In fact, this is consistent with the parameters that were specified at the beginning of the section ($l_p=l$ for all $p\in[P-1]$). Using this strategy, we construct a strategy for Game~\ref{game:adapted_orthogonal_subspace_game} for the parameters $m=Nlk\leq \frac{\tilde d}{2}$ and $J=J_P$ and memory limit $M+\log_2 J +1$.

    The strategy for the orthogonal subspace game is described in \cref{alg:strategy_adapted_orthogonal_subspace_game}. Similarly as the strategy constructed for the deterministic case (\cref{alg:strategy_orthogonal_subspace_game}), it uses the samples $\mb v_1,\ldots,\mb v_{lkN}$ provided by the oracle to construct the depth-$(P-1)$ probing subspaces using $l$ new vectors for each new subspace. The only subtlety is that in line 2 of Game~\ref{game:adapted_orthogonal_subspace_game} the player needs to select the index $\hat j\in[J]$ guiding the samples that will be received in line 4 of Game~\ref{game:adapted_orthogonal_subspace_game}. The knowledge of $\hat j$ is necessary to resample the probing subspaces for depths $p\neq P-1$ (see line 5 of \cref{alg:strategy_adapted_orthogonal_subspace_game}) hence we also add it to the message. The message can therefore be encoded into $M + \ceil{\log_2 J} \leq M+\log_2 J +1 $ bits.

            \begin{algorithm}[ht!]
        
\caption{Strategy of the Player for the Adapted Orthogonal Subspace Game~\ref{game:adapted_orthogonal_subspace_game}}\label{alg:strategy_adapted_orthogonal_subspace_game}

\setcounter{AlgoLine}{0}
\SetAlgoLined
\LinesNumbered

\everypar={\nl}

\hrule height\algoheightrule\kern3pt\relax
\KwIn{depth $P$, dimensions $d$, $\tilde d$, number of vectors $k$, $M$-bit algorithm $alg$ for Game~\ref{game:feasibility_game_randomized} with $P$ layers; $T_P$; maximum index $J_P$}

\vspace{5pt}

{\nonl \textbf{Part 1:}} Construct the message and index\,

For all $p\in[P]\setminus\{P-1\}$, sample independently $E_{p}$ a uniform $\tilde d$-dimensional linear subspaces in $\Rbb^d$, and $J_P$ i.i.d. sequences $\Vbb^{(p,1)},\ldots,\Vbb^{(p,J_P)}$ as in line 2 of Game~\ref{game:feasibility_game_randomized}

Observe $E$ and vectors $\mb v_r^{(j)}$ for $r\in[Nlk]$ and $j\in[J_P]$. Set $E_{p-1}=E$ and for all $j\in[J_P]$ let $\Vbb^{(P-1,j)} := (\mb V^{(P-1,j,a)})_{a\in[0,N)}$ where $\mb V^{(P-1,j,a)}:= (Span(\mb v_{alk + (i-1)l + s}^{(j)},s\in[l]) )_{i\in[k]}$ for $a\in\{0,\ldots,N-1\}$. Given all previous information, construct the $M$-bit message $\mathsf{Memory}$ and the index $\hat j\in[J_P]$ as in line 3 of Game~\ref{game:feasibility_game_randomized}

Submit to the oracle the message $\mathsf{Message}=(\mathsf{Memory},\hat j)$ and the index $\hat j$

\vspace{5pt}

{\nonl \textbf{Part 2:}} Simulate run of Game~\ref{game:feasibility_game_randomized}

Receive $\mathsf{Message}=(\mathsf{Memory},\hat j)$ and samples $\mb v_1^{(\hat j)},\ldots,\mb v_{Nlk}^{(\hat j)}$ from Oracle. Based on the samples, construct $\Vbb^{(p,\hat j)}$ from these samples as in Part 1

For $p\in[P]\setminus\{P-1\}$, resample $E_p$ and $\Vbb^{(p,\hat j)}$ using same randomness as in Part 1 and knowledge of $\hat j$. 

Initialize memory of $alg$ to $\mathsf{Memory}$ and run $T_P$ iterations of the feasibility problem using $\Vbb^{(p,\hat j)}$ for $p\in[P]$ as in lines 5-9 of Game~\ref{game:feasibility_game_randomized}

\lIf{there were less than $k$ depth-$P$ exploratory queries}{Strategy fails; \textbf{end}}

\Return normalized depth-$P$ exploratory queries $\frac{\mb y_1^{(P)}}{\|\mb y_1^{(P)}\|},\ldots,\frac{\mb y_k^{(P)} }{\|\mb y_k^{(P)}\|}$

\hrule height\algoheightrule\kern3pt\relax
\end{algorithm}

We next define some events under which we will show that \cref{alg:strategy_adapted_orthogonal_subspace_game} wins at Game~\ref{game:adapted_orthogonal_subspace_game}. We introduce the event in which the algorithm makes at most $k$ exploratory queries for any of the $N$ depth-$(P-1)$ periods:
    \begin{equation*}
        \Ecal = \bigcap_{a<N} \set{ \text{at most $k$ depth-$(P-1)$ exploratory queries during period }[aT_{P-1},(a+1)T_{P-1})}.
    \end{equation*}
    Next, by \cref{lemma:properties_oracle_construction} and the union bound, on an event $\Fcal$ of probability at least $1-k^2 J_P (N+1) e^{-k\ln d}$, the exploration of the probing subspaces for the $N$ depth-$(P-1)$ periods and the only depth-$P$ period, all satisfy the properties listed in \cref{lemma:properties_oracle_construction}. In particular, under $\Ecal\cap\Fcal$, all depth-$(P-1)$ periods are proper in the following sense: if $\mb x_t$ for $t\in[T_P]$ is a depth-$P$ query then because there were at most $r_{P-1}(t)\leq k$ exploratory queries in the corresponding depth-$(P-1)$ period, we have
    \begin{equation}\label{eq:passed_probes_P-1}
        \|\proj_{E_{P-1}}(\mb x_t)\| \leq \eta_{P-1}.
    \end{equation}
    Finally, let $\Gcal$ be the event on which the strategy wins. Suppose that $\Ecal\cap\Fcal\cap\Gcal$ is satisfied. Because the strategy wins, we know that the oracle used the last depth-$P$ probing subspace $V_k^{(P)}$. Because $\Fcal$ is satisfied, in turn, this shows that there were at least $k$ depth-$P$ exploratory queries. In particular, the strategy from \cref{alg:strategy_adapted_orthogonal_subspace_game} does not fail. We recall that exploratory queries $\mb y_i^{(P)}$ for $i\in[k]$ must satisfy $\|\mb y_i^{(P)}\|\geq \frac{1}{2}$ since $\mb e^\top \mb y_i^{(P)} \leq -\frac{1}{2}$. Then, by Eq~\eqref{eq:passed_probes_P-1} writing $\mb u_i:= \mb y_i^{(P)} / \|\mb y_i^{(p)}\|$ for $i\in[k]$, we obtain
    \begin{equation*}
        \|\proj_E(\mb u_i)\| = \|\proj_{E_{P-1}}(\mb u_i)\| \leq 2 \|\proj_{E_{P-1}}(\mb y_i^{(P)})  \| \leq 2\eta_{P-1},\quad i\in[k].
    \end{equation*}
    Last, by definition, the exploratory queries are robustly-independent:
    \begin{equation*}
        \|\proj_{Span(\mb u_j^{(p)},j<i)^\perp}(\mb u_i^{(p)})\| \geq \|\proj_{Span(\mb y_j^{(p)},j<i)^\perp}(\mb y_i^{(p)})\| \geq \gamma_P,\quad j\in[k].
    \end{equation*}
    In summary, on $\Ecal\cap\Fcal\cap\Gcal$, the algorithm wins at Game~\ref{game:adapted_orthogonal_subspace_game} using memory at most $M+\log_2 J +1$, $m=Nlk\leq \frac{\tilde d}{2}$ queries and parameters $(\beta,\gamma) = (2\eta_{P-1},\gamma_P)$. It suffices to check that the assumptions from \cref{thm:adapted_memory_query_lower_bound} are satisfied. We only need to check the bound on $\gamma/\beta$. We compute
    \begin{equation*}
        \frac{\gamma}{\beta} = \frac{\gamma_P}{2\eta_{P-1}} = \frac{\delta_1^{(p)}}{8k\eta_{P-1}}  \geq \frac{\mu}{16k}\paren{1-\frac{2}{k}}^{k-1} \sqrt{\frac{l}{\tilde d}}  \geq 3e\sqrt{\frac{kd}{\tilde d}}.
    \end{equation*}
    In the last inequality, we used the fact that because $k\geq 3$, we have $(1-2/k)^{k-1} \geq (1-2/3)^2$.
    Combining the previous results, we obtain
    \begin{equation*}
        \Pbb(\Ecal\cap\Fcal\cap \Gcal) \leq \Pbb (\text{\cref{alg:strategy_adapted_orthogonal_subspace_game} wins at Game}~\ref{game:adapted_orthogonal_subspace_game} ) \leq \frac{1}{2\zeta}.
    \end{equation*}
    In turn, this shows that
    \begin{align*}
        \Pbb(\Ecal^c) &\geq \Pbb(\Ecal^c \cap \Fcal\cap \Gcal) =\Pbb(\Fcal\cap\Gcal) - \Pbb(\Ecal\cap\Fcal\cap\Fcal) \\
        &\geq \Pbb(\Gcal) -\Pbb(\Fcal^c)  - \Pbb(\Ecal\cap\Fcal\cap\Gcal)  \\
        &\geq q - k^2 J_P (N+1) e^{-k\ln d} - \frac{1}{2\zeta}  \geq q-\frac{1}{2\zeta}-\frac{1}{2\zeta} =q-\frac{1}{\zeta}.
    \end{align*}

    The last step of the proof is to construct a strategy for Game~\ref{game:feasibility_game_randomized} with $P-1$ layers that wins under the event $\Ecal^c$. Note that this event corresponds exactly to the case when in at least one of the depth-$(P-1)$ periods the algorithm performed $k$ exploratory queries. Hence, the strategy for $P-1$ layers mainly amounts to simulating that period. This can be done thanks to the index $\hat j$ which precisely specifies the probing subspaces needed to simulate that winning period. Because there were $N$ depth-$(P-1)$ periods, the new index parameter becomes $J_{P-1} := NJ_P$. The complete strategy is described in \cref{alg:strategy_recursion_randomized} and is similar to the one constructed for the deterministic case in \cref{lemma:recursion} (\cref{alg:strategy_recursion}). The main difference with the deterministic case is that there is no need to keep exploratory queries for larger depths (here there is only $P$) in the new strategy. Indeed, because the oracle is non-adaptive, these deeper layers do not provide information on the other layers.

    \begin{algorithm}[ht!]
        
\caption{Strategy of the Player for Game~\ref{game:feasibility_game_randomized} with $P-1$ layers given a strategy for $P$ layers}\label{alg:strategy_recursion_randomized}

\setcounter{AlgoLine}{0}
\SetAlgoLined
\LinesNumbered

\everypar={\nl}

\hrule height\algoheightrule\kern3pt\relax
\KwIn{dimensions $d$, $\tilde d$, number of vectors $k$, depth $p$, $M$-bit memory algorithm $alg$ for Game~\ref{game:feasibility_game_randomized} for $P$ layers; maximum index $J_P$}

\KwOut{strategy for Game~\ref{game:feasibility_game_randomized} for $P-1$ layers with maximum index $J_{P-1}=NJ_P$}

\vspace{5pt}

Receive subspaces $E_1,\ldots,E_{P-1}$ and sequences $\Vbb^{(p,1)},\ldots,\Vbb^{(p,J_{P-1})}$ for $p\in[P-1]$

Sample $E_P$ as an independent uniform $\tilde d$-dimensional subspace of $\Rbb^d$ and sample i.i.d. sequences $\Wbb^{(P,1)},\ldots,\Wbb^{(P,J_P)}$ as in line 2 of Game~\ref{game:feasibility_game_randomized} for $P$ layers

For $p\in[P-1]$, reorganize the sequences $\Vbb^{(p,j)}$ as follows. For $j\in[J_P]$ let $\Wbb^{(p,j)}$ be the concatenation of $\Vbb^{(p,N(j-1)+1)},\ldots,\Vbb^{(p,Nj)}$ in that order so that $\Wbb^{(p,j)}$ has $N^{P-p}$ elements

Based on all $E_p$ and $\Wbb^{(p,j)}$ for $j\in[J_P]$ and $p\in[P]$ initialize memory of $alg$ to the $M$-bit message $\mathsf{Memory}$ and store the index $\hat j \in[J_P]$ as in line 3 of Game~\ref{game:feasibility_game_randomized} with $P$ layers

Run $T_P$ iterations of feasibility problem with $alg$ using $\Wbb^{(1,\hat j)},\ldots,\Wbb^{(P,\hat j)}$ as in lines 5-9 of Game~\ref{game:feasibility_game_randomized}.

\lIf{at most $k-1$ exploratory queries for all depth-$(P-1)$ periods}{Strategy fails; \textbf{end}}
\Else{
    Let $[\hat aT_{P-1},(\hat a+1)T_{P-1})$ be the first such depth-$(P-1)$ period  for $\hat a\in[0,N)$
    
    Submit memory state $\mathsf{Message}$  of $alg$ at the beginning of the period (just before iteration $\hat a T_{P-1}$) and index $(\hat j-1)N + \hat a+1 \in [J_{P-1}]$
}

\hrule height\algoheightrule\kern3pt\relax
\end{algorithm}

    It is straightforward to check that the sequences $\Wbb^{(p,1)},\ldots,\Wbb^{(p,J_P)}$ for $p\in[P]$ constructed in \cref{alg:strategy_recursion_randomized} are identically distributed as the sequences constructed by the oracle of Game~\ref{game:feasibility_game_randomized} for $P$ layers. By definition of $\Ecal$, under $\Ecal^c$ there is a depth-$(P-1)$ period with $k$ exploratory queries hence the strategy does not fail. Because of the choice of index $(\hat j-1)N+\hat a+1$, during the run lines 5-9 of Game~\ref{game:feasibility_game_randomized} for depth $P-1$, (since $alg$ reuses exactly the same randomness) the oracle exactly implements the $\hat a$-th depth-$(P-1)$ period. Hence under $\Ecal^c$, \cref{alg:strategy_recursion_randomized} wins. This ends the proof.    
    \end{proof}

Applying the previous result for Game~\ref{game:feasibility_game_randomized} with all depths $p\in\{2,\ldots,P\}$ recursively gives the following query lower bound.

\begin{theorem}\label{thm:final_lower_bound_randomized_game}
    Let $P\geq 2$ and $d\geq 40P$. Suppose that $l_P \geq l$ and
    \begin{equation}\label{eq:assumption_on_k_randomized}
         c_5 \frac{M + P\ln d}{d}P^3\ln d \leq k \leq  c_4 \paren{\frac{d}{P\ln d}}^{1/4}
    \end{equation}
    for some universal constants $c_4,c_5>0$. If a strategy for Game~\ref{game:feasibility_game} for depth $P$ and maximum index $J_P=N$ uses $M$ bits of memory and wins with probability at least $\frac{1}{2}$, then it performed at least
    \begin{equation*}
        T_{max} > T_P \geq \frac{k}{2} \paren{\frac{d}{12Plk}}^{P-1}
    \end{equation*}
    queries.
\end{theorem}

\begin{proof}
    Suppose for now that the parameter $k$ satisfies all assumptions from \cref{lemma:recursive_randomized} for all Games~\ref{game:feasibility_game_randomized} with layers $p\in\{2,\ldots,P\}$ and $\zeta = P$. Then, starting from a strategy for depth $P$, maximum index $J_P=N$, and winning with probability $q$ with $T_P$ queries, we can construct a strategy for the depth-$1$ game with maximum index $J_1 = N^P$ and wins with probability at least $q-\frac{1}{2P}(P-1) \geq q-\frac{P-1}{2P}>0$. As in the deterministic case, to win at Game~\ref{game:feasibility_game_randomized} for depth $1$ one needs at least $k$ queries and we reach a contradiction since $T_1=\frac{k}{2}$. Hence, this shows that an algorithm that wins with probability at least $\frac{1}{2}$ at Game~\ref{game:feasibility_game_randomized} with $P$ layers and $J_P=N$ must make at least the following number of queries
    \begin{equation*}
        T_{max}>T_P = \frac{k}{2}N^{P-1}.
    \end{equation*}
    Now assuming that $\tilde d \geq 4lk$, we obtain $N\geq \tilde d/(4lk)$. Hence, using $\tilde d\geq \frac{d}{3P}$, we obtain the desired lower bound
    \begin{equation*}
        T_{max} >  \frac{k}{2}\floor{ \frac{d}{12Plk} }^{P-1}.
    \end{equation*}

    We now check that the assumptions for \cref{lemma:recursive_randomized} are satisfied for all games with layers $p\in\{2,\ldots,P\}$. It suffices to check that
    \begin{equation}\label{eq:main_assumption_k_randomized}
        \frac{\tilde d}{4l} \geq k \geq 100 P \frac{M+2\log_2 J_1 + 3}{\tilde d} \ln\frac{\sqrt d}{\gamma_1} \lor \frac{\ln(2 P k^2 J_1(N+1))}{\ln d} \lor 3 .
    \end{equation}
    For the upper bound, recalling the definition of $l$ in Eq~\eqref{eq:def_l_randomized}, we have that
    \begin{equation*}
        \frac{\tilde d}{4l} \geq \frac{d}{12Pl} = \Omega\paren{\frac{d}{k^3 P\ln d}}.
    \end{equation*}
    In particular, the left-hand-side of Eq~\eqref{eq:main_assumption_k_randomized} holds for
    \begin{equation*}
        k \leq \Omega\paren{\paren{\frac{d}{P\ln d}}^{1/4}}.
    \end{equation*}
    For the upper bound, because $\log_2 J_1 \leq P\log_2 d$ and $\gamma_1 = \frac{\delta_1^{(1)}}{4k}$, we have
    \begin{equation*}
        100 P \frac{M+2\log_2 J_1 + 3}{\tilde d} \ln\frac{\sqrt d}{\gamma_1} = \Ocal\paren{\frac{M+P\ln d}{d}P^3 \ln d}.
    \end{equation*}
    On the other hand, $\ln(2P k^2 J_P(N+1)) = \Ocal(P\ln d)$, hence the first term in the right-hand side of Eq~\eqref{eq:main_assumption_k_randomized} dominates.
    As a summary, for an appropriate choice of constants $c_4,c_5>0$, Eq~\eqref{eq:main_assumption_k_randomized} holds, which ends the proof.
\end{proof}

The last step of the proof is to link the feasibility problem with the oracle $\tilde\Ocal_t$ for $t\geq 0$ with Game~\ref{game:feasibility_game_randomized}. This step is exactly similar to the deterministic case when we reduced Procedure~\ref{proc:feasibility} to Game~\ref{game:feasibility_game}. By giving a reduction to the Kernel Discovery Game~\ref{game:kernel_discovery} we show that an algorithm that solves the feasibility problem with the oracles $\tilde\Ocal_t$ must solve multiple instances of Game~\ref{game:feasibility_game_randomized} with $P$ layers.

\begin{lemma}\label{lemma:reduction_randomized_oracle_to_game}
    Let $P\geq 2$ and $k\geq 3$. Suppose that $4l_Pk \leq \tilde d$ and $l_P\geq l$. Let $N_P = \floor{\tilde d/(2l_P k)}$. Suppose that there is an $M$-bit algorithm that solves the feasibility problem with the randomized oracles $(\tilde\Ocal_t)_{t\geq 0}$ using at most $M$ bits of memory and $N_PT_P$ iterations, and that finds a feasible solution with probability at least $q$. Then, there exists a strategy Game~\ref{game:feasibility_game_randomized} for depth $P$ and maximum index $J_P=N$ that uses $M$ bits of memory, $T_P$ iterations and wins with probability at least $q-k^2N_P e^{-k\ln d} - e^{-\tilde d/10}$.
\end{lemma}

\begin{proof}
    Fix an $M$-bit feasibility algorithm $alg$ for the randomized oracle $(\tilde\Ocal_t)_{t\geq 0}$ satisfying the hypothesis. We construct from this algorithm a strategy for the Kernel Discovery Game~\ref{game:kernel_discovery} with $m=Nl_P k \leq \frac{\tilde d}{2}$ samples. The construction is essentially the same as that for the deterministic case in \cref{lemma:final_depth_exploration} (\cref{alg:strategy_discovery_game}): we simulate a run of the feasibility problem using for $E_P$ the $\tilde d$-dimensional space sampled by the oracle and using the samples of the oracle $\mb v_1,\ldots,\mb v_m$ to construct the depth-$P$ probing subspaces. The strategy is given in \cref{alg:strategy_discovery_game_randomized}.

    \begin{algorithm}[ht!]
        
    \caption{Strategy of the Player for the Kernel Discovery Game~\ref{game:kernel_discovery}}\label{alg:strategy_discovery_game_randomized}
    
    \setcounter{AlgoLine}{0}
    \SetAlgoLined
    \LinesNumbered
    
    \everypar={\nl}
    
    \hrule height\algoheightrule\kern3pt\relax
    \KwIn{depth $P$, dimensions $d$, $\tilde d$, $l$, parameter $k$, $M$-bit algorithm $alg$, number of samples $m=Nl_Pk$}

    \vspace{5pt}
    
    Sample $E_1,\ldots, E_{P-1}$ i.i.d. uniform $\tilde d$-dimensional subspaces of $\Rbb^d$ and independent sequences $(\mb V^{(p,a)})_{a\geq 0}$ for $p\in[P-1]$ as in the construction of the randomized oracles $\tilde\Ocal_t$.

    Receive samples $\mb v_1,\ldots,\mb v_{Nl_Pk}$. Let $ V_i^{(P,a)}:=Span(\mb v_{al_P k+(i-1)l_P+s},s\in[l_P])$ for $i\in[k],a\in[0,N)$
    
    Set memory of $alg$ to $\mb 0$ and run $N_PT_P$ iterations of the feasibility problem with $alg$ and the oracles $\tilde \Ocal_{\mb V^{(1,\floor{t/T_1})} ,\ldots,\mb V^{(P,\floor{t/T_P})}} $ for $t\in[0,N_PT_P)$

    \uIf{at any time $t\in[0,N_PT_P)$ the oracle outputs $\mathsf{Success}$ on query $\mb x_t$ of $alg$}{
        \Return $\frac{\mb x_t}{\|\mb x_t\|}$ to oracle; \textbf{break}
    }
    \lElse{
        Strategy fails; \textbf{end}
    }

    \hrule height\algoheightrule\kern3pt\relax
    \end{algorithm}

    The rest of the proof uses the same arguments as for \cref{lemma:recursive_randomized}. Define the event
    \begin{equation*}
        \Ecal = \bigcap_{a<N} \set{ \text{at most $k$ depth-$P$ exploratory queries during period }[aT_P,(a+1)T_P)}.
    \end{equation*}
    By \cref{lemma:properties_oracle_construction} and the union bound, on an event $\Fcal$ of probability at least $1-k^2N_P e^{-k\ln d}$, all depth-$P$ periods $[aT_P,(a+1)T_P)$ for $a\in[0,N_P)$ satisfy the properties from \cref{lemma:properties_oracle_construction}. Last, let $\Gcal$ be the event on which $alg$ solves the feasibility problem with oracles $(\tilde\Ocal_t)$ constructed using the same sequences $(\mb V^{(p,a)})_{a\geq 0}$ for $p\in[P]$ as constructed in \cref{alg:strategy_discovery_game_randomized}. Under $\Ecal\cap\Fcal$, any query $\mb x_t$ for $t\in[0,N_PT_P)$ that passed probes at depth $P$ satisfies
    \begin{equation*}
        \|\proj_{E_P}(\mb x_t)\| \leq \eta_P.
    \end{equation*}
    Now note that under $\Gcal$, the algorithm run with the oracles $\tilde\Ocal_t$ for $t\in[0,N_PT_P)$ finds a successful query for the oracles and in particular there is some query $t\in[0,N_PT_P)$ which passed all probes at all depths $p\in[P]$:
    \begin{equation*}
        \tilde \Ocal_{\mb V^{(1,\floor{t/T_1})} ,\ldots,\mb V^{(P,\floor{t/T_P})}} (\mb x_t) = \mathsf{Success}.
    \end{equation*}
    Consider the first time $\hat t$ when such query $\mb x_t$ is successful for the oracles $\tilde \Ocal_{\mb V^{(1,\floor{t/T_1})} ,\ldots,\mb V^{(P,\floor{t/T_P})}} $. We pose $\hat t=N_PT_P$ if there is no such query. By construction, at all previous times the responses satisfy
    \begin{equation*}
        \tilde\Ocal_t(\mb x_t) = \tilde \Ocal_{\mb V^{(1,\floor{t/T_1})} ,\ldots,\mb V^{(P,\floor{t/T_P})}} (\mb x_t),\quad t<\hat t.
    \end{equation*}
    In summary, up until this time $\hat t$, the run in line 3 of \cref{alg:strategy_discovery_game_randomized} is equivalent to a run of the feasibility with the original oracles $(\tilde\Ocal_t)_{t\geq 0}$. 
    Hence, under $\Ecal\cap\Fcal\cap\Gcal$, the algorithm returns the vector $\mb y:= \mb x_{\hat t}/\|\mb x_{\hat t}\|$ and we have $\|\proj_{E_P}(\mb x_{\hat t})\|\leq \eta_P$. The successful query also satisfies $\|\mb x_{\hat t}\|\geq \frac{1}{2}$ since $\mb e^\top \mb x_{\hat t}\leq -\frac{1}{2}$. As a result, under $\Ecal\cap\Fcal\cap\Gcal$,
    \begin{equation*}
        \|\proj_E(\mb y)\| = \|\proj_{E_P}(\mb y)\| \leq 2\|\proj_{E_P}(\mb x_{\hat t})\| \leq 2\eta_P \leq \sqrt{\frac{\tilde d}{20d}}.
    \end{equation*}
    Because $m=Nl_Pk\leq \frac{\tilde d}{2}$, \cref{lemma:kenel_discovery_query_lower_bound} implies
    \begin{equation*}
        \Pbb(\Ecal\cap\Fcal\cap\Gcal) \leq \Pbb(\text{\cref{alg:strategy_discovery_game_randomized} wins at Game~\ref{game:kernel_discovery}}) \leq e^{-\tilde d/10}.
    \end{equation*}
    Hence,
    \begin{equation*}
        \Pbb(\Ecal^c) \geq \Pbb(\Gcal) -\Pbb(\Ecal\cap\Fcal\cap\Gcal) - \Pbb(\Fcal^c) \geq q-k^2N_Pe^{-k\ln d} - e^{-\tilde d/10}.
    \end{equation*}
    From this, using the exact same arguments as in \cref{lemma:recursive_randomized} (\cref{alg:strategy_recursion_randomized}) we can construct a strategy for Game~\ref{game:feasibility_game_randomized} with $P$ layers and maximal index $J_P=N_P$, and wins under the event $\Ecal^c$. It suffices to simulate the specific depth-$P$ period on which $alg$ queried the successful vector $\mb x_{\hat t}$. Note that because $l_P \geq l$, we have $N_P\leq N$. Hence, this strategy also works if instead we have access to a larger maximal index $J_P=N$. This ends the proof.
\end{proof}

We now combine this reduction to Game~\ref{game:feasibility_game_randomized} together with the query lower bound of \cref{thm:final_lower_bound_randomized_game}. This directly gives the following result.

\begin{theorem}\label{thm:final_query_bound_randomized_oracle}
    Let $P\geq 2$ and $d\geq 20P$. Suppose that $k$ satisfies Eq~\eqref{eq:assumption_on_k_randomized} as in \cref{thm:final_lower_bound_randomized_game}. Also, suppose that $4l_Pk \leq \tilde d$ and $l_P \geq l$. If an algorithm solves the feasibility problem with oracles $(\tilde\Ocal_t)_{t\geq 0}$ using $M$ bits of memory and at most $T_{max}$ queries with probability at least $\frac{3}{4}$, then
    \begin{equation*}
        T_{max} >N_PT_P \geq \frac{kl}{2l_P}\paren{\frac{d}{12Plk}}^P.
    \end{equation*}
\end{theorem}

\begin{proof}
    Fix parameters satisfying Eq~\eqref{eq:assumption_on_k_randomized}. Suppose that we have such an algorithm $alg$ for the feasibility problem with oracles $(\tilde\Ocal_t)_{t\geq 0}$ that only uses $T_{max}\leq N_PT_P$ queries and wins with probability at least $\frac{3}{4}$. Because $k^2N_Pe^{-k\ln d} +e^{-\tilde d/10}\leq \frac{1}{4}$, we can directly combine \cref{lemma:reduction_randomized_oracle_to_game} with \cref{thm:final_lower_bound_randomized_game} to reach a contradiction. Hence, $T_{max}>N_PT_P$ which ends the proof.
\end{proof}

Now, we easily check that the same proof as \cref{lemma:check_procedure_is_feasibility_pb} shows that under the choice of parameters, with probability at least $1-e^{-d/40}$ over the randomness of $E_1,\ldots,E_P$, the feasible set $\tilde Q_{E_1,\ldots,E_P}$ contains a ball of radius $\epsilon = \delta_1^{(1)}/2$, hence using the oracles $(\tilde\Ocal_t)_{t\geq 0}$ is consistent with a feasible problem with accuracy $\epsilon$. These observations give the following final query lower bound for memory-constrained feasibility algorithms.

\begin{theorem}\label{thm:query_lower_bound_randomized}
    Let $d\geq 1$ and an accuracy $\epsilon\in(0,\frac{1}{\sqrt d}]$ such that
    \begin{equation*}
        d^{1/4} \ln^2 d \geq c_6 \frac{M}{d} \ln^{3+1/4} \frac{1}{\epsilon}.
    \end{equation*}
    for some universal constant $c_6>0$. Then, any $M$-bit randomized algorithm that solves any feasibility problems for accuracy $\epsilon$ with probability at least $\frac{9}{10}$ makes at least
    \begin{equation*}
        \paren{\frac{d}{M}}^2 \frac{1}{\epsilon^{2 \phi(d,M,\epsilon)}}
    \end{equation*}
    queries, where $\phi(d,M,\epsilon) =  1- 4  \frac{\ln  \frac{M}{d}}{ \ln d} - \Ocal\paren{\frac{\ln \frac{\ln(1/\epsilon)}{\ln d}}{\ln d} + \frac{\ln\ln d}{\ln d}}$.
\end{theorem}

\begin{proof}
    We first define
    \begin{equation*}
        \tilde P:= \frac{2\ln\frac{1}{\epsilon} - \ln\paren{\frac{\ln(1/\epsilon)}{\ln d}+1} - 15 }
        {\ln d +15 }.
    \end{equation*}
    We then use the following parameters,
    \begin{align*}
        P:= \floor{ \tilde P },
        \quad
        k:=\ceil{3c_4 \frac{M}{d}P^3\ln d },
        \quad \text{and}\quad l_P := l \lor \floor{\paren{\frac{\tilde d}{4k}}^{P-\tilde P}}.
    \end{align*}
    Note that in particular, $P\leq P_{max} = \frac{\ln\frac{1}{\epsilon}}{\ln d}+1$. Taking $c_6$ sufficiently large, the hypothesis constraint implies in particular $d\geq 20P$.
    
    We assume $P\geq 2$ from now. The same proof as \cref{lemma:check_procedure_is_feasibility_pb} shows that on an event $\Ecal$ of probability $1-e^{-d/40}$ the feasible set $\tilde Q_{E_1,\ldots,E_P}$ contains a ball of radius $\delta_1^{(1)}/2$. Note that
    \begin{equation*}
        \delta_1^{(1)} = \frac{(1-2/k)^{k-1}}{20\mu_P\mu^{P-2}}\sqrt{\frac{l}{d}} \geq \frac{6k^{3/2}}{\mu^P} \sqrt{\frac{l_P}{l}} \geq \frac{6k}{\mu^{\tilde P}}\sqrt{k\paren{\frac{\tilde d}{4k\mu^2}}^{P-\tilde P}} ,
    \end{equation*}
    where we used $(1-2/k)^{k-1}\geq (1-2/3)^2$ because $k\geq 3$. Furthering the bounds and using $l\geq Ck^3\ln d$ from Eq~\eqref{eq:def_l_randomized} gives
    \begin{equation*}
        \frac{\delta_1^{(1)}}{2} \geq \frac{3k}{2\mu^{\tilde P}} \sqrt{\frac{\tilde d}{\mu^2}} \geq \frac{1}{1400 \mu^{\tilde P}}\sqrt{\frac{l}{Pk}}  \geq \frac{1}{1400 \mu^{\tilde P} \sqrt P} \geq \epsilon.
    \end{equation*}
    Next, because $l\geq Ck^3\ln d$, we have $\mu\leq 1200\sqrt{d/\ln d}$. Now $\tilde P$ was precisely chosen so that
    \begin{equation*}
        \tilde P \leq  \frac{2\ln\frac{1}{\epsilon}-\ln(P_{max})-2\ln(1400)}
        { \ln d -\ln\ln d +2\ln(1200)}.
    \end{equation*}
    The previous equations show that under $\Ecal$, the oracles $(\tilde\Ocal_t)_{t\geq 0}$ form a valid feasibility problem for accuracy $\epsilon>0$. In particular, an algorithm solving feasibility problems with accuracy $\epsilon$ with $T_{max}$ oracle calls and probability at least $\frac{9}{10}$ would solve the feasibility problem with oracles $(\tilde\Ocal_t)_{t\geq 0}$ with probability at least $\frac{9}{10}-\Pbb(\Ecal) \geq \frac{9}{10} - e^{-d/40}$. Because $d\geq 40P\geq 80$, the algorithm wins with probability at least $\frac{3}{4}$. We now check that the assumptions to apply \cref{thm:final_query_bound_randomized_oracle} hold. By construction of $l_P$, we have directly $l_P \geq l$. Further, if Eq~\eqref{eq:assumption_on_k_randomized}) is satisfied, we would have in particular $4lk \leq \tilde d$, hence we would also have $4l_Pk \leq \tilde d$. It now suffices to check that Eq~\eqref{eq:assumption_on_k_randomized} is satisfied.
    
    As in the proof of \cref{thm:deterministic_alg_lower_bound}, without loss of generality we can assume that $M \geq 2d\ln 1/\epsilon$ since this is necessary to solve even convex optimization problems. Hence, $k$ directly satisfies the right-hand side of Eq~\eqref{eq:assumption_on_k_randomized}. Next, the assumption gives
    \begin{equation*}
        d^{1/4} \ln^2 d \geq \frac{60c_5}{c_4} \frac{M}{d} \ln^{3+1/4} \frac{1}{\epsilon}.
    \end{equation*}
    Hence, $c_4 \paren{\frac{d}{P\ln d}}^{1/4} \geq c_4 \paren{\frac{d}{P_{max}\ln d}}^{1/4} \geq 6c_5 (M/d) P_{max}^3 \ln d\geq k$. As a result, the right-hand side of Eq~\eqref{eq:assumption_on_k_randomized} holds. Then, \cref{thm:final_query_bound_randomized_oracle} shows that
    \begin{equation*}\label{eq:actual_bound_randomized}
        T_{max} \geq  \frac{kl}{2l_P} \paren{\frac{d}{12Plk}}^P \geq \frac{k}{2(2l)^{P-\tilde P}} \paren{\frac{d}{12Plk}}^{\tilde P} \geq \frac{1}{8Ck^2 \ln d} \paren{\frac{d}{12Plk}}^{\tilde P} =: \frac{1}{e^{2\tilde\phi(d,M,\epsilon)}} .
    \end{equation*}
    The above equation also holds even if $P<2$ (that is, $\tilde P\leq 1$) because $d$ iterations are necessary even to solve convex optimization problems. We now simplify
    \begin{align*}
        \tilde\phi(d,M,\epsilon) &=- \frac{\ln \frac{M}{d}}{\ln\frac{1}{\epsilon}} + \Ocal\paren{\frac{\ln\ln\frac{1}{\epsilon}}{\ln\frac{1}{\epsilon}}} + \frac{\tilde P\ln \frac{d}{12Plk}}{2\ln\frac{1}{\epsilon}}\\
        &= - \frac{\ln \frac{M}{d}}{\ln\frac{1}{\epsilon}}
        + \frac{\ln d - 4\ln \frac{M}{d} -13\ln \frac{\ln(1/\epsilon)}{\ln d}}
        {\ln d}
        +\Ocal\paren{\frac{\ln\ln d}{\ln d}}.
    \end{align*}
    This ends the proof.
\end{proof}

In the subexponential regime $\ln\frac{1}{\epsilon} \leq d^{o(1)}$, \cref{thm:query_lower_bound_randomized} simplifies to \cref{thm:main_result_randomized}.

\paragraph{Acknowledgements.}

The authors would like to thank to Gregory Valiant for very insightful discussions. This work was partly funded by ONR grant N00014-18-1-2122 and AFOSR grant FA9550-23-1-0182.

\bibliography{refs}

\begin{thebibliography}{63}
\providecommand{\natexlab}[1]{#1}
\providecommand{\url}[1]{\texttt{#1}}
\expandafter\ifx\csname urlstyle\endcsname\relax
  \providecommand{\doi}[1]{doi: #1}\else
  \providecommand{\doi}{doi: \begingroup \urlstyle{rm}\Url}\fi

\bibitem[Altschuler and Parrilo(2023)]{altschuler2023acceleration}
Jason~M Altschuler and Pablo~A Parrilo.
\newblock Acceleration by stepsize hedging {I}: Multi-step descent and the silver stepsize schedule.
\newblock \emph{arXiv preprint arXiv:2309.07879}, 2023.

\bibitem[Anstreicher(1997)]{anstreicher1997vaidya}
Kurt~M Anstreicher.
\newblock On vaidya's volumetric cutting plane method for convex programming.
\newblock \emph{Mathematics of Operations Research}, 22\penalty0 (1):\penalty0 63--89, 1997.

\bibitem[Atkinson and Vaidya(1995)]{atkinson1995cutting}
David~S Atkinson and Pravin~M Vaidya.
\newblock A cutting plane algorithm for convex programming that uses analytic centers.
\newblock \emph{Mathematical programming}, 69\penalty0 (1-3):\penalty0 1--43, 1995.

\bibitem[Balkanski and Singer(2018)]{balkanski2018parallelization}
Eric Balkanski and Yaron Singer.
\newblock Parallelization does not accelerate convex optimization: Adaptivity lower bounds for non-smooth convex minimization.
\newblock \emph{arXiv preprint arXiv:1808.03880}, 2018.

\bibitem[Beame et~al.(2018)Beame, Oveis~Gharan, and Yang]{Beame2018}
Paul Beame, Shayan Oveis~Gharan, and Xin Yang.
\newblock Time-space tradeoffs for learning finite functions from random evaluations, with applications to polynomials.
\newblock In \emph{Proceedings of the 31st Conference On Learning Theory}, pages 843--856. PMLR, 2018.

\bibitem[Bertsimas and Vempala(2004)]{bertsimas2004solving}
Dimitris Bertsimas and Santosh Vempala.
\newblock Solving convex programs by random walks.
\newblock \emph{Journal of the ACM (JACM)}, 51\penalty0 (4):\penalty0 540--556, 2004.

\bibitem[Blanchard et~al.(2023)Blanchard, Zhang, and Jaillet]{blanchard2023quadratic}
Mo\"{i}se Blanchard, Junhui Zhang, and Patrick Jaillet.
\newblock Quadratic memory is necessary for optimal query complexity in convex optimization: Center-of-mass is pareto-optimal.
\newblock In \emph{The Thirty Sixth Annual Conference on Learning Theory}, pages 4696--4736. PMLR, 2023.

\bibitem[Blanchard et~al.(2024)Blanchard, Zhang, and Jaillet]{blanchard2024memory}
Mo{\"\i}se Blanchard, Junhui Zhang, and Patrick Jaillet.
\newblock Memory-constrained algorithms for convex optimization.
\newblock \emph{Advances in Neural Information Processing Systems}, 36, 2024.

\bibitem[Brown et~al.(2021)Brown, Bun, Feldman, Smith, and Talwar]{Brown2021}
Gavin Brown, Mark Bun, Vitaly Feldman, Adam Smith, and Kunal Talwar.
\newblock When is memorization of irrelevant training data necessary for high-accuracy learning?
\newblock In \emph{Proceedings of the 53rd Annual ACM SIGACT Symposium on Theory of Computing}, STOC 2021, page 123–132. Association for Computing Machinery, 2021.

\bibitem[Brown et~al.(2022)Brown, Bun, and Smith]{brown22a}
Gavin Brown, Mark Bun, and Adam Smith.
\newblock Strong memory lower bounds for learning natural models.
\newblock In \emph{Proceedings of Thirty Fifth Conference on Learning Theory}, pages 4989--5029. PMLR, 2022.

\bibitem[Bubeck et~al.(2019)Bubeck, Jiang, Lee, Li, and Sidford]{bubeck2019complexity}
S{\'e}bastien Bubeck, Qijia Jiang, Yin-Tat Lee, Yuanzhi Li, and Aaron Sidford.
\newblock Complexity of highly parallel non-smooth convex optimization.
\newblock \emph{Advances in neural information processing systems}, 32, 2019.

\bibitem[Chen and Peng(2023)]{chen2023memory}
Xi~Chen and Binghui Peng.
\newblock Memory-query tradeoffs for randomized convex optimization.
\newblock In \emph{2023 IEEE 64th Annual Symposium on Foundations of Computer Science (FOCS)}, pages 1400--1413. IEEE, 2023.

\bibitem[Das~Gupta et~al.(2023)Das~Gupta, Van~Parys, and Ryu]{das2023branch}
Shuvomoy Das~Gupta, Bart~PG Van~Parys, and Ernest~K Ryu.
\newblock Branch-and-bound performance estimation programming: A unified methodology for constructing optimal optimization methods.
\newblock \emph{Mathematical Programming}, pages 1--73, 2023.

\bibitem[Davidson and Szarek(2001)]{davidson2001local}
Kenneth~R Davidson and Stanislaw~J Szarek.
\newblock Local operator theory, random matrices and banach spaces.
\newblock \emph{Handbook of the geometry of Banach spaces}, 1\penalty0 (317-366):\penalty0 131, 2001.

\bibitem[Fletcher and Reeves(1964)]{fletcher1964function}
Reeves Fletcher and Colin~M Reeves.
\newblock Function minimization by conjugate gradients.
\newblock \emph{The computer journal}, 7\penalty0 (2):\penalty0 149--154, 1964.

\bibitem[Garg et~al.(2018)Garg, Raz, and Tal]{Garg2018}
Sumegha Garg, Ran Raz, and Avishay Tal.
\newblock Extractor-based time-space lower bounds for learning.
\newblock In \emph{Proceedings of the 50th Annual ACM SIGACT Symposium on Theory of Computing}, STOC 2018, page 990–1002. Association for Computing Machinery, 2018.

\bibitem[Garg et~al.(2019)Garg, Raz, and Tal]{garg2019time}
Sumegha Garg, Ran Raz, and Avishay Tal.
\newblock Time-space lower bounds for two-pass learning.
\newblock In \emph{34th Computational Complexity Conference (CCC)}, 2019.

\bibitem[Garg et~al.(2021)Garg, Kothari, Liu, and Raz]{garg2021memory}
Sumegha Garg, Pravesh~K Kothari, Pengda Liu, and Ran Raz.
\newblock Memory-sample lower bounds for learning parity with noise.
\newblock \emph{arXiv preprint arXiv:2107.02320}, 2021.

\bibitem[Hager and Zhang(2006)]{hager2006survey}
William~W Hager and Hongchao Zhang.
\newblock A survey of nonlinear conjugate gradient methods.
\newblock \emph{Pacific journal of Optimization}, 2\penalty0 (1):\penalty0 35--58, 2006.

\bibitem[Hanson and Wright(1971)]{hanson1971bound}
David~Lee Hanson and Farroll~Tim Wright.
\newblock A bound on tail probabilities for quadratic forms in independent random variables.
\newblock \emph{The Annals of Mathematical Statistics}, 42\penalty0 (3):\penalty0 1079--1083, 1971.

\bibitem[Jiang et~al.(2020)Jiang, Lee, Song, and Wong]{jiang2020improved}
Haotian Jiang, Yin~Tat Lee, Zhao Song, and Sam Chiu-wai Wong.
\newblock An improved cutting plane method for convex optimization, convex-concave games, and its applications.
\newblock In \emph{Proceedings of the 52nd Annual ACM SIGACT Symposium on Theory of Computing}, pages 944--953, 2020.

\bibitem[Khachiyan(1980)]{khachiyan1980polynomial}
Leonid~G Khachiyan.
\newblock Polynomial algorithms in linear programming.
\newblock \emph{USSR Computational Mathematics and Mathematical Physics}, 20\penalty0 (1):\penalty0 53--72, 1980.

\bibitem[Kol et~al.(2017)Kol, Raz, and Tal]{Kol2017}
Gillat Kol, Ran Raz, and Avishay Tal.
\newblock Time-space hardness of learning sparse parities.
\newblock In \emph{Proceedings of the 49th Annual ACM SIGACT Symposium on Theory of Computing}, STOC 2017, page 1067–1080. Association for Computing Machinery, 2017.

\bibitem[Lan et~al.(2020)Lan, Lee, and Zhou]{lan_communication-efficient_2020}
Guanghui Lan, Soomin Lee, and Yi~Zhou.
\newblock Communication-efficient algorithms for decentralized and stochastic optimization.
\newblock \emph{Mathematical Programming}, 180\penalty0 (1):\penalty0 237--284, March 2020.
\newblock ISSN 1436-4646.
\newblock \doi{10.1007/s10107-018-1355-4}.
\newblock URL \url{https://doi.org/10.1007/s10107-018-1355-4}.

\bibitem[Laurent and Massart(2000)]{laurent2000adaptive}
Beatrice Laurent and Pascal Massart.
\newblock Adaptive estimation of a quadratic functional by model selection.
\newblock \emph{Annals of statistics}, pages 1302--1338, 2000.

\bibitem[Lee et~al.(2015)Lee, Sidford, and Wong]{lee2015faster}
Yin~Tat Lee, Aaron Sidford, and Sam Chiu-wai Wong.
\newblock A faster cutting plane method and its implications for combinatorial and convex optimization.
\newblock In \emph{2015 IEEE 56th Annual Symposium on Foundations of Computer Science}, pages 1049--1065. IEEE, 2015.

\bibitem[Levin(1965)]{levin1965algorithm}
Anatoly~Yur'evich Levin.
\newblock An algorithm for minimizing convex functions.
\newblock In \emph{Doklady Akademii Nauk}, volume 160, pages 1244--1247. Russian Academy of Sciences, 1965.

\bibitem[Liu and Nocedal(1989)]{liu_limited_1989}
Dong~C. Liu and Jorge Nocedal.
\newblock On the limited memory {BFGS} method for large scale optimization.
\newblock \emph{Mathematical Programming}, 45\penalty0 (1):\penalty0 503--528, August 1989.

\bibitem[Marsden et~al.(2022)Marsden, Sharan, Sidford, and Valiant]{marsden2022efficient}
Annie Marsden, Vatsal Sharan, Aaron Sidford, and Gregory Valiant.
\newblock Efficient convex optimization requires superlinear memory.
\newblock In \emph{Conference on Learning Theory}, pages 2390--2430. PMLR, 2022.

\bibitem[Mitliagkas et~al.(2013)Mitliagkas, Caramanis, and Jain]{Mitliagkas2013}
Ioannis Mitliagkas, Constantine Caramanis, and Prateek Jain.
\newblock Memory limited, streaming pca.
\newblock In \emph{Proceedings of the 26th International Conference on Neural Information Processing Systems - Volume 2}, NIPS'13, page 2886–2894, Red Hook, NY, USA, 2013. Curran Associates Inc.

\bibitem[Moshkovitz and Moshkovitz(2017)]{Moshkovitz2017}
Dana Moshkovitz and Michal Moshkovitz.
\newblock Mixing implies lower bounds for space bounded learning.
\newblock In \emph{Proceedings of the 2017 Conference on Learning Theory}, pages 1516--1566. PMLR, 2017.

\bibitem[Moshkovitz and Moshkovitz(2018)]{Moshkovitz2018}
Dana Moshkovitz and Michal Moshkovitz.
\newblock {Entropy Samplers and Strong Generic Lower Bounds For Space Bounded Learning}.
\newblock In \emph{9th Innovations in Theoretical Computer Science Conference (ITCS 2018)}, volume~94 of \emph{Leibniz International Proceedings in Informatics (LIPIcs)}, pages 28:1--28:20. Schloss Dagstuhl--Leibniz-Zentrum fuer Informatik, 2018.

\bibitem[Mota et~al.(2013)Mota, Xavier, Aguiar, and Püschel]{Mota2013}
João F.~C. Mota, João M.~F. Xavier, Pedro M.~Q. Aguiar, and Markus Püschel.
\newblock D-admm: A communication-efficient distributed algorithm for separable optimization.
\newblock \emph{IEEE Transactions on Signal Processing}, 61\penalty0 (10):\penalty0 2718--2723, 2013.
\newblock \doi{10.1109/TSP.2013.2254478}.

\bibitem[Nemirovski(1994)]{nemirovski1994parallel}
Arkadi Nemirovski.
\newblock On parallel complexity of nonsmooth convex optimization.
\newblock \emph{Journal of Complexity}, 10\penalty0 (4):\penalty0 451--463, 1994.

\bibitem[Nemirovskij and Yudin(1983)]{nemirovskij1983problem}
Arkadij~Semenovi{\v{c}} Nemirovskij and David~Borisovich Yudin.
\newblock Problem complexity and method efficiency in optimization.
\newblock 1983.

\bibitem[Nesterov(1989)]{nesterov1989self}
Ju~E Nesterov.
\newblock Self-concordant functions and polynomial-time methods in convex programming.
\newblock \emph{Report, Central Economic and Mathematic Institute, USSR Acad. Sci}, 1989.

\bibitem[Nesterov(2003)]{nesterov2003introductory}
Yurii Nesterov.
\newblock \emph{Introductory lectures on convex optimization: A basic course}, volume~87.
\newblock Springer Science \& Business Media, 2003.

\bibitem[Nocedal(1980)]{Nocedal1980}
Jorge Nocedal.
\newblock Updating quasi-newton matrices with limited storage.
\newblock \emph{Mathematics of Computation}, 35\penalty0 (151):\penalty0 773--782, 1980.

\bibitem[Peng and Rubinstein(2023)]{peng2023near}
Binghui Peng and Aviad Rubinstein.
\newblock Near optimal memory-regret tradeoff for online learning.
\newblock In \emph{2023 IEEE 64th Annual Symposium on Foundations of Computer Science (FOCS)}, pages 1171--1194. IEEE, 2023.

\bibitem[Peng and Zhang(2023)]{peng2023online}
Binghui Peng and Fred Zhang.
\newblock Online prediction in sub-linear space.
\newblock In \emph{Proceedings of the 2023 Annual ACM-SIAM Symposium on Discrete Algorithms (SODA)}, pages 1611--1634. SIAM, 2023.

\bibitem[Pilanci and Wainwright(2017)]{pilanci2017newton}
Mert Pilanci and Martin~J Wainwright.
\newblock Newton sketch: A near linear-time optimization algorithm with linear-quadratic convergence.
\newblock \emph{SIAM Journal on Optimization}, 27\penalty0 (1):\penalty0 205--245, 2017.

\bibitem[Raz(2017)]{Raz2017}
Ran Raz.
\newblock A time-space lower bound for a large class of learning problems.
\newblock In \emph{2017 IEEE 58th Annual Symposium on Foundations of Computer Science (FOCS)}, pages 732--742, 2017.
\newblock \doi{10.1109/FOCS.2017.73}.

\bibitem[Raz(2018)]{raz2018fast}
Ran Raz.
\newblock Fast learning requires good memory: A time-space lower bound for parity learning.
\newblock \emph{Journal of the ACM (JACM)}, 66\penalty0 (1):\penalty0 1--18, 2018.

\bibitem[Reddi et~al.(2016)Reddi, Konecn{\'y}, Richt{\'a}rik, P{\'o}czos, and Smola]{Reddi2016AIDEFA}
Sashank~J. Reddi, Jakub Konecn{\'y}, Peter Richt{\'a}rik, Barnab{\'a}s P{\'o}czos, and Alex Smola.
\newblock Aide: Fast and communication efficient distributed optimization.
\newblock \emph{ArXiv}, abs/1608.06879, 2016.

\bibitem[Roosta-Khorasani and Mahoney(2019)]{roosta2019sub}
Farbod Roosta-Khorasani and Michael~W Mahoney.
\newblock Sub-sampled newton methods.
\newblock \emph{Mathematical Programming}, 174:\penalty0 293--326, 2019.

\bibitem[Rudelson and Vershynin(2013)]{rudelson2013hanson}
Mark Rudelson and Roman Vershynin.
\newblock Hanson-wright inequality and sub-gaussian concentration.
\newblock 2013.

\bibitem[Rudelson and Zeitouni(2016)]{rudelson2016singular}
Mark Rudelson and Ofer Zeitouni.
\newblock Singular values of gaussian matrices and permanent estimators.
\newblock \emph{Random Structures \& Algorithms}, 48\penalty0 (1):\penalty0 183--212, 2016.

\bibitem[Shamir et~al.(2014)Shamir, Srebro, and Zhang]{shamir14}
Ohad Shamir, Nati Srebro, and Tong Zhang.
\newblock Communication-efficient distributed optimization using an approximate newton-type method.
\newblock In Eric~P. Xing and Tony Jebara, editors, \emph{Proceedings of the 31st International Conference on Machine Learning}, volume~32 of \emph{Proceedings of Machine Learning Research}, pages 1000--1008, Bejing, China, 22--24 Jun 2014. PMLR.
\newblock URL \url{https://proceedings.mlr.press/v32/shamir14.html}.

\bibitem[Sharan et~al.(2019)Sharan, Sidford, and Valiant]{Sharan2019}
Vatsal Sharan, Aaron Sidford, and Gregory Valiant.
\newblock Memory-sample tradeoffs for linear regression with small error.
\newblock In \emph{Proceedings of the 51st Annual ACM SIGACT Symposium on Theory of Computing}, STOC 2019, page 890–901. Association for Computing Machinery, 2019.

\bibitem[Smith et~al.(2017)Smith, Forte, Ma, Tak\'{a}\v{c}, Jordan, and Jaggi]{Smith2017}
Virginia Smith, Simone Forte, Chenxin Ma, Martin Tak\'{a}\v{c}, Michael~I. Jordan, and Martin Jaggi.
\newblock Cocoa: A general framework for communication-efficient distributed optimization.
\newblock \emph{J. Mach. Learn. Res.}, 18\penalty0 (1):\penalty0 8590–8638, jan 2017.
\newblock ISSN 1532-4435.

\bibitem[Srinivas et~al.(2022)Srinivas, Woodruff, Xu, and Zhou]{srinivas2022memory}
Vaidehi Srinivas, David~P Woodruff, Ziyu Xu, and Samson Zhou.
\newblock Memory bounds for the experts problem.
\newblock In \emph{Proceedings of the 54th Annual ACM SIGACT Symposium on Theory of Computing}, pages 1158--1171, 2022.

\bibitem[Steinhardt and Duchi(2015)]{Steinhardt15}
Jacob Steinhardt and John Duchi.
\newblock Minimax rates for memory-bounded sparse linear regression.
\newblock In \emph{Proceedings of The 28th Conference on Learning Theory}, pages 1564--1587. PMLR, 2015.

\bibitem[Steinhardt et~al.(2016)Steinhardt, Valiant, and Wager]{steinhardt16}
Jacob Steinhardt, Gregory Valiant, and Stefan Wager.
\newblock Memory, communication, and statistical queries.
\newblock In \emph{29th Annual Conference on Learning Theory}, pages 1490--1516. PMLR, 2016.

\bibitem[Tao(2023)]{tao2023topics}
Terence Tao.
\newblock \emph{Topics in random matrix theory}, volume 132.
\newblock American Mathematical Society, 2023.

\bibitem[Tarasov(1988)]{tarasov1988method}
Sergei~Pavlovich Tarasov.
\newblock The method of inscribed ellipsoids.
\newblock In \emph{Soviet Mathematics-Doklady}, volume~37, pages 226--230, 1988.

\bibitem[Vaidya(1996)]{vaidya1996new}
Pravin~M Vaidya.
\newblock A new algorithm for minimizing convex functions over convex sets.
\newblock \emph{Mathematical programming}, 73\penalty0 (3):\penalty0 291--341, 1996.

\bibitem[Vershynin(2020)]{vershynin2020high}
Roman Vershynin.
\newblock High-dimensional probability.
\newblock \emph{University of California, Irvine}, 10:\penalty0 11, 2020.

\bibitem[Viswanath and Trefethen(1998)]{viswanath1998condition}
Divakar Viswanath and LN~Trefethen.
\newblock Condition numbers of random triangular matrices.
\newblock \emph{SIAM Journal on Matrix Analysis and Applications}, 19\penalty0 (2):\penalty0 564--581, 1998.

\bibitem[Wang et~al.(2017)Wang, Wang, and Srebro]{wang17aMemo}
Jialei Wang, Weiran Wang, and Nathan Srebro.
\newblock Memory and communication efficient distributed stochastic optimization with minibatch prox.
\newblock In Satyen Kale and Ohad Shamir, editors, \emph{Proceedings of the 2017 Conference on Learning Theory}, volume~65 of \emph{Proceedings of Machine Learning Research}, pages 1882--1919. PMLR, 07--10 Jul 2017.
\newblock URL \url{https://proceedings.mlr.press/v65/wang17a.html}.

\bibitem[Wangni et~al.(2018)Wangni, Wang, Liu, and Zhang]{Wangni2018}
Jianqiao Wangni, Jialei Wang, Ji~Liu, and Tong Zhang.
\newblock Gradient sparsification for communication-efficient distributed optimization.
\newblock In \emph{Proceedings of the 32nd International Conference on Neural Information Processing Systems}, NIPS'18, page 1306–1316, Red Hook, NY, USA, 2018. Curran Associates Inc.

\bibitem[Woodworth and Srebro(2019)]{woodworth2019open}
Blake Woodworth and Nathan Srebro.
\newblock Open problem: The oracle complexity of convex optimization with limited memory.
\newblock In \emph{Conference on Learning Theory}, pages 3202--3210. PMLR, 2019.

\bibitem[Yudin and Nemirovski(1976)]{yudin76evaluation}
David Yudin and Arkadii Nemirovski.
\newblock Evaluation of the information complexity of mathematical programming problems.
\newblock \emph{Ekonomika i Matematicheskie Metody}, 12:\penalty0 128--142, 1976.

\bibitem[Zhang et~al.(2012)Zhang, Duchi, and Wainwright]{Zhang2012Communication}
Yuchen Zhang, John~C. Duchi, and Martin~J. Wainwright.
\newblock Communication-efficient algorithms for statistical optimization.
\newblock In \emph{2012 IEEE 51st IEEE Conference on Decision and Control (CDC)}, pages 6792--6792, 2012.
\newblock \doi{10.1109/CDC.2012.6426691}.

\end{thebibliography}

\newpage

\appendix

\section{Concentration inequalities}

We first state a standard result on the concentration of normal Gaussian random variables. 

\begin{lemma}[Lemma 1 \cite{laurent2000adaptive}]
\label{lemma:concentration_gaussian}
    Let $n\geq 1$ and define $\mb y\sim\Ncal(0,Id_n)$. Then for any $t\geq 0$,
    \begin{align*}
        \Pbb(\|\mb y\|^2 \geq n+2\sqrt{nt} + 2t)  &\leq e^{-t},\\
        \Pbb(\|\mb y\|^2 \leq n-2\sqrt{nt}) &\leq e^{-t}.
    \end{align*}
    In particular, plugging $t=n/2$ and $t=(3/8)^2 n\geq n/8$ gives
    \begin{align*}
        \Pbb(\|\mb y\| \geq 2\sqrt n) &\leq e^{-n/2},\\
        \Pbb\paren{\|\mb y\| \leq \frac{\sqrt n}{2}} &\leq e^{-n/8}. 
    \end{align*}
\end{lemma}

For our work, we need the following concentration for quadratic forms.

\begin{theorem}[\cite{hanson1971bound,rudelson2013hanson}]
\label{thm:hanson-wright}
    Let $\mb x = (X_1,\ldots,x_d)\in\Rbb^d$ be a random vector with i.i.d. components $X_i$ which satisfy $\Ebb[X_i] = 0$ and $\|X_i\|_{\psi_2}\leq K$ and let $\mb M\in \Rbb^{d\times d}$. Then, for some universal constant $c_{hw}>0$ and every $t\geq 0$,
    \begin{multline*}
        \max\set{\Pbb\paren{\mb x^\top \mb M \mb x - \Ebb[\mb x^\top \mb M\mb x] >t} , \Pbb\paren{\Ebb[\mb x^\top \mb M\mb x] - \mb x^\top \mb M \mb x >t} }\\
        \leq \exp\paren{-c_{hw} \min\set{\frac{t^2}{K^4 \|\mb M\|_F^2},\frac{t}{K^2 \|\mb M\|}}}.
    \end{multline*}
\end{theorem}

We will only need a restricted version of this concentration bound, for which we can explicit the constant $c_{hw}$.

\begin{lemma}\label{lemma:concentration_projection}
    Let $P$ be a projection matrix in $\Rbb^d$ of rank $r$ and let $\mb x\in\Rbb^d$ be a random vector sampled uniformly on the unit sphere $\mb x\sim\Ucal(S_{d-1})$. Then, for any $t>0$,
    \begin{align*}
        \Pbb\paren{\|P(\mb x)\|^2 \geq  \frac{r}{d}(1+t)} &\leq  e^{-\frac{r}{8}\min(t,t^2)},\\
        \Pbb\paren{\|P(\mb x)\|^2 \leq \frac{r}{d}(1-t) } & \leq e^{-\frac{r}{4}t^2}.
    \end{align*}
    Also, for $t\geq 1$, we have
    \begin{equation*}
        \Pbb\paren{\|P(\mb x)\|^2 \leq \frac{r}{dt}} \leq e^{-\frac{r}{2}\ln(t) +\frac{d}{2e}}.
    \end{equation*}
\end{lemma}

\begin{proof}
    We start with the first inequality to prove. Using the exact same arguments as in \cite[Proposition 20]{blanchard2023quadratic} show that for $t\geq 0$,
    \begin{equation*}
        \Pbb\paren{\|P(\mb x)\|^2 - \frac{r}{d} \geq t} \leq \exp\paren{-\frac{d}{2}D\paren{\frac{r}{d}\parallel\frac{r}{d}+t }}.
    \end{equation*}
    Applying this bound to the projection $I_d-P$ implies the other inequality
    \begin{equation*}
        \Pbb\paren{\|P(\mb x)\|^2 - \frac{r}{d} \leq - t} \leq \exp\paren{-\frac{d}{2} D\paren{\frac{d-r}{d} \parallel \frac{d-r}{d} +t }} =  \exp\paren{-\frac{d}{2} D\paren{\frac{r}{d}\parallel \frac{r}{d} -t }}.
    \end{equation*}
    It only remains to bound the KL divergence. Consider the function $f(x) = x\in [-\frac{r}{d},1-\frac{r}{d}]\mapsto D(\frac{r}{d} \parallel \frac{r}{d} +x)$. Then, $f'(0)=0$ and
    \begin{equation*}
        f''(x) = \frac{r/d}{(r/d+x)^2} + \frac{1-r/d}{(1-r/d-x)^2} \geq \frac{r/d}{(r/d+x)^2}
    \end{equation*}
    In particular, if $x\leq 0$, $f''(x)\geq \frac{d}{r}$ so that Taylor's expansion theorem directly gives $f(x) \geq \frac{dx^2}{2r}$. Next, for $|x|\leq\frac{r}{d}$, we have $f''(x) \geq \frac{d}{2r}$. Hence $f(x) \geq \frac{dx^2}{4r}$. Otherwise, if $x\geq \frac{r}{d}$, we have
    \begin{equation*}
        D\paren{\frac{r}{d} \parallel \frac{r}{d} +x} \geq \int_0^x \frac{r/d(x-u)}{2(r/d+u)^2}du = \frac{x}{2} -\frac{r}{2d}\ln \frac{r/d+x}{r/d} \geq \frac{2-\ln 2}{4}x \geq \frac{x}{4}.
    \end{equation*}
    In the last inequality, we used $\ln(1+t)\leq \frac{\ln 2}{2}t$ for $t\geq 2$. In summary, we obtained for $t\in[0,1]$,
    \begin{equation*}
        \Pbb\paren{\|P(\mb x)\|^2 - \frac{r}{d} \leq - \frac{r}{d} t} \leq e^{-\frac{rt^2}{4}}.
    \end{equation*}
    And for any $t\geq 0$,
    \begin{equation*}
        \Pbb\paren{\|P(\mb x)\|^2 - \frac{r}{d} \geq  \frac{r}{d} t} \leq e^{-\frac{r}{8}\min(t,t^2)}.
    \end{equation*}

    We last prove the third claim of the lemma using the same equation:
    \begin{align*}
        \Pbb\paren{\|P(\mb x)\|^2 \leq \frac{r}{dt}} \leq \exp\paren{-\frac{d}{2} D\paren{\frac{r}{d}\parallel \frac{r}{dt} }} &\leq \exp\paren{-\frac{d}{2} \paren{\frac{r}{d} \ln(t) +(1-p)\ln(1-p) }}\\
        &\leq e^{-\frac{r}{2}\ln(t) +\frac{d}{2e}}.
    \end{align*}
    This ends the proof.
\end{proof}

The previous result gives concentration bounds for the projection of a single random vector onto a fixed subspace. We now use this result to have concentration bounds on the projection of points from a random subspace sampled uniformly, onto a fixed subspace. Similar bounds are certainly known, in fact, the following lemma can be viewed as a variant of the Johnson–Lindenstrauss lemma. We include the proof for the sake of completeness.

\begin{lemma}\label{lemma:concentration_projecting_subspaces}
    Let $P$ be a projection matrix in $\Rbb^d$ of rank $r$ and let $E$ be a random $s$-dimensional subspace of $\Rbb^d$ sampled uniformly. Then, for any $t\in(0,1]$,
    \begin{equation*}
        \Pbb \paren{ (1-t)\|\mb x\| \leq  \sqrt{\frac{d}{r} } \|P(\mb x)\| \leq (1+t)\|\mb x\|,\; \forall \mb x\in E } \geq 1-\exp\paren{ s\ln\frac{Cd}{rt} - \frac{rt^2}{32}},
    \end{equation*}
    for some universal constant $c>0$.
\end{lemma}

\begin{proof}
    We use an $\epsilon$-net argument. Let $\epsilon=\frac{t}{2}\sqrt{\frac{r}{d}}$ and construct an $\epsilon$-net $\Sigma$ of the unit sphere of $E$, ($E\cap S_{d-1}$) such that each element $\mb x\in \Sigma$ is still distributed as a uniform random unit vector. For instance, consider any parametrization, then rotate the whole space $Q$ again by some uniform rotation. Note that $|\Sigma| \leq (c/\epsilon)^s$ for some universal constant $c>0$ (e.g. see \cite[Lemma 2.3.4]{tao2023topics}). Combining \cref{lemma:concentration_projection}, the union bound, and the observation that $\sqrt{1-t} \geq 1-t$ and $\sqrt{1+t}\leq 1+t$, we obtain
    \begin{equation*}
        \Pbb \paren{ 1-\frac{t}{2} \leq  \sqrt{\frac{d}{r} } \|P(\mb x)\| \leq 1+\frac{t}{2},\; \forall \mb x\in \Sigma } \geq 1-|\Sigma|e^{-rt^2/32} \geq 1-\exp\paren{s\ln\frac{cd}{rt} - \frac{rt^2}{32}}.
    \end{equation*}
    Denote by $\Ecal$ this event. For unit vector $\mb y\in E\cap S_{d-1}$ there exists $\mb x\in\Sigma$ with $\|\mb x-\mb y\|\leq \epsilon$. As a result, we also have $\|P(\mb y) - P(\mb x)\| \leq \epsilon$. Under $\Ecal$, the triangular inequality then shows that
    \begin{equation*}
        1-t \leq  \sqrt{\frac{d}{r} } \|P(\mb x)\| \leq 1+ t ,\quad \mb x\in E\cap S_{d-1}.
    \end{equation*}
    For an arbitrary vector $\mb x\in E\setminus\{0\}$, we can then apply the above inequality to $\mb x/\|\mb x\|$, which gives the desired result under $\Ecal$. This ends the proof.
\end{proof}

Last, we need the following result which lower bounds the smallest singular value for rectangular random matrices.

\begin{theorem}[Theorem 2.13 \cite{davidson2001local}]
\label{thm:random_rectangular_matrices}
    Given $m,n\in\Nbb$ with $m\leq n$. Let $\beta = m/n$ and let $\mb M\in\Rbb^{n\times m}$ be a matrix with independent Gaussian $\Ncal(0,1/n)$ coordinates. The singular values $s_1(\mb M)\leq \ldots \leq s_m(\mb M)$ satisfy
    \begin{equation*}
        \max\set{ \Pbb(s_1(\mb M) \leq 1-\sqrt \beta -t), \Pbb(s_m(\mb M)\geq 1+\sqrt \beta +t)} \leq e^{-nt^2/2}.
    \end{equation*}
\end{theorem}

\section{Decomposition of robustly-independent vectors}

Here we give the proof of \cref{lemma:gram-schmidt_marsden}, which is essentially the same as in \cite[Lemma 34]{marsden2022efficient} or \cite[Lemma 22]{blanchard2023quadratic}.

\begin{proof}[of \cref{lemma:gram-schmidt_marsden}]
        Let $\mb B=(\mb b_1,\ldots, \mb b_r)$ be the orthonormal basis given by the Gram-Schmidt decomposition of $\mb y_1,\ldots, \mb y_r$. By definition of the Gram-Schmidt decomposition, we can write $\mb Y = \mb B\mb C$ where $\mb C$ is an upper-triangular matrix. Further, its diagonal is exactly $diag(\|P_{Span(\mb y_{l'},l'<l)^\perp}(\mb y_l)\|, l\leq r)$. Hence,
    \begin{equation*}
        \det(\mb Y) = \det(\mb C) = \prod_{l\leq r} \|P_{Span(\mb y_{l'},l'<l)^\perp}(\mb y_l)\| \geq \delta^r.
    \end{equation*}
    We now introduce the singular value decomposition $\mb Y = \mb U 
 diag(\sigma_1,\ldots,\sigma_r)\mb V^\top$, where $\mb U\in \Rbb^{d\times r}$ and $\mb V\in \Rbb^{r\times r}$ have orthonormal columns, and $\sigma_1\geq\ldots\geq \sigma_r$. Next, for any vector $\mb z\in\Rbb^r$, since the columns of $\mb Y$ have unit norm,
    \begin{equation*}
        \|\mb Y\mb z\|_2 \leq \sum_{l\leq r}|z_l| \|\mb y_l\|_2 \leq \|\mb z\|_1 \leq \sqrt r \|\mb z\|_2. 
    \end{equation*}
    In the last inequality we used Cauchy-Schwartz. Therefore, all singular values of $\mb Y$ are upper bounded by $\sigma_1\leq \sqrt r$. Thus, with $r' = \lceil r/s\rceil$
    \begin{equation*}
        \delta^r \leq \det(\mb Y) =\prod_{l=1}^r \sigma_l \leq r^{(r'-1)/2} \sigma_{r'}^{r-r'+1} \leq r^{r/2s}\sigma_{r'}^{(s-1)r/s},
    \end{equation*}
    so that $\sigma_{r'}\geq \delta^{s/(s-1)}/r^{1/(2s)}$. We are ready to define the new vectors. We pose for all $i\leq r'$, $\mb z_i = \mb u_i$ the $i$-th column of $\mb U$. These correspond to the $r'$ largest singular values of $\mb Y$ and are orthonormal by construction. Then, for any $i\leq r'$, we also have $\mb z_i = \mb u_i = \frac{1}{\sigma_i}\mb Y \mb v_i$ where $\mb v_i\in\Rbb^r$ is the $i$-th column of $\mb V$. Hence, for any $\mb a\in \Rbb^d$,
    \begin{equation*}
        |\mb z_i^\top \mb a| =\frac{1}{\sigma_i}|\mb v_i^\top \mb Y^\top \mb a| \leq \frac{\|\mb v_i\|_1}{\sigma_i} \|\mb Y^\top \mb a\|_\infty \leq \frac{r^{1/2+1/(2s)}}{\delta^{s/(s-1)}}\|\mb Y^\top \mb a\|_\infty.
    \end{equation*}
    This ends the proof of the lemma.
\end{proof}

\end{document}